\documentclass{article}
\usepackage{
amsmath,
amsfonts,
latexsym, 
amsrefs,
amssymb
}

\usepackage{titlesec}
\usepackage{titletoc}\usepackage[title,titletoc]{appendix}

\usepackage{bm}

\usepackage{tocloft}

\usepackage{cite}

\usepackage{graphicx}
\newcommand{\Lnorm}[2]{\left\| #2 \right\|_{{#1}}}

\def\cA{{\mathcal A}}
\def\cB{{\mathcal B}}

\def\cW{{\mathcal W}}

\DeclareMathOperator{\Prob}{\mathbb{P}}   

\newcommand{\1}{\mathds{1}}
\usepackage{enumerate}
\usepackage{mathrsfs}
\usepackage{wrapfig}

\usepackage{color}

\numberwithin{equation}{section}

\newcommand{\eps}{\varepsilon}
\newcommand{\pt}{\partial}
\newcommand{\rd}{{\rm d}}

\newcommand{\bR}{{\mathbb R}}

\newcommand{\bZ}{{\mathbb Z}}

\newcommand{\bx}{{\bf{x}}}
\newcommand{\by}{{\bf{y}}}
\newcommand{\bu}{{\bf{u}}}
\newcommand{\bv}{{\bf{v}}}
\newcommand{\bw}{{\bf{w}}}

\newcommand{\bz} {{\bf {z}}}

\newcommand{\al}{\alpha}
\newcommand{\de}{\delta}

\newcommand{\be}{\begin{equation}}
\newcommand{\ee}{\end{equation}}

\newcommand{\bbb}{{\bf{b}}}

\newcommand{\e}{{\varepsilon}}

\newcommand{\la}{\lambda}

\newcommand{\om}{{\omega}}
\newcommand{\si}{\sigma}

\newcommand{\cG}{{\cal G}}
\newcommand{\cF}{{\cal F}}

\def\cA{{\mathcal A}}

\def\cW{{\mathcal W}}

\newcommand{\fa}{{\frak a}}

\newcommand{\cS}{{\mathcal S}}

\newcommand{\cR}{{\mathcal R}}
\newcommand{\cU}{{\mathcal U}}

\usepackage{amsmath} 
\usepackage{amssymb}
\usepackage{amsthm}

\def\RR{{\mathbb R}}

\renewcommand{\b}[1]{\bm{\mathrm{#1}}} 

\renewcommand{\cal}{\mathcal}

\newcommand{\wh}{\widehat}
\newcommand{\wt}{\widetilde}

\newcommand{\ii}{\mathrm{i}} 
\newcommand{\dd}{\mathrm{d}}

\renewcommand{\epsilon}{\varepsilon}
\renewcommand{\leq}{\leqslant}
\renewcommand{\geq}{\geqslant}

\renewcommand{\le}{\leq}
\renewcommand{\ge}{\geq}

\renewcommand{\P}{\mathbb{P}}
\newcommand{\E}{\mathbb{E}}
\newcommand{\R}{\mathbb{R}}

\newcommand{\N}{\mathbb{N}}
\newcommand{\Z}{\mathbb{Z}}

\newcommand{\NN}{\mathbb{N}}

\DeclareMathOperator{\var}{Var}
\DeclareMathOperator{\supp}{supp}

\DeclareMathOperator{\OO}{O}
\DeclareMathOperator{\oo}{o}

\theoremstyle{plain} 
\newtheorem{theorem}{Theorem}[section]
\newtheorem*{theorem*}{Theorem}
\newtheorem{lemma}[theorem]{Lemma}
\newtheorem*{lemma*}{Lemma}
\newtheorem{corollary}[theorem]{Corollary}
\newtheorem*{corollary*}{Corollary}
\newtheorem{proposition}[theorem]{Proposition}
\newtheorem*{proposition*}{Proposition}

\newtheorem{definition}[theorem]{Definition}
\newtheorem*{definition*}{Definition}

\newtheorem*{example*}{Example}
\newtheorem{remark}[theorem]{Remark}

\newtheorem*{remark*}{Remark}
\newtheorem*{remarks*}{Remarks}

\makeatletter
\renewcommand{\subsection}{\@startsection
{subsection}
{2}
{0mm}
{-\baselineskip}
{0 \baselineskip}
{\normalfont\itshape}} 
\makeatother

\usepackage{dsfont}
\usepackage{stmaryrd}

\newcommand{\nc}{\normalcolor}

\def\bR{{\mathbb R}}
\def\bZ{{\mathbb Z}}

\newcommand{\Tr}{\mbox{Tr\,}}

\renewcommand{\b}[1]{\boldsymbol{\mathrm{#1}}} 

\def\@empty{}

\def\author#1{\par
    {\centering{\authorfont#1}\par\vspace*{0.05in}}
}

\def\titlefont{\fontsize{13}{15}\bfseries\boldmath\selectfont\centering{}}
\def\authorfont{\fontsize{13}{15}}
\def\abstractfont{\fontsize{8}{10}}

\let\affiliationfont\rhfont

\def\address#1{\par
    {\centering{\affiliationfont#1\par}}\par\vspace*{11pt}
}

\def\keywords#1{\par
    \vspace*{8pt}
    {\authorfont{\leftskip18pt\rightskip\leftskip
    \noindent{\it\small{Keywords}}\/:\ #1\par}}\vskip-12pt}

\def\body{
\setcounter{footnote}{0}
\def\thefootnote{\alph{footnote}}
\def\@makefnmark{{$^{\rm \@thefnmark}$}}
}

\def\title#1{
    \thispagestyle{plain}
    \vspace*{-14pt}
    \vskip 79pt
    {\centering{\titlefont #1\par}}%
    \vskip 1em
}

\renewenvironment{abstract}{\par%
    \vspace*{6pt}\noindent 
    \abstractfont
    \noindent\leftskip18pt\rightskip18pt
}{%
  \par}

\renewcommand{\b}[1]{\boldsymbol{\mathrm{#1}}} 
 
\newcommand{\for}{\qquad \text{for} \quad}

\usepackage{tocloft}

\makeatletter
\renewcommand{\section}{\@startsection
{section}
{1}
{0mm}
{-2\baselineskip}
{1\baselineskip}
{\normalfont\large\scshape\centering}} 
\makeatother

\newcommand{\fq}{{\frak q}}

\begin{document}


\title{Fixed energy universality for generalized Wigner matrices}

\vspace{1.2cm}
\noindent\begin{minipage}[b]{0.5\textwidth}

 \author{P. Bourgade}

\address{Cambridge University\\
bourgade@math.ias.edu}
 \end{minipage}
\begin{minipage}[b]{0.5\textwidth}

 \author{L. Erd{\H o}s}

\address{Institute of Science and Technology Austria\\
lerdos@ist.ac.at}

 \end{minipage}

\noindent\begin{minipage}[b]{0.5\textwidth}

 \author{H.-T. Yau}

\address{Harvard University\\
htyau@math.harvard.edu}
 \end{minipage}
\begin{minipage}[b]{0.5\textwidth}

 \author{J. Yin}

\address{University of Wisconsin, Madison\\
jyin@math.wisc.edu}

 \end{minipage}

~\vspace{0.3cm}

\begin{abstract}
We prove the Wigner-Dyson-Mehta conjecture at fixed energy  in the bulk of the spectrum
 for generalized symmetric and Hermitian Wigner matrices.
Previous results concerning  the universality of random matrices  either require an
averaging in the energy parameter
or they hold only  for Hermitian matrices if the energy parameter is fixed.
We develop a homogenization theory of the Dyson Brownian motion and show that microscopic universality
follows from mesoscopic statistics. 
\end{abstract}

\vspace{1cm}

\tableofcontents

\let\thefootnote\relax\footnote{\noindent 
The work of P. B. is partially supported by the NSF grant DMS-1208859. 
The work of L. E. is partially supported  by ERC Advanced Grant, RANMAT 338804.
The work of H.-T. Y. is partially supported by the NSF grant DMS-1307444 and the Simons investigator fellowship.
The work of J. Y. is partially supported by NSF Grant DMS-1207961. The major part of
this research was conducted when all authors were visiting IAS and were also supported by
the NSF Grant DMS-1128255.}
\newpage

\section{Introduction}

E. Wigner discovered that energy levels of large quantum systems exhibit  remarkably
simple universality patterns. He introduced a fundamental model, the Wigner matrix ensemble,
and postulated that the statistics of the eigenvalue gaps, i.e. differences of 
consecutive eigenvalues, depend only on the symmetry
class of the model  and are independent of the details of the ensemble.    
Although  the central  universal objects in Wigner's original work   
were the eigenvalue gap distributions, 
 the subsequent developments showed that the correlation functions play a key role. 
In fact,  a few years after Wigner's pioneering work, Gaudin,  Mehta and Dyson 
computed explicitly the eigenvalue  correlation functions for the Gaussian cases
and expressed the eigenvalue gap distributions in terms of them. 
Later on, Mehta formalized a version of the (Wigner-Dyson-Mehta) universality conjecture   
 in his seminal book  
\cite{Meh1991} by   stating    that the appropriately rescaled correlation
functions for any Wigner ensemble coincide with those for the Gaussian  cases 
as $N$,  the size of the matrix,  tends to infinity. 
This holds for both real symmetric and complex Hermitian ensembles
 (Conjectures 1.2.1 and 1.2.2 in \cite{Meh1991}). 
The topology  of the convergence, however,  was not specified explicitly.

One possible topology for the correlation functions is the pointwise convergence. But  the convergence in this topology   
cannot  hold  for Wigner ensembles with  discrete (e.g. Bernoulli) matrix elements  
so  it could   only be used  for a  certain  subclass of Wigner matrices. 
Thus a reasonably strong  topology suitable  for the universality of the whole class of Wigner matrices is
the vague convergence of the local correlation functions, rescaled around
a fixed energy $E$; in short we will call it {\it fixed energy universality}
(see Section~\ref{sec:main} for the precise definitions). Certainly, instead of fixing 
the energy $E$, we can also take weak convergence in $E$ or equivalently, taking some average in 
the energy. We will call  universality in  this  weaker  topology {\it averaged  energy universality}.
 Finally, one can go back to  Wigner's original point of 
view and ask for universality of the gap distributions.

The  Wigner-Dyson-Mehta 
 conjecture has been widely open until the recent work \cite{ErdPecRamSchYau2010}
where a general scheme to approach it  was outlined and carried out 
for complex Hermitian matrices. The basic idea is  first to establish   a  local version
of the  semicircle law 
and use it as an input to control the correlation function asymptotics in the Brezin-Hikami formula
(which is related to Harish-Chandra/Itzykson-Zuber  formula). This provides 
universality for  the so-called Gaussian divisible models with  a  very  small  Gaussian component, or ``noise'' 
(previously, this universality was established by Johansson \cite{Joh2001} when the noise is of order one).  
The last step is  an approximation of  a general  Wigner  ensemble by Gaussian divisible ones
and this  leads  to the fixed energy universality for Hermitian Wigner matrices   whose matrix elements
have smooth distributions. 
The various restrictions on  the laws of matrix elements  were greatly  relaxed in subsequent works
\cite{TaoVu2011, ERSTVY, ErdKnoYauYin2012, ErdYau2011}.
In particular, using the local semicircle law \cite{ErdSchYau2009} as a main input, Tao-Vu 
\cite{TaoVu2011}
proved   a comparison theorem which  provides an approximation result  
for Wigner matrices satisfying  a four moment matching condition. 
Finally, the 
conditions for tail distributions 
of the matrix elements were greatly relaxed in \cite{ErdKnoYauYin2012, ErdYau2011,TaoVu2011EJP}. 
For a concise review on the recent progress on the universality for random matrices, see, e.g., \cite{ErdYau2012}.

For real symmetric matrices,  no algebraic formula  in the spirit of
Brezin-Hikami is known.  A completely new method based on  relaxation 
of the Dyson Brownian Motion (DBM)  to local  equilibrium was  developed
in a series of papers \cite{ErdSchYau2011,ErdSchYauYin2012,ErdYauYin2012Univ}.
This approach is very robust and applies to all  symmetry classes of  random matrices,
including also  sample covariance matrices and sparse matrices 
\cite{ErdSchYauYin2012,  ErdKnoYauYin2012},  
but it yields only the {\it average energy universality}. 
Although the energy  averaging  is on a very small scale,  it so far  cannot be 
completely  removed  with this method.

We now  comment on a parallel development for the universality of the eigenvalue gaps which was
Wigner's  original interest.
Correlation functions at a fixed energy $E$ carry full information about
the distribution of the eigenvalues near $E$.    In particular, 
the Fredholm determinant  and the Jimbo-Miwa-Mori-Sato formulae yield the probability that no eigenvalues  appear 
in a neighborhood around $E$. 
The universality of the distribution of  the gap with a {\it fixed label}  
(which we will call gap universality),  e.g. the difference between,  say, 
the $N/2$-th and $(N/2-1)$-th eigenvalues, however cannot be deduced rigorously from 
 the fixed energy universality. Conversely, the gap universality  does   not imply the 
fixed energy universality  either. 
 The reason is that 
eigenvalues with a fixed label fluctuate on a scale
larger than the mean eigenvalue spacing,  so fixed  energy and fixed label
universalities are not equivalent. 
The gap universality
 was established in \cite{ErdYau2012singlegap} via a De Giorgi-Nash-Moser type H\"older
regularity result for a discrete parabolic equation with time dependent random coefficients  $B_{ij}(t)=(x_i(t)-x_j(t))^{-2}$
where $\bx(t)$ is  the DBM trajectory. The  gap universality for the 
 special case of   Hermitian matrices satisfying the four moment matching condition  was proved earlier in \cite{Tao2013}.

To summarize, the Wigner-Dyson-Mehta  conjecture was completely resolved  in the sense 
of averaged energy and  fixed label  gap universalities for both symmetric and Hermitian ensembles. 
In the sense of fixed energy universality, it  was proved for the Hermitian matrices, but not for real symmetric ones.  
In the current paper,  we settle this last remaining case of the 
Wigner-Dyson-Mehta  conjecture by proving  the  universality of local correlation functions at any fixed energy $E$
in the bulk spectrum for generalized Wigner matrices of any symmetry classes. 
Our theorem in particular implies the following three new results for real symmetric matrices
(including the Bernoulli cases):
(1) existence of the density of states on microscopic scales for generalized Wigner matrices,
(2)    the  extension of the  Jimbo-Miwa-Mori-Sato formula of the gap probability to
 generalized real Wigner matrices, (3)  the precise distribution of the condition number
 or  the smallest (in absolute value) eigenvalue of generalized  Wigner matrices. 
 Our proof also applies to the third symmetry class, the symplectic matrices, but we
will focus on the real symmetric case as this is  the most complicated case  
from  the technical point of view.

The essence of the current work is a homogenization theory for the  discrete parabolic
 equation with time dependent random coefficients  $B_{ij}(t)=(x_i(t)-x_j(t))^{-2}$. 
By a rigidity property of the DBM trajectories, the 
random coefficients are close to  deterministic ones, $B_{ij}(t) \approx (\gamma_i-\gamma_j)^{-2}$
if $|i-j|\gg 1$ (the typical locations $\gamma_i$ are defined in (\ref{eqn:typ})).
The continuous version of the corresponding heat kernel is explicitly known; 
in fact locally it is
given by $e^{-t|p|}(i,j)$ where $|p|=\sqrt{-\Delta}$.
By coupling two DBM for two different initial conditions $\bx(0)$ and $\by(0)$
(one for Wigner, one for a reference Gaussian ensemble),
we show that after a sufficiently long time, the difference between $x_i(t)$ and $y_i(t)$
is given by the deterministic heat kernel acting on the difference of the initial data.
Due to the scaling properties of the explicit heat kernel, this latter involves only 
mesoscopic linear statistics of the initial conditions which are more accessible than
microscopic ones. Homogenization thus enables us to transfer 
mesoscopic statistics to microscopic ones. The main steps of the proof 
will be described in the next section in more details.

 {\it Convention.} For two $N$-dependent positive quantities $a=a_N$, $b=b_N$ we say that $a$ and $b$ 
are comparable, $a\sim b$, if there  exists a constant $C>0$, independent of $N$, such that $C^{-1}\le a/b\le C$.

\section{Main result and sketch of the proof}\label{sec:main}

\subsection{The model and the result.}\ We consider the following class of random matrices. 

\begin{definition}\label{def:wig}
A generalized Wigner matrix $H_N$ is a  Hermitian or symmetric $N\times N$ matrix whose upper-triangular elements $h_{ij}=\overline{h_{ji}}$, $i\leq j$, are independent random variables with mean zero and variances $\sigma_{ij}^2=\E(|h_{ij}|^2)$  that satisfy  the following two conditions:
\begin{enumerate}[(i)]
\item Normalization: for any $j\in\llbracket 1,N\rrbracket$, $\sum_{i=1}^N\sigma_{ij}^2=1$.
\item Non-degeneracy:  $ \sigma_{ij}^2\sim N^{-1}$ for all $i,j\in\llbracket 1,N\rrbracket$. 
\end{enumerate}
In the Hermitian case, we furthermore
assume that $\var\Re(h_{ij})\sim \var\Im(h_{ij})$ for $i\neq j$ and that one of the following holds: (1)  $\Re(h_{ij}), \Im(h_{ij})$ are independent, 
or (2)  the law of $h_{ij}$ is isotropic,  i.e. $|h_{ij}|$ is independent of $\arg h_{ij}$, which is uniform on $(0,2\pi)$.
\end{definition}

\noindent We additionally assume that there exists $p>0$ large but fixed such that
\begin{equation}\label{eqn:forHansonWright}
\sup_{i,j,N}\E\left((\sqrt{N}|h_{ij}|)^p\right)<\infty.
\end{equation}
For example $p=10$ is sufficient for our purpose, in this work we will not try to get the lowest possible exponent $p$, for the clarity of exposition.
We denote by 
$$
x_1\leq\dots\leq x_N
$$
the $N$ random eigenvalues of a generalized Wigner matrix $H_N$. Let $\mu^{(N)}(\bu)$ be the associated probability distribution of the spectrum, where $\bu=(u_1,\dots,u_N)$ is an element of the simplex
$\Sigma = \{ \bu \; : \; u_1\le \ldots \le u_N\}\subset \R^N$.
The universal limiting point process for random spectra will be uniquely characterized by the limits of the $k$-point correlation functions
for $k=1,2,\ldots $ as $N\to\infty$. These are defined by
\be\label{def:corr}
\rho_k^{(N)}(u_1,\dots,u_k)=\int_{\RR^{N-k}}\widetilde\mu^{(N)}(\bu)\rd u_{k+1}\dots\rd u_N,
\ee
where $\widetilde\mu^{(N)}$ is the symmetrized version of $\mu^{(N)}$, defined on $\RR^N$ 
instead of the simplex: $\widetilde\mu^{(N)}(\bu)=\frac{1}{N!}\mu^{(N)}(\bu^{(\sigma)})$ where $\bu^{(\sigma)}=(u_{\sigma(1)},\dots,u_{\sigma(N)})$ with $u_{\sigma(1)}\leq \dots\leq u_{\sigma(N)}$.
The limiting density ($k=1$ point correlation function) 
of the  eigenvalues 
is the Wigner semicircle law and it will be denoted
$$\rd \varrho(x)=\varrho(x)\rd x=\frac{1}{2\pi}\sqrt{(4-x^2)_+}\rd x.$$
 In the fundamental particular case where $H_N$ is a  real symmetric  matrix from the  Gaussian Orthogonal Ensemble (GOE), 
  the correlation functions are known to converge on microscopic scales, the limit being expressible as a determinant  \cite{MehGau1960,Dys1962,Meh1991,Dys1970}: for any $\bv\in\RR^k$ and $E\in(-2,2)$, we have
\begin{equation}\label{GOEconv}
\frac{1}{\varrho(E)^k}\rho_{k}^{(N,\text{\tiny GOE})}\left(E+\frac{\bv}{N\varrho(E)}\right)=\frac{1}{\varrho(E)^k}\rho_{k}^{(N)}\left(E+\frac{v_1}{N\varrho(E)},\dots,E+\frac{v_k}{N\varrho(E)}\right)
\to \rho_{k}^{(\text{\tiny GOE})}\left(\bv\right),
\end{equation}
where this limit is independent of $E\in(-2,2)$. For  complex Hermitian  matrices from the 
 Gaussian Unitary Ensemble (GUE),  the same statement holds with a different limit
$\rho_{k}^{(\text{\tiny GUE})}\left(\bv\right)$.

Bulk universality for generalized Wigner matrices was considered for various convergence types, notably the two following ones.
We state them only in the symmetric case, the Hermitian setting being similar.\\

\noindent{\it Fixed energy universality  (in the bulk).  } For any   $k\geq 1$,  $F:\RR^k\to\RR$  continuous and compactly supported and
 for any $\kappa>0$, we have,  uniformly in $E\in[-2+\kappa,2-\kappa]$, 
$$
\lim_{N\to\infty}\frac{1}{\varrho(E)^k}\int \rd\bv F(\bv)\rho_{k}^{(N)}\left(E+\frac{\bv}{N\varrho(E)}\right)
=
\int\rd\bv F(\bv)\rho_{k}^{(\text{\tiny GOE})}\left(\bv\right).
$$

\noindent{\it Averaged energy universality  (in the bulk).} For any  $k\geq 1$,  $F:\RR^k\to\RR$ continuous and compactly supported, 
and for any  $\e,\kappa>0$, we have, uniformly in $E\in[-2+\kappa,2-\kappa]$,
$$
\lim_{N\to\infty}\frac{1}{\varrho(E)^k}\int_{E}^{E+s}\frac{\rd x}{s}\int \rd \bv F(\bv)\rho_{k}^{(N)}\left(x+\frac{\bv}{N\varrho(E)}\right)\rd \bv
=
\int\rd \bv F(\bv)\rho_{k}^{(\text{\tiny GOE})}\left(\bv\right), \qquad  s:= N^{-1+\e}. 
$$

Fixed energy universality obviously implies averaged energy universality. 
As mentioned in the introduction, fixed energy universality was proved for Hermitian matrices from the generalized Wigner ensemble. This required the use of the Br\'ezin-Hikami-Johansson formula, a tool with no known analogue for symmetric matrices.  General methods developed in the past five years, such as the local relaxation flow, allowed to prove universality for the symmetric class only in the sense of averaged energy. Our result establishes universality at fixed energy, with no need for any averaging, for the symmetric class. It  also  provides a new proof for the Hermitian class.

\begin{theorem}[Universality at fixed energy]\label{thm:main}
For symmetric or Hermitian  matrices from the generalized Wigner ensemble satisfying (\ref{eqn:forHansonWright}), fixed energy universality holds in the bulk of the spectrum.
\end{theorem}

The above theorem implies for example that the joint interval probabilities converge,
 more precisely, 
for  disjoint intervals $I_1,\dots,I_\ell$, and integers  $n_1,\dots,n_\ell\in\NN$,  the limit 
$$
\lim_{N\to\infty}\P\left(\left|\{x_i\in E+\frac{I_j}{N\varrho(E)}\}\right|=n_j,1\leq j\leq \ell\right)
$$
exists. It is independent of $E\in(-2,2)$ and of the details of the distributions of the matrix entries, in particular they can
be computed in the Gaussian ensemble where more explicit formulas
are available. 
For example, the gap probability for Bernoulli random matrices converges on the microscopic scale:
$$
\lim_{N\to\infty}\P\left(\{x_i\in E+\frac{[0,t]}{N\pi\varrho(E)}\}=\varnothing\right)=E_{1}(0,t),
$$
where $E_1$ can be made explicit from the solution to a Painlev\'e  equation of fifth type \cite{JimMiwMorSat1980,TraWid1996}. \\

Before giving the main ideas of the proof, we introduce the typical locations of eigenvalues with respect to the semicircular distribution: they are defined by  
\begin{equation}\label{eqn:typ}
\int_{-2}^{\gamma_k}\rd\varrho=\frac{k  + \frac{1}{2} \nc}{N},  \qquad k=1,2,\ldots, N.\nc
\end{equation}

\subsection{Sketch of the proof.}\ We now outline the main steps towards the proof of Theorem \ref{thm:main}, in the symmetric case. As mentioned in the introduction, it does not rely on improvements of existing methods such as the local relaxation flow.
 The Dyson Brownian motion plays again a key role in the following approach, but surprisingly our method requires
 understanding its behaviour for relatively
large time, $t=N^{-\tau}$,  for some small $\tau$,
 instead of $t=N^{-1+\e}$ for small $\e$.\\

\noindent{\it First step. Coupling and discrete integral operator.} 
We run a coupled DBM with two different initial conditions,
one from the Wigner ensemble we wish to study and one from
a comparison Gaussian ensemble.
At time $0$, let $\bx=\bx(0)$ be the ordered spectrum of a generalized Wigner matrix, and let $\by(0)$ 
denote the eigenvalues of an independent GOE matrix.   In the actual proof
we have to start the coupling at a time $t_0\sim N^{-\tau_0}$, $\tau_0> \tau$ instead of time 0, but
we neglect this technical issue in the current presentation. 
Consider the unique strong solutions for the following Dyson Brownian motion, more precisely its Ornstein-Uhlenbeck version:
\begin{align}\label{DBMdynamics}
&\rd x_\ell(t)=\sqrt{\frac{2}{N}}\rd B_{\ell}(t)+\left(\frac{1}{N}\sum_{k\neq \ell}\frac{1}{x_\ell(t)-x_k(t)}-\frac{1}{2}x_\ell(t)\right)\rd t,\\
&\rd y_\ell(t)=\sqrt{\frac{2}{N}}\rd B_{\ell}(t)+\left(\frac{1}{N}\sum_{k\neq \ell}\frac{1}{y_\ell(t)-y_k(t)}-\frac{1}{2}y_\ell(t)\right)\rd t.\notag
\end{align}
Note that the underlying Brownian trajectories $(B_\ell)_{1\leq \ell\leq N}$ are the same.
Then the normalized differences $\delta_\ell(t):=e^{t/2}(x_\ell(t)-y_\ell(t))$ satisfy an integral equation of parabolic type, namely
\begin{equation}\label{aftercoupling}
\partial_t \delta_\ell(t) = 
\sum_{k\neq \ell} b_{k\ell}(t)(\delta_k(t)-\delta_\ell(t)),\ \ b_{k\ell}(t)=\frac{1}{N(x_\ell(t)-x_k(t))(y_\ell(t)-y_k(t))}.
\end{equation}

\noindent{\it Second step. Homogenization.}
We consider the following  continuous analogue of (\ref{aftercoupling}):
\begin{equation}\label{contop}
 \partial_tf_t= -Kf_t, \qquad (K f)(x):=\int_{-2}^2\frac{f(x)-f(y)}{(x-y)^2}\varrho(y)\rd y. \nc
\end{equation}
A key step in our approach consists in proving that (\ref{contop}) gives a good approximation for (\ref{aftercoupling}). Indeed, if the initial conditions match in the sense that  $f_0$  is smooth enough and $f_0(\gamma_k)=\delta_k(0)$, 
then for any $t=N^{-\tau}$,  with a sufficiently small  $\tau>0$ 
there exists $\e>0$ such that for any bulk index $\ell$ (i.e.  $\ell\in\llbracket \al N,(1-\al)N\rrbracket$ for some small fixed $\al>0$) we have
\begin{equation}\label{homo}
\delta_\ell(t)=(e^{-t K}f_0)_\ell+\OO(N^{-1-\e}).
\end{equation}
The above homogenization result holds for generic trajectories $\bx(t)$, $\by(t)$. It relies on optimal rigidity estimates for these eigenvalues \cite{ErdYauYin2012Rig},  a level repulsion bound similar to \cite{ErdSchYau2010}, 
 and H\"older continuity for equations of type (\ref{aftercoupling}), obtained in \cite{ErdYau2012}.\\

\noindent{\it Third step. The continuous heat kernel.} The heat kernel for the equation (\ref{contop})  can be
expressed by an explicit formula, see   (\ref{pt}). For
short times,  i.e.  $\tau$ close to 1 it almost coincides with $e^{-t|p|}$ where $|p|=\sqrt{-\Delta}$.

 However, we will need $\tau$ close to $0$, hence the effect of the curvature from the semicircle
 law cannot be  neglected, and the explicit formula  will be useful.
This allows us to compute explicitly $e^{-tK} f_0$ and to rewrite (\ref{homo}) as follows. There exists $\e>0$ such that, for 
a fixed $E$ in the bulk and for any $\ell$
satisfying $|\gamma_\ell-E|<N^{-1+\e}$, we have
\begin{equation}\label{appro}
x_\ell(t)-y_\ell(t)=\wt\zeta_t^\bx-\wt\zeta_t^\by+\OO(N^{-1-\e}),
\end{equation}
where 
\begin{equation}\label{linearstat}
\wt\zeta_t^\bx:=\frac{1}{N}\sum_{k=1}^N \left(P_t(x_k(0))-P_t(\gamma_k)\right),
\end{equation}
 is a smooth  linear statistics of $\bx=\bx(0)$  on the  {\it  mesoscopic} scale $t=N^{-\tau}\ll 1$. \nc
Here $P_t$ is an  explicit function, the
antiderivative of the heat kernel (\ref{pt}) (see (\ref{eqn:Pt})). 
We repeat the above steps with the initial condition $\bx$ replaced by $\bz$,
 the spectrum of another  GOE independent of $\bx$ and $\by$. 
In summary, 
 we proved
\begin{equation}\label{resulthomo}
x_\ell(t)=y_\ell(t)-\wt\zeta_t^\by+\wt\zeta_t^\bx+\OO(N^{-1+\e}),\ \ \ 
z_\ell(t)=y_\ell(t)-\wt\zeta_t^\by+\wt\zeta_t^\bz+\OO(N^{-1+\e}).
\end{equation}

\noindent{\it Fourth step: Reformulation of universality through mesoscopic observables.}\
For any continuous and compactly supported test
function $Q:\RR^k\to\RR$ and $E\in(-2,2)$, define
$$\mathcal{Q}(\bx,E):=\sum_{i_1,\dots,i_k=1}^N 
Q(N(x_{i_1}-E),N(x_{i_2}-x_{i_1}),\dots,N(x_{i_k}-x_{i_1})).$$
Theorem \ref{thm:main} can be restated as
\begin{equation}\label{aim}
\E\mathcal{Q}(\bx(0),E)=
\E\mathcal{Q}(\bz(0),E)+\oo(1).
\end{equation}
Let $\widehat Q$ denote the Fourier transform in the first variable.
By a standard approximation argument, it is sufficient to prove (\ref{aim}) 
for any $Q$ such 
 that $\widehat Q$ is compactly supported, in $[-m,m]$, say. We will first prove that (\ref{aim}) holds for the corresponding DBM trajectories after
 some time $t=N^{-\tau}$, where $\tau$ will depend on $m$:
\begin{equation}\label{aim2}
\E\mathcal{Q}(\bx(t),E)=
\E\mathcal{Q}(\bz(t),E)+\oo(1).\\
\end{equation}
Using the representation (\ref{resulthomo}), we easily see that (\ref{aim2}) holds if
we have
\begin{equation}\label{aim3}
\E\mathcal{Q}(\by(t)-\wt\zeta_t^\by,E-\wt\zeta_t^\bx)=
\E\mathcal{Q}(\by(t)-\wt\zeta_t^\by,E-\wt\zeta_t^\bz)+\oo(1).
\end{equation}
Note that $\by(t)$ and $\wt\zeta_t^\by$ are independent from $\wt\zeta_t^\bx$ and $\wt\zeta_t^\bz$.
Moreover, (\ref{aim3}) is simpler than the original microscopic universality problem, because $\wt\zeta_t^\bx$ and $\wt\zeta_t^\bz$ are mesoscopic observables.
In the next two steps, we explain how (\ref{aim3}) can be proved under the following, strange, compatibility assumption between the dynamics time and the Fourier support of the test function:
\begin{equation}\label{tau}
\tau\leq \frac{c}{m^2}.
\end{equation}

\noindent{\it Fifth step. Mesoscopic fluctuations for the Gaussian orthogonal ensemble.} 
To justify (\ref{aim3}), we first prove that the distribution of  $\zeta_t^\bz:=N \wt\zeta_t^\bz$ 
  is very close to a Gaussian, with variance of order $\tau\log N$. This central limit theorem for the linear statistics of type (\ref{linearstat}) relies on the method initiated in \cite{Joh1998}. For reasons apparent in the next step, we will need to control
the the distribution of  $\zeta_t^\bz$  even beyond
its natural scale $(\tau\log N)^{1/2}$. More precisely, we will prove that, for some fixed constants $b$ and $c>0$,  we have
\begin{equation}\label{gaussmes}
\widehat\mu_\bz(\la)=\widehat\mu_{\bz,t}(\la)=\E\left(e^{-\ii\la \zeta_t^\bz}\right)=e^{-\frac{\la^2}{2}\tau\log N-\ii\la b}+\OO(N^{-c}).
\end{equation}
For {\it macroscopic} linear statistics, corresponding to  $t$ independent of $N$ in our notation, 
Johansson proved the central limit theorem in \cite{Joh1998}
 by considering the logarithm of their Laplace transform. The proof of (\ref{gaussmes}) involves additional technicalities because $\widehat\mu_\bz$ may vanish, see Section \ref{sec:gauss}. In particular we will need rigidity estimates from \cite{BouErdYau2013} to prove (\ref{gaussmes}).\\

\noindent{\it Sixth step. Reverse heat flow.}
For any fixed $a,h\in\RR$, consider the functions
$$
F(a):=\E\mathcal{Q}(\by(t)-\wt\zeta_t^\by,E-a),\quad F_h(a):=F(a-h)-F(a).
$$
We can express the convolution of $F_h$ with $\mu_\bz$ as follows:
\begin{multline}\label{speed}
(F_h\ast \mu_\bz)(a)=
\E\mathcal{Q}(\by(t)-\wt\zeta_t^\by,E-a+h-\wt\zeta_t^\bz)-
\E\mathcal{Q}(\by(t)-\wt\zeta_t^\by,E-a-\wt\zeta_t^\bz)\\
=
\E\mathcal{Q}(\bz(t),E-a+h)-
\E\mathcal{Q}(\bz(t),E-a)+\OO(N^{-c})=\OO(N^{-c}).
\end{multline}
for some $c>0$. Here the second equality follows from 
the second formula in \eqref{resulthomo}. The last
step in \eqref{speed} uses that
on  microscopic scales and with
a high accuracy the  distribution
of the GOE spectrum is translation invariant;
a fact that follows from an effective polynomial speed of the
the convergence (\ref{GOEconv}) uniformly in $E$ in the bulk.

From the estimate  \eqref{speed} on $F_h\ast\mu_\bz$ we  bound $F_h$. This is a reverse heat flow type of question, because $\mu_\bz$ is almost a Gaussian distribution, from the previous step. One can reverse the heat flow because (1) $F_h$ is analytic, explaining our original Fourier support restriction on $Q$, and (2) the estimate (\ref{gaussmes}) is precise enough. Namely, taking the Fourier transform in (\ref{speed}), we obtain
$$
\widehat F_h(\la)=\OO\left(N^{-c}\widehat\mu_\bz(\la)^{-1}\right)=\OO(N^{m^2\tau-c})
$$ 
for any $\la$ in the Fourier support of $Q$, where we used (\ref{gaussmes}). This explains why we need (\ref{tau}) in order to prove that $\widehat F_h$, and then $F_h$, are $\oo(1)$. We therefore obtained
$$
\E\mathcal{Q}(\by(t)-\wt\zeta_t^\by,E-a+h)=\E\mathcal{Q}(\by(t)-\wt\zeta_t^\by,E-a)+\oo(1)
$$
uniformly in $a,h$. It is then elementary, by simple convolution, to prove (\ref{aim3}) and therefore (\ref{aim2}).\\

\noindent{\it Seventh step. Green function comparison theorem.} 
Finally, to obtain universality for the eigenvalues $\bx(0)$ of
the initial Wigner ensemble from the time evolved ones $\bx(t)$, $t=N^{-\tau}$,
we use a Green function comparison theorem, 
of type close to the one introduced in \cite{ErdYauYin2012Univ}: (\ref{aim2}) implies (\ref{aim}).

\section{Homogenization}

From this section, we only consider symmetric matrices which is the most involved case
since the level repulsion estimate requires an additional regularization. Section~\ref{sub:regul}
would not be  necessary for $\beta>1$, the rest of the proof is insensitive to the value of $\beta$. 

\subsection{Regularized dynamics.} \label{sub:regul} \ Our goal is to estimate the coupled difference $\delta_\ell(t)=e^{t/2}(x_\ell(t)-y_\ell(t))$, which satisfies the dynamics (\ref{aftercoupling}).  Notice that the singularity of the coefficient $b_{jk}$ is not integrable
and this will create serious difficulties in the analysis. 
We first perform a cutoff to tame this singularity
 which requires a level repulsion estimate. Since such estimate
holds  only for large enough time, we will perform the regularization
only after an initial time $t_0:=N^{-\tau_0}/2$ with some $\tau_0>0$. We then show that  
 the difference between the original and cutoff dynamics is negligible for times $t\in (t_0, 1)$.
 The estimates in this section
 are valid for any fixed $\tau_0$.  
 We assume  $\tau_0\le 1$ which is the relevant regime.   

We recall the equation for the $\bx(t)$ dynamics from \eqref{DBMdynamics} and 
for $t\in (t_0,1)$, we define its  regularized version as 
\be 
\rd \widehat x_j(t)=\sqrt{\frac{2}{N}}\rd B_{j}(t)+\left(\frac{1}{N}\sum_{k\neq j}\frac{1}{x_j(t)-x_k(t) + \e_{jk}} - \frac{1}{2}\widehat x_j(t)  \right)\rd t\label{hat}
\ee
with $\e_{jk} =\e$ for $j>k$, $\e_{jk}:=-\e$ for $j<k$,
and we set $\widehat x_i(s):=x_i(s)$, for $s\leq t_0$. 
Notice that $\widehat \bx(t)$  may not 
preserve the ordering,  but we do not need this property.  
Let $q_i:= N( x_i - \widehat x_i)$
 denote the rescaled difference  between the original dynamics 
and the regularized one. 
It satisfies  the equation ($t>t_0$)
\be\label{DBMd}
 \frac { \dd q_{i}} { \dd  t}  =  \Omega_i- \frac {q_i} 2, \qquad \mbox{with}\quad
 \Omega_i(t):=\sum_{j\not= i}\frac{\e_{ij}}{(x_{i}-x_{j})(t)((x_{i}-x_{j})(t) + \e_{ij})   }.
\ee
Since $q_i(t_0)=0$, we  can solve this equation by
\be
  q_i(t) = \int_{t_0}^t   e^{-(t-s)/2} \Omega_i(s) \rd s, \qquad t\ge t_0.
\ee
Let $p>2$ and $p'$ be its conjugate exponent. We have 
\be\label{schwa}
\E \sup_{t_0 \le t \le 1}  \left | \int_{t_0}^t \Omega_i(s) \rd s \right | 
\le \sup_{t_0 \le t \le 1} \E |\Omega_i (t)| 
\le \sup_{t_0 \le t \le 1} \sum_{j\not= i}
\left(\E \frac{1}{|x_{i}-x_{j}|^{p'} (t)}\right)^{1/p'} 
\left( \E \frac{\e^p}{\big|(x_{i}-x_{j})(t) + \e_{ij}\big|^p} \right)^{1/p}.
\ee 
Recall the rigidity estimate
from  \cite{ErdYauYin2012Rig} 
  asserting that  for any $\xi,D>0$,  if  for some large  moment $p=p(\xi,D)$  (\ref{eqn:forHansonWright}) is satisfied, then  there exists $C>0$ such that
\be\label{rigi}
\P (   \mathcal{G}_{\xi,\bx,i}) \ge  1-C N^{-D}, \quad \mbox{with}\quad 
 \mathcal{G}_{\xi,\bx,i}: = \{|x_i(t)-\gamma_i|\leq N^{-2/3+\xi}(\widehat i)^{-1/3},0\leq t\leq 1\},
\ee
where $\widehat i:=\min(i,N+1-i)$  (the subscript $\bx$ refers to the $\bx(t)$ process). 
 The original rigidity estimate in   \cite{ErdYauYin2012Rig} was formulated for any fixed generalized
Wigner ensemble, i.e. for a fixed $t$. A minor continuity argument in the time variable
ensures that rigidity holds simultaneously for all times in a compact interval
(see Lemma 9.3 of \cite{ErdYau2012singlegap} for a similar argument). 
Note also that the original rigidity estimate from  \cite{ErdYauYin2012Rig}
assumes subexponential decay of the entries distribution, but this is easily weakened to
the finite moment assumption (\ref{eqn:forHansonWright}) (see remark 2.4 in \cite{ErdKnoYauYin2013}).

Denote by $g_i(t):=x_{i+1}(t)-x_i(t)$ the gap at $i$-th location. 
A trivial estimate yields that 
\[
\left(\E \mathcal{G}_{\xi,\bx,i}\frac{\e^p}{( g_i(t) + \e)^p} \right)^{1/p} \le \left( \E \mathcal{G}_{\xi,\bx,i}\frac{\e^2}{ (g_i(t) + \e)^2} \right)^{1/p}
\]
(with a slight abuse of notations we write $\mathcal{G}_{\xi,\bx,i}$ 
instead of its characteristic function within the expectation). 
Using   the  level repulsion estimate, i.e.,  Corollary \ref{cor:Wegner}, 
 we have for any $\xi>0$ that 
\be\label{lr}
\left( \E \mathcal{G}_{\xi,\bx,i}\frac{\e^2}{ (g_i(t) + \e)^2} \right)^{1/p} \le C_0 \e^{2/p}  N^{2/p} N^{(C_0\tau+\xi)/p}|\log\e|^{1/p},
\qquad  t =  N^{-\tau} \nc \in [t_0, 1] 
\ee
where $C_0$ is the constant  from Corollary \ref{cor:Wegner}.
 We  introduced the notation $t=N^{-\tau}$ and we will use $t$ and $\tau$ 
in parallel, similarly to the notation $t_0=N^{-\tau_0}/2$. 
The other factor in \eqref{schwa}  is even easier to estimate  and it gives
$$
\left(\E \frac{1}{|x_{i}-x_{j}|^{p'} (t)}\right)^{1/p'} \le  N^{1/p'} N^{(C_0\tau+\xi)/p'} .
$$
\nc Choosing for example  $\e= N^{-3C_0-100 }$,   $2<p<3$,  and $\xi$ small, 
 we therefore proved that
\be
\P(\cG_{\bx}) \ge 1 -  C N^{-2}, \quad \mbox{with} \quad \mathcal{G}_{\bx}:=
\{ \sup_{t_0 \le t \le 1, i\in\llbracket 1,N\rrbracket} | x_i - \widehat x_i|(t) \le  N^{-2C_0-50}   \}.
\ee
Hence
the trajectories of $\widehat \bx$ and $\bx$ are very close to each other.

We also regularize the $\by(t)$ dynamics, i.e. we have 
\begin{align}\label{DBMdynamics3}
&\rd y_j(t)=\sqrt{\frac{2}{N}}\rd B_{j}(t)+\left(\frac{1}{N}\sum_{k\neq j}\frac{1}{y_j(t)-y_k(t)}
- \frac 1 2 {y_j(t)} \right)\rd t,\\
&\rd \widehat y_j(t)=\sqrt{\frac{2}{N}}\rd B_{j}(t)+\left(\frac{1}{N}\sum_{k\neq j}\frac{1}{y_j(t)-y_k(t) + \e_{jk}} 
- \frac 1 2 {\widehat y_\ell(t)}\right)\rd t.\nonumber
\end{align}
with the same definition  as previously for $\e_{ij}$, and $\widehat y_j(t):=y_j(t)$ for $t<t_0$. 
Note that  $B_j$ represents the same Brownian motion in each of the
  equations  \eqref{DBMdynamics}, \eqref{hat} and \eqref{DBMdynamics3}.
 With a similar argument,  we can assume that $\by$ is very close to $\widehat \by$ on another set $\cG_{\by}$. We now define 
$\cG_1:= \cG_\bx \cap \cG_\by$.  

Now we analyse the difference of the two cutoff dynamics. Setting
 $w_i:= Ne^{t/2}( \widehat  x_i - \widehat y_i)$,  $\bw$ satisfies 
an equation of the form ($t>t_0$) 
\begin{align}\label{DBMd}
 &\frac { \dd w_{i}} { \dd  t} =\frac{1}{N}\sum_{j\not= i}\frac{w_j-w_i}{[(x_{i}-x_{j})(t) + \e_{ij}][(y_{i}-y_{j})(t) + \e_{ij}]} + 
\zeta_i
\end{align}
with an error term $\zeta_i $ satisfying
\[
|\zeta_i (t) |  \le  \sum_{j\not= i}\frac{  |x_i - \widehat x_i| +  |x_j - \widehat x_j| + |y_i - \widehat y_i|+|y_j - \widehat y_j|  }{\left((x_{i}-x_{j})(t) + \e_{ij}\right)\left((y_{i}-y_{j})(t) + \e_{ij}\right)}.\notag
\]
With the level repulsion estimate  as in \eqref{lr},  we have
\be\label{schwar}
  \E \cG_1 |\zeta_i (t)| \le   CN^{-2C_0-50}   \sum_{j\not= i}    \left( \E \frac{1}{\left((x_{i}-x_{j})(t) + \e_{ij}\right)^2}  \right)^{1/2} 
\left(\E \frac{1}{\left((y_{i}-y_{j})(t) + \e_{ij}\right)^2} \right)^{1/2}
\le C N^{-C_0-30}.
\ee
Therefore, there is a set $\cG_2$ with $\P(\cG_2) \ge 1 - C N^{-1}$ such that on this set we have 
\be\label{xi}
\sup_{t_0\le t \le 1} \max_i \int_{t_0}^t|\zeta_i(s)|\rd s\le N^{-C_0-20}.
\ee

We now show that 
  $\zeta_i$ is negligible in the equation \eqref{DBMd}.  
 This follows from the stability of  
 the   parabolic equation
\be\label{ve1}
  \pt_{ t} \bv( t)= -\cB( t) \bv( t),
\ee
where $\cB$ is the positive and positivity preserving matrix defined by
\begin{align}
& \left(\cB ( t)\bv\right)_i=\sum_{j=1}^N B_{ij}(t) (v_i-v_j),\\
\label{Be}
&B_{ij}(t):= 
\left\{
\begin{array}{l}
\frac{1}{N (x_{i}-x_{j})(t) (y_{i}-y_{j})(t)}\ \quad \mbox{if}\ t\le t_0,\\
\frac{1}{N \left((x_{i}-x_{j})(t) + \e_{ij}\right)\left((y_{i}-y_{j})(t) + \e_{ij}\right)}\ \quad \mbox{if}\ t>t_0.
\end{array}
\right.
\end{align}
Indeed, suppose that $\bw$ satisfies  \eqref{DBMd}, i.e.
\be\label{we}
  \pt_{ s} \bw( t)= -\cB( t) \bw(t) + \b\zeta(t)
\ee
with  the same initial data  at time $t_0$, $\bv(t_0)=\bw(t_0)$, 
  and $\b\zeta$ satisfying the estimate \eqref{xi}. Then we have 
\be\label{we}
  \pt_{ t} (\bw - \bv)( t)= -\cB( t) (\bw-\bv) ( t) +  \b \zeta(t)
\ee
with vanishing initial data.
Let $U_\cB(s,t)$ denote the semigroup associated with \eqref{ve1}
from time $s$ to time $t>s$, i.e.
\be\label{ucAdef}
 \pt_t U_\cB( s,  t) = -\cB(t)  U_\cB( s,  t)
\ee
for any $ t\ge s$ and $U_\cB(s,  s)=I$.
 By the Duhamel formula, we have  
\[
 (\bw - \bv)(t) = \int_{t_0}^t \cU_\cB (s, t) \b \zeta (s) \rd s .
\]
Since $\cU_\cB$ is a contraction (in any $L^p$ norm, in particular in $L^\infty$), using (\ref{xi}) in $\mathcal{G}_2$ we have 
$$
\Lnorm \infty {(\bw - \bv)(t)} \le  N^{-C_0-20}, \qquad  t\in [t_0, 1], 
$$
i.e., the effect of the perturbative term $\b \zeta$ on the solution  is negligible.

To summarize, we proved that  the set $\mathcal{G}=\mathcal{G}_1\cap\mathcal{G}_2$ satisfies  $\Prob(\mathcal{G})\geq 1-C N^{-1}$ uniformly
 in  $0\leq \tau\le 1$   
 and in this set  $\mathcal{G}$ we have,  for any $i$ and $t\in( t_0 ,1)$, that
\begin{equation}\label{eqn:regularisee}
Ne^{t/2}(x_i(t)-y_i(t))=v_i(t)+\OO\left(N^{-1}\right)
\end{equation}
where $\bv$ satisfies (\ref{ve1}) with initial condition  $\bv^0=  N (\bx^0-\by^0)$.

\subsection{Continuous space operator.}\ We now construct an operator in the continuum which approximates the discrete operator defined by $\cB$.  
Recall the definition of the typical location $\gamma_k$ from (\ref{eqn:typ}).
If we replace  $x_i$  and $y_i$ 
by $\gamma_i$ and neglect the regularization $\e$, \nc
 we have the following  classical  operator $\cU$ on $\ell^2(\llbracket  1, N\rrbracket)$:
\be\label{Ufdef}
    \left(\cU \bu\right)_j : = \sum_{i\neq j} \frac{1}{N |\gamma_i-\gamma_j|^2} (u_j-u_i).
\ee
We now define an operator $K$ acting on smooth functions on $[-2,2]$ as
\be\label{Kdef}
(K f)(x)=
\int_{-2}^2\frac{f(x)-f(y)}{(x-y)^2}\rd \varrho(y),
\ee
where the integral is in the principal value sense. Then $K$ is the continuum limit of $\cU$ in the sense 
that, for large $N$,
$
(\cU \bu^f)_j\approx (K f)(\gamma_j)$, where $\bu^f_j =  f(\gamma_j). 
$
The following lemma provides an explicit formula for the evolution kernel $e^{-t K}$.

\begin{lemma}\label{lem:diagonalization}
Let $f$ be smooth with all derivatives uniformly bounded.
For any $x,y\in(-2,2)$, denote $x=2\cos\theta$, $y=2\cos\phi$ with 
$\theta,\phi\in(0,\pi)$.
Then 
\be
(e^{-t K } f)(x) = \int p_t(x,y) f(y) \rd \varrho(y) 
\ee
where the kernel is given by
\begin{align}\label{pt}
p_t(x,y)   & : =  \frac{1-e^{-t}}{(1+e^{-t}- 2 e^{-t/2} \cos (\theta+\phi))(1+e^{-t}- 2 e^{-t/2} \cos (\theta-\phi))}
=\frac{1-e^{-t}}{|e^{\ii(\theta+\phi)}-e^{-t/2}|^2\, |e^{\ii(\theta-\phi)}-e^{-t/2}|^2}.
\end{align}
\end{lemma}

\noindent{\it Remark 1.} The above formula is the same as the one in 
\cite[page 462]{Bia1997}. Lemma \ref{lem:diagonalization} shows that Biane's $q$-Ornstein Uhlenbeck generator coincides (for $q=0$) with the convolution kernel $(\ref{Kdef})$.\\

\noindent{\it Remark 2.} If we neglect the curvature of the semicircle,  i.e. $\gamma_i$'s
are equidistant on scale $1/N$, \nc and  formally extend the operator $\cU$ to $\bZ$, 
we obtain  the following translation invariant operator $\cU^\infty$ on $\ell^2(\bZ)$:
\be\label{Uinfdef}
    (\cU^\infty \bu)_j : = \sum_{i\in \bZ\setminus\{ j\}} \frac{ N }{|i-j|^2} (u_j-u_i), \qquad \bu\in \ell^2(\bZ).
\ee
The Fourier transform of the kernel $1/k^2$ is given by
$
   \sum_{k\in \bZ\setminus\{ 0\}} \frac{1}{k^2} (1- e^{-ikp}) = \frac{\pi}{12}|p|,
$
i.e. 
$$
   \wh{ (\cU^\infty \bu)}(p) = c_0  N |p| \wh u(p), \qquad  c_0: =  \frac{\pi}{12}, \quad p\in [-\pi, \pi], 
$$
where 
$
   \wh u(p): = \sum_{k\in \bZ} e^{-ipk} u_k$ and $u_k = \frac{1}{2\pi}\int_{-\pi}^\pi e^{ipk} \wh u(p) \rd p.
$
Therefore the heat kernel of $\cU^\infty$ can be computed by Fourier transform for any $t\ge 0$:
\be\label{Uinfty}
   e^{-t\cU^\infty}(0,k) = \frac{1}{2\pi} \int_{-\pi}^\pi e^{-tc_0 N |p|} e^{-ikp} \rd p 
  = \frac{1}{N} \frac{2c_0 t}{(tc_0)^2 + (k /N  )^2}\Big( 1- (-1)^ke^{-\pi c_0 t}\Big).
\ee
The operator  $\cU^\infty$ is the discrete analogue of the operator $N \sqrt{-\Delta}$. 
 The heat kernel $e^{-t N \sqrt{-\Delta}}$
is closely related to $p_t(x,y)$, but, compared with \eqref{pt},  there are substantial differences near the edges and also when $t$ is large.

\begin{proof}[Proof of Lemma \ref{lem:diagonalization}]
Let $U_n$ be the Chebishev polynomial of the second kind, defined by
$$
U_N(\cos\theta)=\frac{\sin((n+1)\theta)}{\sin\theta}.
$$
and we defined $P_n(x)=U_n(x/2)$. The proof relies on the following diagonalization:
for any $n\geq 0$, 
\begin{equation}\label{eqn:diago}
K P_n =\frac{n}{2}P_n.
\end{equation}
For $n=0$, (\ref{eqn:diago}) is obvious. For $n=1$ this is the classical equilibrium relation
\begin{equation}\label{eqn:initialmiracle}
\int_{-2}^2\frac{\rd \varrho(y)}{x-y}=\frac{x}{2}.
\end{equation}
The following recursion relation is classical:
\begin{equation}\label{eqn:rec}
P_{n+1}(x)=x P_n(x)-P_{n-1}(x).
\end{equation}
This yields, assuming that (\ref{eqn:diago}) holds up to the index $n$, 
\begin{align*}
K P_{n+1}(x)&=x K P_n(x)-K P_{n-1}(x)
+\int_{-2}^2 \frac{P_n(y)}{x-y}\rd \varrho(y)\\
&=x \frac{n}{2}P_n(x)-\frac{n-1}{2}P_{n-1}(x)
+P_n(x)\int_{-2}^2 \frac{\rd\varrho(y)}{x-y}
+\int_{-2}^2 \frac{P_n(y)-P_n(x)}{x-y}\rd \varrho(y)\\
&=\frac{n+1}{2} x P_n(x)-\frac{n-1}{2}P_{n-1}(x)
+\int_{-2}^2 \frac{P_n(y)-P_n(x)}{x-y}\rd\varrho(y)
\end{align*}
where we used (\ref{eqn:initialmiracle}).
Hence  (\ref{eqn:diago})  will be proved with $n+1$ instead of $n$ if
\begin{equation}\label{eqn:sufficient}
\int_{-2}^2 \frac{P_n(x)-P_n(y)}{x-y}\rd\varrho(y)=P_{n-1}(x)
\end{equation}
holds.
To prove  (\ref{eqn:sufficient}) for any $n$, one can again proceed by induction. This formula is obviously true for $n=0,1$.
Assuming it is true up to index $n$, with (\ref{eqn:rec}) we get 
$$
\int_{-2}^2 \frac{P_{n+1}(x)-P_{n+1}(y)}{x-y}\rd\varrho(y)
=
x P_{n-1}(x)+\int P_{n-1}(y)\rd\varrho(y)-P_{n-2}(x)=P_n(x)
$$
where we used that $P_{n-1}$ is orthogonal to 1 with respect to the semicircle measure.
 This concludes the proof of (\ref{eqn:sufficient}) and therefore (\ref{eqn:diago}) for all  $n$.

The conclusion of the lemma now easily follows:
the kernel, defined through $(e^{-t K } f)(x) = \int p_t(x,y) f(y)\rd\varrho(y)$, can be written in the eigenbasis as 
$
p_t(x,y)=\sum_{n\geq 0} e^{-\frac{n}{2}t}P_n(x)P_n(y).
$
Using the representation $P_n(2\cos\theta)=\sin((n+1)\theta)/\sin\theta$ and expanding the $\sin$ to get four geometric series concludes the proof.
\end{proof}

We record some properties of the kernel \eqref{pt} that easily follow 
from the explicit formula and from the asymptotics $\gamma_j+2\sim (j/N)^{2/3}$
for $j\le N/2$,
\be\label{pupp}
   p_t (\gamma_i, \gamma_j)\le \frac{Ct}{t^2 + (\gamma_i-\gamma_j)^2}, \qquad i\in \llbracket \alpha N, (1-\al)N\rrbracket,  \quad j\in \llbracket 1, N \rrbracket, \quad t\le 1,
\ee
\be\label{psum}
  \sum_j  p_t (\gamma_i, \gamma_j)
\le \sum_j \frac{Ct}{t^2 + (\gamma_i-\gamma_j)^2}
 \le C, \qquad i\in \llbracket \alpha N, (1-\al)N\rrbracket,    \quad t\le 1,
\ee
\be\label{pderi}
 |\pt_x  p_t (\gamma_i, x )|\le
\frac{Ct|\gamma_i-x|}{ (t^2 + (\gamma_i-x)^2)^2}
 \qquad i\in \llbracket \alpha N, (1-\al)N\rrbracket,  \quad t\le 1,
\ee
where the constant $C$ depends only on the positive parameter $\al>0$.

\nc

\subsection{The homogenization result.}\
For  any $\delta\in \R$  and $E\in (-2,2)$  we define the index set
\be\label{definICE}
I (\delta)=I( E, \, \delta):= \{i:\;  \left|\gamma _i -E\right|\le N^{-1+\delta }\}.
\ee
The main result of this section is the following theorem.  

 \begin{theorem}\label{T1}
 Suppose $\bx(t)$ and $\by(t)$ are two DBM 
 driven by the same Brownian motions (see (\ref{DBMdynamics})) and with initial data given by the spectra of two generalized Wigner matrices.
There exist positive constants 
 $\tau_0\le 1/4$,  $\delta_1,\de_2,\de_3$ such that 
 for any  $t\in[ 2t_0,1]$, with $t_0:=N^{-\tau_0}/2$, 
  and $|E|<2-\kappa$ with $\kappa>0$  we have 
 \be\label{zhimao}
\P\left( \max_{ i\in I (E, \delta_1)} \left|Nx_i(t)-N  \Big (y_i(t)  + 
 (\Psi_{ t-t_0} \b x (t_0))_i -  (\Psi_{t-t_0} \b y (t_0))_i \Big)  \right|\ge N^{-\delta_2} \right) \le N^{-\delta_3},
\ee
where $ \Psi_s$ is a linear operator defined by 
\be\label{Psidef}
(\Psi_s \b x)_i := e^{-s/2}  \frac{1}{N}
 \sum_j p_s(\gamma_{i}, \gamma_j)  x_j.
\ee
\end{theorem}

The main tool to prove Theorem \ref{T1} is a homogenization result. In order to state it, 
we first construct a partition of unity as follows.  
 For any $j=1,2,\ldots, N$ let 
$$
\tilde g_j : = \min (\gamma_{j+1}  - \gamma_j, \gamma_{j}-\gamma_{j-1})  \sim N^{-2/3}(\wh j)^{-1/3}
$$
 with the convention that $\gamma_{N+1} = \infty$, $\gamma_{ 0} = - \infty$. 
For all even indices $j$ define a smooth  function $\xi_j$ supported in $[\gamma_{j-1}+\tilde g_j/100, \gamma_{j+1}
-  \tilde g_j/100]$ 
and   $\xi_j(x) =1$ for $ |x-\gamma_j| \le  \tilde g_j/100 $
    such that 
\be\label{normal}
 \int_{\gamma_{j-1}}^{\gamma_j}  \xi_j(x)  \rd\varrho(x) 
= \int_{\gamma_{j}}^{\gamma_{j+1}}  \xi_j(x)  \rd\varrho(x)   \nc
 = \frac{1}{2N}.
\ee
For odd indices $j$, we
define $\xi_j$ by $\xi_j(x) = 1-\xi_{j-1}(x)$ 
for $\gamma_{j-1} \le x \le \gamma_j$ and
by $\xi_j(x) = 1-\xi_{j+1}(x)$ 
for $\gamma_{j} \le x \le \gamma_{j+1}$.  We thus
have $\xi_j (x) + \xi_{j+1}(x) =1$ for any $j$ whenever  $x\in [\gamma_j, \gamma_{j+1}]$. 
In particular,
$\sum_j  \xi_j(x) = 1$.  Notice that
by construction 
\be\int   \xi_j(x) \rd\varrho(x)  = \frac{1}{N},\qquad 
\mbox{supp} \,\xi_j\subset [\gamma_{j-1}+\tilde g_j/100, \gamma_{j+1}
-  \tilde g_j/100]
\label{normxi}
\ee  
hold for all $j=1,2,\ldots, N$. 
For any discrete function  (i.e. vector)  ${\bf v}: i\to v_i$ define 
 its continuous extension by
\be\label{e-def}
 e_{\bf v}  (x):= \sum_j \xi_j(x) v_j.
\ee
 Notice that $e_\bv(\gamma_i)= v_i$.

The main homogenization result is the  following theorem.  It is
formulated for the parabolic equation \eqref{ve1} with general
random coefficients $B_{ij}(t)$ under certain conditions. Later we will verify  
that rigidity and level repulsion for $\bx(t)$ and $\by(t)$  imply
that $B_{ij}(t)$ defined in \eqref{Be} satisfy these conditions.

\begin{theorem}\label{thm:appr1}  
Fix three small positive constants, $\xi, \rho, \al$. Choose $\tau_0\in [0,  \frac{1}{4}] $,
 and set  $t_0:= N^{- \tau_0}/2$.
Consider the equation \eqref{ve1} with time dependent random coefficients  $B_{ij}(s)$
 in the time interval $s\in [ t_{in}, t_{end}]$ with $t_0< t_{end}-t_{in}\le C$. 
Denote by  $\cF= \cF_{\xi}$ the event on which the following two bounds hold:
\be\label{B1}
  \Big| B_{ij}(s) - \frac{1}{N(\gamma_i -\gamma_j)^2}\Big| \le \frac{ N^{ -\frac{2}{3}+\xi}\big[ (\wh i)^{-1/3}
+ (\wh j)^{-1/3}\big] \nc }{N(\gamma_i -\gamma_j)^3} ,  
 \qquad \forall i, j, \; |i-j|\ge N^\xi, \quad \forall s\in [ t_{in}, t_{end} ],
\ee
\be\label{B2}
    B_{ij}(s) \ge \frac{N^{-\xi}}{N(\gamma_i -\gamma_j)^2}, \qquad \forall i,j, \quad \forall
 s\in [t_{in}, t_{end}].
\ee
Furthermore we assume that 
\be\label{EB}
   \max_{ij} \max_{s\in [t_{in}, t_{end}]} \E \; \big[ \cF  |B_{ij}(s)|\big]
 \le \frac{N^{\rho}}{N|\gamma_i-\gamma_j|^2}.
\ee
If $\xi$  and $\rho$ are small enough, then 
there are  constants $c_4,  c_5> 0$,
so that the following holds.
For any fixed space-time
point  $(t, i)\in [ t_{in}+t_0, t_{end} \nc]\times  \llbracket \alpha N,  (1- \alpha) N \rrbracket$
there is an event $\cS\subset\cF $ 
with $\P(\cF\setminus \cS)\le N^{-c_4}$ so that on $\cS$ we  have 
\be\label{eq:appr}
     \Big|     (U_\cB(t_{in},  t) \bv)_{i} -  
      \int p_{ t-t_{in}\nc}(\gamma_{i}, y) 
       e_{\bv} (y)\rd\varrho(y) \Big| \le N^{-c_5}
     \| \bv\|_{\infty}
\ee
for any   vector $\bv\in \bR^N$
and for any sufficiently large $N\ge N_0(\al)$.  
Note that the set $\cS$ depends on the choice $(t, i)$, but the exponents $c_4, c_5$ do not.
\end{theorem}

{\it Remark.}   Our proof  can easily be extended to hold for any $\tau_0<1/3$, but
then the smallness of $\xi, \rho, c_4, c_5$ will depend on how close $\tau_0$ is to 1/3.
However, even the 1/3 threshold for the exponent $\tau_0$ is not 
 optimal, it is
due to various cutoffs that  can  be improved with more work.   We do not pursue this direction 
since, for the purpose of this paper,
 only  the small $\tau_0$ regime is needed. 

The following statement asserts that rigidity and level repulsion
estimates on the DBM trajectories ensure that the conditions
 in Theorem~\ref{thm:appr1} hold for $B_{ij}$ given 
in \eqref{Be}  with a high probability provided that $\tau_0$ is small.

\begin{theorem}\label{thm:appr2} 
 There exist positive constants $c_4, c_5>0$ and $\tau_0>0$ \nc
such that the following holds.  Fix   $ \alpha \in (0,1)$,  
 set $t_0:= N^{-\tau_0}/2$. 
Consider the equation \eqref{ve1} with coefficients $B_{ij}(s)$  given by  two coupled DBM $\bx(s)$
and $\by(s)$, $s\in [0,1]$,
 as defined in \eqref{Be}. Then
for any space-time
point  $(t, i)\in [2 \nc t_0,1]\times  \llbracket \alpha N,  (1- \alpha) N \rrbracket$
 there exists a set $\cS= \cS(t, i)$ in the
joint probability space of the coupled DBM's  $\bx(s)$
and $\by(s)$, with $\P (\cS) \ge 1-   N^{-c_4}$,
such that on the set $\cS$ \eqref{eq:appr} holds
for any $ \bv \nc \in \bR^N$ 
and $N\ge N_0(\al, \tau_0)$.
\end{theorem}

{\it Proof.}  Choose $t_{in}:= t_0$, $t_{end}:=1$ in Theorem~\ref{thm:appr1}. 
The estimates \eqref{B1} and \eqref{B2}  directly
follow from the rigidity bound \eqref{rigi}  on the set 
$$
  \wt\cF_\xi:= \bigcap_i \cG_{\xi, \bx, i}\cap \cG_{\xi, \by, i},
$$
thus $\cF_\xi\supset \wt\cF_\xi$
(here we used the fact that the parameter 
  $\e$  in the definition  of $B_{ij}$
is much smaller than the rigidity threshold $N^{-1+\xi}$).
 From \eqref{rigi} $ \wt\cF_\xi$  has  a very high probability, 
 $\P(\wt\cF_\xi)\ge 1-N^{-D}$
for any $D$ (note that $\cF$ and $\wt\cF$ are independent of $\tau_0$). 
For \eqref{EB} 
 we claim that
$$
  \E \cF_\xi  |B_{jk}(s)| \le \frac{N^{C_0\tau_0+3\xi}}{N|\gamma_j-\gamma_k|^2}
$$
holds for any $\xi>0$. Indeed, for $|j-k|\ge N^\xi$ this 
follows form the rigidity estimates. For $|j-k|\le N^\xi$ with $j<k$
one may estimate $B_{jk}(s) \le B_{j,j+1}(s)$ and then use 
a Schwarz inequality similar to \eqref{schwar}. Finally, applying
Corollary~\ref{cor:Wegner} as in \eqref{lr}, we get
$$
  \E B_{j,j+1}(s) \le N^{C_0\tau+\xi} N^{1/3} (\wh j)^{2/3}|\log \e| 
  \le \frac{N^{C_0\tau_0+2\xi}}{N|\gamma_j-\gamma_{j+1}|^2} 
    \le \frac{N^{C_0\tau_0+3\xi}}{N|\gamma_j-\gamma_{k}|^2} .
$$
Setting $\rho = C_0\tau_0 +3\xi$, we verified \eqref{EB}.
Choosing $\tau_0$ and $\xi$  sufficiently small,
 we can apply Theorem~\ref{thm:appr1}
to conclude \eqref{eq:appr}. For the probability of $\cS$ we have
$\P(\cS) \ge 1- N^{-D} - N^{-c_4}$ which 
satisfies the required bound by reducing $c_4$ a bit.
\qed

\begin{proof}[Proof of Theorem \ref{T1}]  
Pick positive constants $c_4, c_5, \tau_0$ sufficiently small  so that
 Theorem \ref{thm:appr2} applies. Without loss of generality
 we may assume that  $\tau_0, c_4, c_5\le 1/100$ and $\tau_0\le c_5/100$. 
Recall the notation $t_0= N^{-\tau_0}/2$. 
For brevity we write
$\bx:=\bx(t_0)$ and $\by:= \by(t_0)$.  We would like to apply
 Theorem \ref{thm:appr2} for the vector $\bv$  of the form 
there is such a factor: x and y evolve by OU, but v does not. \nc  
\be\label{eqn:bound1}
 v_j : = N e^{t_0 /2}(x_j- y_j) \lesssim N^\xi  (\wh j/N)^{-1/3},
\ee
but then $\|\bv\|_\infty \sim N^{1/3+\xi}$ in (\ref{eq:appr}) 
would be too large, as the edge indices contribute.
 So we have to perform a cutoff  and  use \eqref{eq:appr} only for the bulk indices
and use an $L^1\to L^\infty$ heat kernel  bound to control the contribution near the edge.
  We therefore
rewrite $\bv = \bw+ \bu$ where $w_j:=v_j$ if  $N^{1-\nu} \le j \le N - N^{1-\nu} $ and $w_j:=0$ 
otherwise, for some exponent   $\nu>0$ chosen later. 
 Equation  \eqref{eq:appr} with initial condition $\bw$ yields 
\be\label{eq:appr11}
    \Big|     (U_\cB(t_0, t) \bw)_{i} -   \int p_{t-t_0}(\gamma_{i}, y)  e_{\bw} (y) \rd \varrho(y) \Big|  \le N^{-c_5}
     \| \bw\|_{\infty}\le N^{-c_5 + \frac{\nu}{3} + \xi}
\ee
on the set $\cal S(t, i)$ for any $i\in I(E,\delta_1)$. 
 Using the definition of $e_\bw$, 
from \eqref{pderi}, \eqref{psum} and  \eqref{normxi} we have
\begin{align*}
\Big | \int p_{t-t_0}(\gamma_{i}, y) & e_{\bw} (y)\rd \varrho(y)  - \frac{1}{N} \sum_j p_{t-t_0}
(\gamma_{i}, \gamma_j) w_j \Big| \\
&
 \le  \sum_j |w_j| \Big| \int p_{t-t_0}(\gamma_i, y) \xi_j(y)\rd\varrho(y) - 
 \frac{1}{N} p_{t-t_0}(\gamma_i,\gamma_j)\Big| \\
&
 \le \sum_j |w_j|  \int |p_{t-t_0}(\gamma_i, y)-  p_{t-t_0}(\gamma_i,\gamma_j)| \xi_j(y)\rd\varrho(y)
\\
& \le C  N^{-2+\frac{2\nu}{3}+\xi} \sum_j 
\frac{t|\gamma_i-\gamma_j|}{ (t^2 + (\gamma_i-\gamma_j)^2)^2}
  \le  C t^{-1} N^{-1+\frac{2\nu}{3}+\xi}, 
 \quad  i\in I(E,\delta_1), \quad  t\ge 2t_0. 
\end{align*}
since $|y-\gamma_j|\le CN^{-1+\nu/3}$
and $|w_j|\le N^{\frac{\nu}{3}+\xi}$
 on the support of $\xi_j$ with $\wh j\ge N^{1-\nu}$.
Moreover, from \eqref{pupp} and using that $u_j\ne0$
 only for $\wh j\le N^{1-\nu}$, 
\be\label{ud}
 \Big| \frac{1}{N} \sum_j p_{t-t_0}
(\gamma_{i}, \gamma_j) u_j \Big|  \le \frac{Ct}{N}\sum_j |u_j| \le 
Ct N^{-\frac{2}{3}\nu+\xi},
\quad  i\in I(E,\delta_1), \quad  t\ge 2t_0. 
\ee
Together with \eqref{eq:appr11} 
this gives   that \nc
\be\label{eq:appr1}
    \Big|    (U_\cB(t_0,t) \bw)_{i}  -  \frac{1}{N}\sum_j p_{t-t_0}(\gamma_{i}, \gamma_j) v_j \Big| 
\le C \left(N^{-c_5 + \frac{\nu}{3} + \xi} + 
  t^{-1} N^{-1+\frac{2\nu}{3}+\xi}  
 +t N^{-\frac{2}{3}\nu+\xi}
\right)
\ee
on the set  $\cS(t, i)$ 
for any $t\ge 2t_0$,  $i\in I(E,\delta_1)$.

Moreover, thanks to the following  Proposition \ref{decayB}  (in our application $b=N^{-\xi}$,  we use \eqref{B2} 
and we shift the initial time from 0 to $t_0$ \nc) we have 
\be\nonumber
\Lnorm \infty  {  U_\cB(t_0, t)\bu -\bar u} \le C N^{3\xi } t^{- 3}  N^{-1} \sum_j |u_j- \bar u|, \qquad
 t\ge 2t_0,
\ee
where $ \bar u: =  N^{-1}  \sum_j u_j$ and thus $|\bar u|  \le CN^{- 2\nu/3+\xi}$.  Hence we have proved that 
\be\label{infb}
\Lnorm \infty  { U_\cB(t_0, t)\bu } \le  CN^{- \frac{2}{3}\nu+\xi} + t^{- 3}  N^{-\frac{2}{3}\nu+4\xi}.
\ee
Combining this with
 \eqref{eq:appr1}, choosing  $\nu= c_5$, $\xi= c_5/100$ 
  and recalling $\tau_0\le c_5/100$, 
we have  proved that
\be\label{conl}
\Big|     (U_\cB(t_0,t) \bv)_{i} -  N^{-1} \sum_j p_{t-t_0}(\gamma_{i}, \gamma_j) v_j \Big|  \le
 C N^{-c_5/2}, 
\ee  
on the set $\cS(t, i)$     for any $i\in I(E,\delta_1)$
and any $t$ with $1 \ge  t\ge  2t_0$. 

Finally, we need to guarantee that \eqref{conl} holds for all $i \in I(E,\delta_1)$ simultaneously,
i.e. we take the intersection 
$$
    \cS(t): = \bigcap_{i\in I(E, \delta_1)} \cS(t, i).
$$
The cardinality of $I(E, \delta_1)$ is bounded by $CN^{\delta_1}$ and
$\P(\cS(t, i))\ge 1-N^{-c_4}$,  so by choosing $\delta_1<c_4$, 
we obtain that $\P (\cS(t))\ge 1-\frac{1}{2} N^{-c_4}$. Now we choose $\delta_2 < c_5/2$ and $\delta_3< c_4$ and
together with (\ref{eqn:regularisee}), we conclude the proof of Theorem~\ref{T1}. \nc
\end{proof}

For any $\bu\in\bR^N$ we define the $\ell^p$ norms as
$$
  \| \bu\|_p : = \big(  \frac{1}{N}  \sum_{i=1}^N |u_i|^p\big)^{1/p}.
$$
The following decay estimate extends   Proposition  10.4   of \cite{BouErdYau2013}. 
Notice that the convention of $\ell^p$ norm in this paper differs from that used in 
\cite{BouErdYau2013} by a normalization factor $N^{-1}$. 

\begin{proposition}\label{decayB} 
  Suppose that the coefficients of the equation \eqref{ve1}    satisfy  for some
constant $b$  that 
\be\label{B}
  B_{jk}(s) \ge   \frac b   { N (\gamma_j-\gamma_k)^2}, \quad 0 \le s \le \sigma.
\ee
Then  for any  $\bu$ with $\sum_j u_j = 0$  we have the decay estimate
\be\label{decay}
\|  U_\cB(0,s)\bu  \|_\infty  \le C ( sb)^{-3} 
    \| \bu \|_{ 1 }, \quad 0 \le s \le \sigma.
\ee
\end{proposition}

\begin{proof}
We first prove the same inequality for
the operator $K$, i.e., 
 for any mean zero function $f$ that 
\be\label{decayK}
  \| e^{-2 t K} f\|_\infty \le \frac{C}{ t^3 }  \|f\|_{ \varrho,  1},
\ee
where
\be
 \| f \|_{\varrho, p} := \left ( \int |f(x)|^p   \rd \varrho(x) \right )^{1/p}.
\ee
Recall Corollary 4 of \cite{Bia1997}  (see also  (10.19) of \cite{BouErdYau2013}) asserting  that there is a constant $C$  so that 
\be\label{sobolev1}
\| f \|_{\varrho, 3}^2 \le C \left (\| f \|_{\varrho, 2}^2+ \langle f, K f \rangle_\varrho\right).
\ee
By  the explicit diagonalization  of $K$, \eqref{eqn:diago},
 the spectral gap of $K$ is equal to  $1/2$.  Hence 
for  $\int f \rd \varrho =0$, we have $\| f \|_{\varrho, 2}^2 \le  2  \langle f, K f \rangle_\varrho$
and thus   
\be\label{sobolev}
\| f \|_{\varrho, 3}^2 \le  C \langle f, K f \rangle_\varrho.
\ee
We shall drop the subscript $\varrho$ in the following argument. 
Suppose $f_t$ solves the equation 
\be
\partial_t f_t = -K f_t
\ee
and the initial data has zero mean, i.e.,  $\int f_0  \rd \varrho = 0$. 
Then we have 
$$
   \partial_s \| f_s\|_2^2 = -\langle f_s, K f_s \rangle  \le -  C  
  \|f_s\|_2^{\frac{8}{3}} \|f_s\|_1^{-\frac{2}{3}},
$$
where we have used \eqref{sobolev} and  the H\"older inequality 
$$
   \| f \|_3^2 \ge \|f\|_2^{\frac{8}{3}} \|f\|_1^{-\frac{2}{3}}.
$$
Since $\|f_s\|_1$ is non-increasing, we can integrate this inequality to have 
$$
\| f_t\|_2 \le \frac{C}{ t^{3/2}}  \|f_0\|_1.
$$
For any $g$ with $\int g\rd\varrho=0$,  we have 
\be
| \langle g, e^{-tK} f_t \rangle| = | \langle e^{-tK} g,  f_t \rangle|  \le \Lnorm 2 {e^{-tK} g} \Lnorm 2 {f_t} 
\le \frac{C}{ t^3 }  \|g\|_1 \|f_0\|_1.
\ee
Since $\int e^{-tK} f_t \rd \varrho = \int f_0 \rd \varrho = 0$ by assumption, the mean zero condition of $g$ 
can be removed and we have thus proved  \eqref{decayK}.

We can now follow the similar argument to  prove  \eqref{decay}. After a time rescaling, we can assume that $b=1$. 
The key ingredient in the previous argument 
is the Sobolev inequality \eqref{sobolev1}. Now we will need a discrete version. This can be achieved 
by extending  a discrete function to the continuum 
with a simple interpolation procedure.
This idea was used in \cite{BouErdYau2013} and we will not repeat it here. 
Once a discrete version of \eqref{sobolev1} is proved, the rest of the proof is identical to the one in the continuum. 
Thus  we have  proved \eqref{decay}.
\end{proof}

\subsection{Proof of Theorem \ref{thm:appr1}.}\
Without loss of generality, we can assume that $t_{in}=0$
 by a simple time shift. For simplicity, we also set
 $t_{end}=1$,  as the actual value of $t_{end}$ influences only
 irrelevant constant prefactors. 
 By definition \eqref{e-def} and the equation \eqref{ve1}, we have 
\be\label{ucAdef-1}
 \pt_t e_{\bv(t)}  = - \cR_t  \bv(t),
\ee
where 
\be
(\cR_t  \bv) (x) := \sum_{j, k=1}^N  \xi_j(x) ( v_j- v_k ) B_{kj}(t) 
\ee
takes  the vector $\bv$ to the  function $\cR_t \bv$ for any fixed $t$. 
Suppose that $f=f(t,x)$ is a solution to the continuum equation 
\be
\partial_t f(x) = -(K f)(x),
\ee
where $K$ is defined in \eqref{Kdef}. 
Then we have 
\be
\partial_t (e_{\bv(t)} - f(t))  = -K (e_\bv(t) - f(t )) + [K e_{\bv(t)} - \cR_t \bv(t)  ].
\ee
We will need to solve this equation from time $t_{in}=0$ to $t$. 
We will take the  initial condition  at time $t_{in}=0$ 
to be  $\bv$ for the discrete equation and $f(0) = e_{\bv}$ 
for the continuous one.

By the Duhamel formula, we have  
\be
e_{\bv(t)} - f(t) = \int^t_{0} e^{-(t-s) K }  [K e_{\bv(s)} - \cR_s \bv(s)  ]\rd s =   \Phi -\Omega,  
\ee
where  the functions $\Phi, \Omega$ are given by 
\begin{align}
& \Omega (z)  :=
\int_{0}^t \rd s \int  d\varrho(x)   p_{t-s}(z, x) \int   \rd \varrho(y) \frac{e_\bv (s, y)- e_\bv (s, x)}{|x-y|^2},   \label{Omega1}\\
& \Phi  (z)  : =  \int_{0}^t  \rd s \int \int \rd \varrho(x) \rd \varrho(y)  \sum_{j, k}  p_{t-s}(z, x)  \xi_j(x)  (v_k(s)- v_j(s) ) B_{kj}(s)    \label{Phi1}.
\end{align} 
We have used that $\int   \rd \varrho(y)  =1$ in \eqref{Phi1}.
 We will need these functions for $z:=\gamma_i$ in order to obtain \eqref{eq:appr},
but we will keep the shorter $z$ notation. Note that $e_{\bv (t)}(\gamma_i) = v_i(t) = (U_\cB(0,t )\bv)_i$,
thus the left hand side of \eqref{eq:appr} is  $\Phi(z)-\Omega(z)$ with $z=\gamma_i$. 

\nc

\bigskip
 
\noindent 
{\it Step 1: Cutoff of long range part.} We first cutoff the contributions to $\Omega, \Phi$ when $|x-y| \ge \ell$ for some 
$  N^{-2/3}  \ll \ell \ll 1 $ 
to be fixed later on.  In this regime, we will need to use cancellation between $\Omega$ and $\Phi$. 
We start with the following definition that for any subset $D$ in $\R^2\times \bR$ define 
\be
\Omega_D   :=    \int^{t }_0  \rd s \int  \int {\bf 1}_D (x, y, s)  \rd\varrho(x) \rd \varrho(y)  p_{t-s}(z, x)   \frac{e_\bv (s, y)- e_\bv (s, x)}{|x-y|^2}.
\ee
If $D$ is  symmetric under $x \leftrightarrow y$, then we have 
\be
\Omega_D = \frac 1 2  \int^{t }_0  \rd s \int\int {\bf 1}_D (x, y, t)  \rd\varrho(x) \rd \varrho(y)  [p_{t-s}(z, x) - p_{t-s}(z, y)   ] \frac{e_\bv (s, y)- e_\bv (s, x)}{ |x-y|^2}.
\ee
Similarly we can define 
\be\label{Phi11}
\Phi_D:=  \int_0^t \rd s \int \int \rd \varrho(x) \rd \varrho(y)  {\bf 1}_D (x, y, t)  \sum_{j, k}  p_{t-s}(z, x)  \xi_j(x)  (v_k(s)- v_j(s) ) B_{kj}(s),
\ee
and for symmetric $D$ we have 
\be
\Phi_D= \frac 1 2  \int_0^t \rd s \int \int \rd \varrho(x) \rd \varrho(y)  {\bf 1}_D (x, y, t) \sum_{j, k}  [p_{t-s}(z, x)  \xi_j(x) - p_{t-s}(z, y)   \xi_k(y) ]  (v_k(s)- v_j(s) ) B_{kj}(s)  \label{Phi2}.
\ee 
Let 
\be
\widehat A_{\ell} :=  \{ (x, y): 
|x-y| \ge \ell  \},  \quad D_\ell: =  \widehat A_{\ell} \times [0, t].
\ee
Using \eqref{Phi1} and \eqref{Omega1}, we can 
decompose  the error term  $\Phi_{D_\ell}-\Omega_{D_\ell}$ in this region into 
$\Phi_{D_\ell}^1-\Omega_{D_\ell}^1 + \Phi_{D_\ell}^2-\Omega_{D_\ell}^2$,
where 
\be\label{d4}
\Phi_{D_\ell}^1-\Omega_{D_\ell}^1 :=  - \int_{0}^ {t} \rd s \int\int  \rd \varrho(x)  \rd \varrho(y)   
{\bf 1}_{\widehat A_{\ell}}(x, y)   p_{t-s}(z, x)     
 \sum_{j}   \xi_j(x)   v_j(s)   \Big [  \sum_k    B_{jk}(s) - \frac{1}{|x-y|^2} \Big ].
\ee
The second term is 
\begin{align}
 \Phi_{D_\ell}^2-\Omega_{D_\ell}^2 =  \int_{0}^ {t} \rd s\int  \int  \rd \varrho(x) \rd \varrho(y) {\bf 1}_{\widehat A_{\ell}}(x, y)  p_{t-s}(z, x)     
 \sum_{k}   v_k(s)  \Big [   \sum_j  \xi_j (x)       B_{jk}(s) -   \xi_k(y)  \frac{1}{|x-y|^2} \Big ].
\end{align}
Since $\sum_j \xi_j (x)=1$, we can replace   $ \xi_k(y)  \frac{1}{|x-y|^2}$ by 
$\sum_j \xi_j (x) \xi_k(y)  \frac{1}{|x-y|^2}$. Recall  the normalization condition 
 $\int   \rd \varrho(y)  \xi_k(y) =1/N$ \eqref{normxi}.  
If  we could neglect the factor ${\bf 1}_{\widehat A_{\ell}}(x, y)$  in this normalization,   we had 
\begin{align}\label{d2}
 \Phi_{D_\ell}^2-\Omega_{D_\ell}^2    \approx \int_{0}^ {t} \rd s \int\int  \rd \varrho(x) \rd \varrho(y)  {\bf 1}_{\widehat A_{\ell}}(x, y)   p_{t-s}(z, x)     
 \sum_{k, j}   v_k(s)   \xi_j (x)   \xi_k(y)  \Big [N  B_{jk}(s) -     \frac{1}{|x-y|^2} \Big ].
\end{align}
 Indeed, the difference between the two sides of \eqref{d2} is
\be\label{d3}
\int_{0}^ {t} \rd s \int  \rd \varrho(x)     p_{t-s}(z, x)     
 \sum_{k, j}   v_k(s)   \xi_j (x)  B_{jk}(s)\int_{|x-y|\ge \ell} \rd \varrho(y) \Big[ 1-  \xi_k(y)  N\Big].
\ee
Notice that $|\gamma_j -x |\le \wt g_j$ from the support property of $\xi_j(x)$ and that
the last integral is zero unless the support of $\xi_k(y)$ overlaps with one of the
boundaries $y= x\pm \ell$ of the integration regime.  Thus \eqref{d3} can be bounded from above by
\begin{align}
 \int_{0}^ {t}  \rd s \int\int_{|x-y|\ge \ell} & \rd \varrho(x) \rd \varrho(y)    p_{t-s}(z, x)     
 {
   \sum_{k, j, ||\gamma_j-\gamma_k| - \ell| \le 4 \wt g_j + 4 \wt g_k}  } |v_k(s)|  \xi_j (x)   \xi_k(y)  N  B_{jk}(s)
\nonumber \\
\le & C\| \bv\|_\infty\ell^{-2} \int_{0}^ {t}  \rd s \int\int_{|x-y|\ge \ell}  \rd \varrho(x) \rd \varrho(y)    p_{t-s}(z, x)     
 {
   \sum_{k, j, ||\gamma_j-\gamma_k| - \ell| \le 4 \wt g_j + 4 \wt g_k}  }  \xi_j (x)   \xi_k(y)\nonumber
\\
\le & C\| \bv\|_\infty\ell^{-2} \int_{0}^ {t}  \rd s \int\int {\bf 1}\Big( \big| |x-y|-\ell \big|\le CN^{-1} [\varrho(x)^{-1} +
\varrho(y)^{-1}] \Big)
\rd \varrho(x) \rd \varrho(y)    p_{t-s}(z, x)\nonumber
\\ \le & C \| \bv\|_\infty\ell^{-2} N^{-1}t, \label{bdr}
\end{align}
where in the first step we used that under the constraints
on the summations, we have $NB_{jk}\le C\ell^{-2}$ from \eqref{B1} and from the fact that $\wt g_j, \wt g_k \ll \ell$.
In the second step we translated the constraint on the indices $j,k$ to a constraint
on $x, y$ using that $\wt g_j \sim \varrho^{-1}(\gamma_j)$, 
and finally we integrated out $y$, $x$ and $s$ in this order.  We also used
the contraction property $\|\bv (s) \|_\infty \le \|\bv\|_\infty$.

 For the term on the r.h.s. of \eqref{d2},
we will  use the coordinate system  $x = 2 \cos \Theta(x)$ with $0 \le \Theta(x) \le \pi$. 
From the  estimate \eqref{B1},  for $|x-y| \ge \ell  \gg  N^{-2/3}\nc $ we have  
\be
 \sum_{k, j}    \xi_j (x)   \xi_k(y)  \Big |N  B_{jk}(s) -     \frac{1}{|x-y|^2} \Big | 
\le N^{-1+ \xi}  \frac{1}{ |x-y|^3 }    [ (\sin \Theta(x))^{ -1}+ (\sin \Theta(y))^{-1}   ] . 
\ee
Together with  $\|\bv (t) \|_\infty \le \|\bv\|_\infty$, we have  
\begin{align}
\Phi_{D_\ell}^{ 2 }-\Omega_{D_\ell}^{ 2} & \le  N^{-1 + \xi} \Lnorm \infty {\bv}  \int_0^ {t} \rd s \int\int_{ |x-y| \ge \ell}\rd \varrho(x)  \rd\varrho(y)    p_{t-s}(z, x)      
 \frac{1}{|x-y|^3}  [ (\sin \Theta(x))^{-1 }+ (\sin \Theta(y))^{ -1}   ]
  \nonumber \\
 &   \label{err1}
\le 
\frac { t  |\log \ell|  } {\ell^2}  N^{-1 + \xi} \Lnorm \infty {\bv}.
\end{align}
 In the last step we used that $z$ is away from the edge, so in the regime where $x$ is
near the edge and $[\sin \Theta(x)]^{-1}\sim \varrho(x)^{-1}$ becomes singular,
we know that $p_{t-s}(z,x)\le C$ from \eqref{pt}.

The estimate of the $\Phi_{D_\ell}^1-\Omega_{D_\ell}^1$ term is similar.  We write the $\rd \varrho(y)$ integration 
in \eqref{d4} as
\begin{align*}
  \int_{|x-y|\ge \ell} \rd \varrho(y) & \Big [  \sum_k    B_{jk}(s) - \frac{1}{|x-y|^2} \Big ]
\\
&  =  \int_{|x-y|\ge \ell} \rd \varrho(y)  \sum_k \xi_k(y)  \Big[ NB_{jk}(s) -\frac{1}{|x-y|^2} \Big ] 
 +  \sum_k    B_{jk}(s)  \int_{|x-y|\ge \ell} \rd \varrho(y) (1-N\xi_k(y))
\end{align*}
The first term can be estimated  exactly the r.h.s. of \eqref{d2} (the only difference
is $v_j(s)$ in \eqref{d4} instead of $v_k(s)$ in \eqref{d2} but these factors
are estimated by $\| \bv\|_\infty$ anyway). The second term is analogous to \eqref{d3}.
This completes the estimate of the regime $|x-y|\ge \ell$.

From now on, we will  work on the complement of $\wh A_\ell$, i.e. in the regime $|x-y|\le \ell$.
We will   not use cancellation between $\Phi$ and $\Omega$ and will estimate them separately
 by splitting the integrals into further subregions. 
As the estimates for $\Phi$ and $\Omega$ 
are similar,   we will work out only  one of them in every region.\\

\noindent 
{\it Step 2: Time region away from the final time $t$  via  
the energy bound. } 
 In this step, we estimate the
 contribution to the integrals \eqref{Omega1}, \eqref{Phi1}  for times $s \in  [0, t-t_1]$ 
with some $t_1 \ll t$.  The main idea to deal with this regime
 is to use energy bound for 
the dynamics \eqref{ve1} and the regularity of the continuous evolution kernel $p_{t-s}(x, y)$.

 We start with a general estimate to show how energy bound is used to control $\Omega_D$.
For a set $D\subset \R^2\times [0,t]$, symmetric under $x\leftrightarrow y$, using 
 the Schwarz inequality, we have 
\begin{align}\label{OmS}
&\Omega_D =   \frac{1}{2}   \int^{t}_0  \rd s \int\int {\bf 1}_D (x, y, t)  \rd\varrho(x) \rd \varrho(y)  [p_{t-s}(z, x) - p_{t-s}(z, y)   ] \frac{e_\bv (s, y)- e_\bv (s, x)}{|x-y|^2}  \le \sqrt { W_1^D W_2^D}, \\
&W_1^D :=   \frac{1}{2}  \int^{t }_0\rd s\int\int  \rd\varrho(x) \rd \varrho(y)  {\bf 1}_D [p_{t-s}(z, x) - p_{t-s}(z, y)   ] ^2   \frac{1}{|x-y|^2}, \\
&W_2^D :=    \frac{1}{2}  \int^{t }_0  \rd s\int\int   \rd\varrho(x) \rd \varrho(y)  {\bf 1}_D \frac{ [e_\bv (y)- e_\bv (x)]^2}{|x-y|^2}.
\end{align}
We start with the second term $W_2^D$.
For $\gamma_j \le x \le \gamma_{j+1} $ and $\gamma_k \le y \le \gamma_{k+1} $, from the construction of $\xi$'s, we have  
\begin{align}
e_\bv (x) - e_\bv (y)  & = \xi_j(x) v_j+ \xi_{j+1}(x) v_{j+1} -\xi_k(y) v_k -\xi_{k+1}(y) v_{k+1} \nonumber \\
& = (1-\xi_{j+1}(x)) v_j+ \xi_{j+1}(x) v_{j+1} -(1-\xi_{k+1}(y)) v_k -\xi_{k+1}(y) v_{k+1}\nonumber\\
& =v_j - v_k +  \xi_{j+1}(x)  (v_{j+1} - v_j) - 
   \xi_{k+1}(y)  (v_{k+1} - v_k).
\end{align}
In particular, when $j=k$, we have 
\[
|e_\bv (x) - e_\bv (y)| = \Big |   \xi_{j+1}(x) -\xi_{j+1}(y)  \Big | 
 |v_{j+1} - v_j| \le C
 \wt g_{j+1}^{-1}  |x-y| |v_{j+1} - v_j|  \le \frac{C |x-y| |v_{j+1} - v_j|}{|\gamma_{j+1}-\gamma_j|^2}. \nc 
\]
For neighboring indices, i.e.   when $k=j+1$,
\begin{align}
e_\bv (x) - e_\bv (y)   =  [\xi_{j+1}(x)-1]  (v_{j+1} - v_j) - 
   \xi_{j+2}(y)  (v_{j+2} - v_{j+1}) .
\end{align}
Notice that when $|x-y| \le  \wt g_{j+1}/400 $, we have $|x-\gamma_{j+1}| \le \wt g_{j+1}/200$,
 $|y-\gamma_{j+1}| \le \wt g_{j+1}/200$  and by definition  of the $\xi$'s we have
$\xi_{j+1}(x) =1$, $\xi_{j+2}(y)=0$, so \nc $e_\bv (x) - e_\bv (y)  = 0$. 
Therefore for any set $D$, we can bound  $W_2^D$  by   
\begin{align}
W_2^D  & = \sum_{j,  k} \int^{t}_0   \int_{\gamma_j}^{\gamma_{j+1} }   \rd\varrho(x) 
\int_{\gamma_k}^{\gamma_{k+1}} \rd \varrho(y)   \frac{ [e_\bv (x) - e_\bv (y)]^2}{|x-y|^2}  \rd s  \nonumber  \\
& \le
C  \sum_{j \not =  k} \int^{t}_0    \frac{ [v_j (s)- v_k (s)]^2}{   N^2 |\gamma_j- \gamma_k|^2} \rd s 
+ C \sum_j  \int^{t}_0
\frac{ [v_j (s)- v_{j+1} (s)]^2}{   N^2 |\gamma_j- \gamma_{j+1}|^2}  \rd s  \nonumber\\
& \le  N^\xi  
 \frac 1 N  \sum_{j,  k} \int^{t}_0    [v_j (s)- v_k (s)]^2 B_{kj}(s)  \rd s  \label{W2} \\ 
& =  N^\xi   \frac 1 N  \sum_j \big[ v_j^2(0) - v_j^2(t)\big]   \le N^\xi   \|\bv\|_2^2 \le   N^{\xi}  \|\bv\|_\infty^2. \nonumber 
\end{align}
In \eqref{W2} we have used \eqref{B2}.
 The last step is the energy estimate that can be obtained by integrating the time derivative
$\pt_s\| \bv(s)\|_2^2$.

Choose $t_1    = t N^{-2\fa} $ with some $\fa>0$ to be fixed later
and define  
$ D_1: = \{ |x-y| \le \ell \} \times [0, t-t_1]$
(Here  $D_1$ is
a new set, not to be confused with $D_\ell$ defined earlier). 
For $z$ in the bulk,  we can use the explicit formula of $p_t$ \eqref{pt}
so that
\be
W_1^{D_1}=   \int^{t-t_1}_0   \rd s \int_{ |x-y| \le \ell}   \rd\varrho(x) \rd \varrho(y) [p_{t-s} (z, x) - p_{t-s} (z, y)  ]^2  \frac{1}{|x-y|^2}
 \le   \frac \ell { t_1^2} . 
\ee
Together with \eqref{OmS} and \eqref{W2}, we have proved that  $|\Omega_{D_1}| \le C  N^\xi \frac {\sqrt \ell}  { t_1} \Lnorm \infty {\bv} $. 
Similarly, we can bound $\Phi_{D_1}$  and obtain
\be\label{err2}
|\Phi_{D_1}| + |\Omega_{D_1}| \le C N^\xi \frac {\sqrt \ell}  { t_1} \Lnorm \infty {\bv}.
\ee

\noindent 
{\it Step 3: Time region near the final time $t$  via   the H\"older regularity. } 
 In this step  and the next one  we consider the final time region $s \in [t-t_1, t]$. 
Notice that we will not use  the smoothness of the continuous kernel $p_{t-s}(x, y)$ which depends on  $t-s$
and becomes singular when $s$ is close to $t$.  Instead, in Step 3 we consider the regime in \eqref{Phi2} where $x$
(hence also $y$) is not too far from the fixed reference point $z$. In this case
we will use the the H\"older 
regularity of the solution to the equation \eqref{ve1}. In Step 4, we look at the complement
regime,  when $x$ and $y$ are far from $z$,  and we can use the large distance decay
of the kernel $p_t$. \nc

 We first recall this basic H\"older estimate from 
  \cite{ErdYau2012singlegap}. We will need
this result in the following form and in Appendix \ref{sec:Holder}
we will explain how this particular version follows from
the general statement in \cite{ErdYau2012singlegap}.

\begin{lemma}\label{holder}
For any $t \ge t_0=N^{-\tau_0}/2$ and a small constant $0< \fa<1-\tau_0$ fixed, 
we set $\ell_1= tN^{-\fa}$.
For any real $z$ with $|z| < 2$ 
define 
\be\label{psi}
\Xi_z(\ell_1):= 
 \big\{ (j, k): 1 \le j, k \le N, 
|\gamma_j- z|\le  \ell_1, \;  |\gamma_k-  z|\le   \ell_1 \big\}.
\ee
Consider the equation \eqref{ve1} with coefficients \eqref{Be} satisfying \eqref{B1}--\eqref{EB}.
If the exponent $\rho>0$ in \eqref{EB} is sufficiently small, depending on $\fa$,  then there exists
 a set  $\cG\subset [t -  t N^{-\fa} ,t ]$ 
of ``good times"  with Lebesgue measure
\be\label{goodtime}
   |[t-t N^{-\fa}, t]\setminus \cG|\le   ( t N^{  -  \fa})^{1/4} N^{-3/4} ,
\ee
and a set 
$\cal R_{z, t}$ in the probability space with
 \be\label{R}
\P (\cal R_{z, t}) \ge 1 -  N^{-\rho} 
\ee 
such that in the set  $  \cal R_{z, t}$ the following  oscillation estimate holds: for any time $s \in \cG $ and  indices $j, k \in  \Xi_z(\ell_1)$  we have
\be \label{holder1}
     \big| v_j(s) - v_k(s)\big|  \le   N^{-\fq \fa }   \|\bv\|_\infty.
\ee
Here the exponent $\fq$ is a positive constant independent
of any parameters.
\end{lemma}

If \eqref{holder1} holds, 
we say that H\"older regularity holds at the space time point  $(z,t)$.\\

For any  $t_1\ll \ell_1  \ll t$ and $ N^{-2/3}\ll  \ell \ll \ell_1$,  denote  by 
\be
A_{\ell, \ell_1} :=  \{ (x, y): |x-y| \le \ell; \;   |z-x| \le \ell_1  \text{ and }       
|z-y| \le \ell_1  \} 
\ee
and  consider  $\Phi_D$ from \eqref{Phi11} with $D_{\ell, \ell_1}:= A_{\ell, \ell_1}\times [t-t_1, t]$. 
With a similar estimate on the boundary terms of the set $A_{\ell, \ell_1}$ as in \eqref{bdr}, one obtains
\be\label{syme}
    \Phi_{ D_{\ell, \ell_1}} = \wt \Phi_{ D_{\ell, \ell_1}} + O\Big( \| \bv\|_\infty \ell^{-2} N^{-2/3} t_1\Big),
\ee
where  
\be\label{tilphi}
\wt \Phi_{ D_{\ell, \ell_1}}: =  \frac{1}{2} \int_{t -{t_1} }^t \int   \rd \varrho(x) \rd \varrho(y)   {\bf 1}_{A_{\ell, \ell_1}}(x, y)  [p_{t-s}(z, x)  - p_{t-s}(z, y)  ]  \sum_{j, k}  \xi_j(x)   \xi_k(y)
N B_{kj}(s) (v_k(s)- v_j(s) )   \rd s.
\ee  \nc
Notice that the characteristic function on $x, y$ puts a constraint on the indices $j, k$
via the support properties of $\xi$'s,  in particular $(j,k)  \in \Xi_z(\ell_1)$. \nc
From Lemma~\ref{holder}  there is a set $\cG$ of "good times" and an event $\cR_{z,t}$ \nc such that 
the H\"older estimate \eqref{holder1} holds  in the intersection of  $\cR_{z,t}$  and
the event $\cal F$ defined in Theorem \ref{thm:appr1}.  
Thus there is a positive constant  $\frak q > 0$ such that for $j, k$  with 
$\xi_j(x)   \xi_k(y) {\bf 1}_{A_{\ell, \ell_1}}(x, y) \not = 0$
 and $t-t_1 \le s \le t$ we have  
\be\label{expused}
  \E   {\bf 1}(\cal G)  {\bf 1}(  \cal  F \cap \cR_{z,t}) 
  |B_{jk}(s)|  \big| v_j(s) - v_k(s)\big|  \le N^{-\frak q \fa }       
   \|\bv\|_\infty \E    \cF   |B_{jk}(s)|.
\ee
Using the  estimate  \eqref{EB}, we have 
\be \label{expused2}
 \E    {\bf 1}(\cal G)  {\bf 1}(  \cal  F \cap \cR_{z,t})    |B_{jk}(s)|  \big| v_j(s) - v_k(s)\big| \le  N^{\rho }
   \frac{1}{N |\gamma_j-\gamma_k|^{2}  } 
 N^{-\frak q \fa } \|\bv\|_\infty.
\ee
We can use
\be\label{regull}
    \frac{1}{|\gamma_j-\gamma_k|^{2}} \le \frac{C}{|x-y|^2 + N^{-2}},
\ee
whenever $\xi_j(x)   \xi_k(y)\ne 0$ and $j\ne k$. \nc
By splitting the time integration into good and bad times,
 we can bound the expectation 
of $\wt\Phi_{ D_{\ell, \ell_1}}$ by   
\begin{align} \nonumber
\E  {\bf 1}(  \cal  F \cap \cR_{z,t})  \wt\Phi_{ D_{\ell, \ell_1}} & \le  N^{\rho } N^{-\frak q \fa } \|\bv\|_\infty 
  \int_{t-t_1}^ {t}\rd s  \int  \rd \varrho(x)  \rd\varrho(y) {\bf 1}_{A_{\ell, \ell_1}}(x, y)     \frac{|p_{t-s}(z, x)    - p_{t-s}(z, y)| }{ |x- y|^{2} +N^{-2} }
+   N^{-\fa+\rho}\|\bv\|_\infty 
\\ \label{err3}
& \le  N^{\rho }  N^{-\frak q \fa }  \log N \|\bv\|_\infty,
\end{align}
 where the second term comes from the ``bad'' times $s$ after using the estimate \eqref{goodtime}
and estimating $|v_k(s)-v_j(s)|\le 2\|\bv \|_\infty$ in \eqref{tilphi}.

\bigskip

\noindent 
{\it Step 4: Time region near the final time $t$  via  the  decay of  the kernel $p_{t-s}(x, y)$. }
We now consider the contribution from the region 
\be
\tilde A_{\ell, \ell_1} :=  \{ (x, y): |x-y| \le \ell;  |z-x| >  \ell_1   \text { or }    \quad
|z-y| >  \ell_1    \}, \qquad \wt D_{\ell, \ell_1}:= \tilde A_{\ell, \ell_1}\times [t-t_1, t] ,
\ee
 i.e. estimate $\Phi_{ \wt D_{\ell, \ell_1}}$, see \eqref{Phi2}. As in \eqref{syme}, it is sufficient
to consider the more symmetrized version 
$$
\wt\Phi_{ \wt D_{\ell, \ell_1}}:=  \frac{1}{2} \nc \int_{t -{t_1} }^t \int   \rd \varrho(x) \rd \varrho(y)   {\bf 1}_{\wt A_{\ell, \ell_1}}(x, y)  [p_{t-s}(z, x)  - p_{t-s}(z, y)  ]  \sum_{j, k}  \xi_j(x)   \xi_k(y)
N B_{kj}(s) (v_k(s)- v_j(s) )   \rd s
$$
with a common factor $\xi_j(x)\xi_k(y)$. 
 Using $\ell \ll \ell_1$,  we see that both $|z-x|$ and $|z-y|$ are bounded from below by $\ell_1/2$, so
the  $p_{t-s}$ kernels are not singular. 
By \eqref{EB}   and  $\|\bv (t) \|_\infty \le \|\bv\|_\infty$, we have 
\be
    \E \cF  |B_{jk}(s)|  \big| v_j(s) - v_k(s)\big|  \le   N^{\rho}     \frac{1}{ N |\gamma_j-\gamma_k|^{ 2}  }  
       \|\bv\|_\infty.
\ee
 Using \eqref{regull},    we can thus bound the expectation of  $\wt\Phi_{\tilde D_{\ell, \ell_1}}$ by 
\begin{align} \nonumber
\E \cF  \wt\Phi_{\wt D_{\ell, \ell_1}}  & \le
 N^{\rho  }   \|\bv\|_\infty  \int_{t-t_1}^ {t} \rd s \int \rd \varrho(x)  \rd\varrho(y)   {\bf 1} _{\wt A_{\ell, \ell_1} }(x, y)
 |p_{t-s}(z, x)    - p_{t-s}(z, y)  |   \frac{ 1 }{ |x- y|^{2} + N^{-2}   } 
\\ \label{err4}
& \le N^{\rho }  \frac { t_1^2} { \ell_1^2}   \|\bv\|_\infty.
\end{align}

\noindent{\it Step 5: The conclusion.}  
Collecting all error terms   from \eqref{bdr}, \eqref{err1}, \eqref{err2}, \eqref{err3}, \eqref{err4},
  and neglecting irrelevant logarithmic factors,  we have 
\be
 \E   {\bf 1}(  \cal  F \cap \cR_{z,t}) |\Phi-\Omega| \le \Big [
\frac { t } {\ell^2}  N^{-1 + \xi} + \frac {\sqrt  \ell} { t_1} 
+ N^{\rho }N^{-\fq \fa } + N^{\rho }  \frac { t_1^2} { \ell_1^2} \Big ]   \|\bv\|_\infty.
\ee
 Recall  the choices $\ell_1= tN^{-\fa}$, $t_1= tN^{-2\fa}$, 
we have 
\be
 \E   {\bf 1}(  \cal  F \cap \cR_{z,t})   |\Phi-\Omega| \le  N^{\xi+\rho}
\Big [ \frac { t } {\ell^2}  N^{-1}
+ \frac {\sqrt  \ell} { t} N^{2\fa}  + N^{-\fq\fa} \Big ]   \|\bv\|_\infty.
\ee
Choosing  $\ell = t^2 N^{-5\fa}$  so that the second term is small, we have 
\be
 \E  {\bf 1}(  \cal  F \cap \cR_{z,t}) |\Phi-\Omega| \le  N^{\xi+\rho}
 \Big [ t^{-3}  N^{-1 +   10 \fa} +  N^{- \fa\fq }  \Big ]   \|\bv\|_\infty.
\ee
Hence for $t \ge t_0 = N^{-\tau_0}/2$ with $\tau_0\le 1/4$, one can choose
sufficiently small positive exponents $\xi, \rho, \fa$, so that  
 $\E {\bf 1}(  \cal  F \cap \cR_{z,t}) |\Phi-\Omega| \le N^{-c}   \|\bv\|_\infty$
with some positive $c>0$. 
 We can choose
$c\le \rho$. After a
Markov inequality and using \eqref{R}, we see that $ |\Phi-\Omega| \le N^{-c/2}   \|\bv\|_\infty$
on an event $\cS$ with probability larger than $1-N^{-c/2}$  
 This completes the proof of Theorem \ref{thm:appr1}. \qed

\section{Proof of the universality at fixed energy}

In this section, we prove our main result Theorem \ref{thm:main}. The key ingredient of the
 proof is Lemma~\ref{main lem} below, asserting that local eigenvalue statistics of DBM for sufficiently  large   but still of order $\oo(1)$
times  converges to 
those of GOE. 
In order to state this lemma, we first introduce some notations.

The trajectory $(\bx(t))_{t\geq 0}$will always denote Dyson Brownian motion dynamics, on the simplex $x_1(t)\leq \dots\leq
x_N(t)$, with initial condition given by eigenvalues of a generalized Wigner. See (\ref{DBMdynamics}). The processes $(\by(t))_{t\geq 0}$, $(\bz(t))_{t\geq 0}$ follow the same dynamics on the simplex, with different, independent, initial conditions, given by the spectrum of a GOE. Omission of the time parameter means initial condition: $x_k=x_k(0)$.

For any $k\in \N$ and any smooth function  $O: \R^k\to \R$, we 
denote the  $W^{2,\infty}(\R^k)$ Sobolev norm by 
\be\label{Sobnorm}
\| O \|_{W^{2,\infty}}= \sum_{\sum_j \alpha_j \le 2}  \Big \| \prod_{j=1}^k  \partial_{x_j}^{\alpha_j} O(\bx) \Big \|_\infty.
\ee
 We will consider test functions $O\in W^{2,\infty}(\R^k)$ that are compactly supported
in $[-L,L]^k$ for some $L>0$.

 For an initial Wigner matrix $H_0$ we define the Ornstein-Uhlenbeck matrix flow 
as the solution of the SDE
\be\label{def Ht}
  \rd H_t = \frac{\rd B_t}{\sqrt{N}}  - \frac{H_t}{2} \rd t, \qquad H_{t=0}=H_0,
\ee
where $B_t$ is a matrix of standard real or complex Brownian motions in the same symmetry class
as $H_0$.  The distribution of $H_t$ coincides with 
\be\label{def Ht1}
H_t \overset{d}\sim e^{-t/2} H_0 +(1-e^{-t})^{1/2} H^G, 
\ee
where $H^G$ is a standard GOE matrix, independent of $H_0$. 
Recall the well-known fact that the law of the solution $\bx(t)$  to the DBM \eqref{DBMdynamics} 
is the same as that of the eigenvalues of  $H_t$ provided that the law of the initial data for \eqref{DBMdynamics} 
is given by the eigenvalues of  $H_0$. 
Recall the definition  
of the correlation functions $\rho_k^{(N)}$  
from Section~\ref{sec:main}   and  define the  rescaled correlation functions
 around a fixed energy $E$ 
by
\be\label{rescro}
 \rho_{k,E}^{(N, \, \rm{resc})}(\bv) := \frac{1}{\varrho(E)^k}
\rho_{k}^{(N)}\left(E+\frac{\bv}{N\varrho(E)}\right).
\ee
We will use $  \rho_{k,E,t}^{(N ,\, \rm{resc})}(\bv)$ for the rescaled correlation functions 
of the eigenvalues of $H_t$.

 \begin{lemma}\label{main lem}  
  For a fixed $k\in \N$ and $L>0$, let $O\in W^{2,\infty}(\R^k)$ be a test function supported in  $[-L,L]^k$. 
 Suppose that $H_0$ satisfies all the assumptions in Definition \ref{def:wig}
and \eqref{eqn:forHansonWright}. 
For a  fixed positive number $\tau$ we set
$t=N^{-\tau}.$
Fix any $\kappa>0$. Then for 
 any   $|E|\le 2-\kappa$ 
we have  
\be\label{co}
\left| \int \rd\bv O(\bv)  \rho_{k,E,t}^{(N, \, \rm{resc})}(\bv)
-
\int\rd\bv O(\bv)  \rho_{k}^{(\text{\rm\tiny GOE})}( \bv) \right|\le C  \| O \|_{W^{2,\infty}}\tau^{1/2}
\ee
holds for  any small enough $\tau\le \tau_0(\kappa)$
 any sufficiently large  $N\ge N_0(\tau, \kappa)$. Here the constant $C$  depends only on $L$ and $\kappa$.  
 \end{lemma}

Throughout this section we use the relation $t=N^{-\tau}$ between $t$ and $\tau$, and
we will use both letters in parallel. 
In order to extend the universality result from 
 Wigner ensembles $H_t$ with a Gaussian component of size of order
 $t=N^{-\tau}$   to all Wigner ensembles, 
we follow  the standard approach via the following  Green function comparison theorem. 

\begin{lemma}
\label{weak lemma1} 
Consider
two  $N\times N$ generalized  Wigner matrices, 
$H^{(v)}$ 
and $H^{(w)}$  with matrix elements  $h_{ij}$
given by the random variables $N^{-1/2} v_{ij}$ and 
$N^{-1/2} w_{ij}$, respectively, and 
satisfying  the assumptions in Definition \ref{def:wig} and the moment condition
 \eqref{eqn:forHansonWright}. 
We assume that the first four moments of
  $v_{ij}$ and $w_{ij}$ satisfy, for some $\delta > 0$, that  
\be\label{4}
    \big | \E  (\Re \,  v_{ij})^a (\Im \, v_{ij})^b  -
  \E  (\Re \,  w_{ij})^a (\Im \, w_{ij})^b  \big | \le N^{-\delta -2+ (a+b)/2},
  \qquad 1\le a+b\le 4.
\ee
Let $\rho_{k}^{(N, v)}$ and $\rho_{k}^{(N, w)}$
be the  $k-$point correlation functions of the eigenvalues w.r.t. the probability law of the matrix $H^{(v)}$
and $H^{(w)}$, respectively. 
Then for any test function $O$ and any $|E|\le 2-\kappa$   we have   
\be \label{6.3} 
\lim_{N\to\infty} \int \rd\bv O(\bv)\left( \rho_{k,E}^{(N, \, v, \, \rm{resc})}(\bv)
 -\rho_{k,E}^{(N, \, w, \, \rm{resc})}(\bv) \right) =0.
\ee
 \end{lemma}
  
\begin{proof}  Recall \cite[Lemma 3.4]{EYYBernoulli}, where it was  proved  that for any  real random variable  $\theta$   such that 
$$
\E\,\theta=0, \quad \E\, \theta^2=1, \quad \E |\theta|^4\le C,
$$
and small $t>0$, there exists random variable $\wt\theta=\wt \theta (\theta, t)$
and an independent, standard normal random variable $X\sim  \cal N(0,1)$  such that 
\begin{enumerate}
\item $\wt \theta$ has subexponential decay; 
\item the first three moments of 
$e^{-t/2}\wt \theta+(1-e^{-t})^{1/2} X$ equal to those of $\theta$;
\item  the difference between the fourth moment of $e^{-t/2}\wt \theta+(1-e^{-t})^{1/2}X$ and $\theta$ is $\OO(t)$. 
\end{enumerate}
 Inspecting the proof in \cite{EYYBernoulli}, one can easily show that (i) can be strengthened to
require that $\wt\theta$ has a Gaussian decay. Moreover,
one  can easily extend this result to complex random variables $\theta$ 
 if  (1)  $\Re(\theta), \Im(\theta)$ are independent, or (2) the law of $\theta$ is isotropic,  i.e. $|\theta|$ is independent of $\arg \theta$, which is uniform on $(0,2\pi)$.  In this case   there exists a complex random variable $\wt \theta$ satisfying the corresponding condition (1) or (2) and each item (i)-(iii).

We apply this result to  each entry of $H$. Therefore, there exists a generalized Wigner matrix $\wt H$,
   satisfying the assumptions   in Definition \ref{def:wig}   and \eqref{eqn:forHansonWright} 
such that if we define 
   $$
   \wt H_t  \; = \; e^{-t/2} \wt H  +(1-e^{-t})^{1/2} H^{G},
   $$
   then    the first four moments of the matrix entries of $\wt H_t$ almost match those of $H$ in the following sense:  
    $$
\E\, \big[ \Re(\wt H_t)_{ij}\big] ^a \big[\Im (\wt H_t)_{ij}\big]^b=\E \, \big[\Re  H_{ij}\big]^a
\big[ \Im  H_{ij}\big]^b, \quad \quad 0\le a,b\le 3, \quad  1\le a+b\le 3,
$$
$$
\left|\E\, \big[ \Re(\wt H_t)_{ij}\big] ^a \big[\Im (\wt H_t)_{ij}\big]^b-\E \, \big[\Re  H_{ij}\big]^a
\big[ \Im  H_{ij}\big]^b\right|
\le CN^{-2} t, \quad \quad a+b=4.
$$
 Furthermore, $\wt H_t$ satisfies the assumptions in Definition \ref{def:wig} and
the decay condition \eqref{eqn:forHansonWright}. Applying Lemma \ref{weak lemma1} 
with the choice $ H^{(v)}=H$ and $ H^{(w)}=\wt H_t$,  $t= N^{-\tau}$    and $\delta:=\tau$,  we obtain that  
 the correlation functions of $H$ asymptotically match those of $\wt H_t$, i.e.,    

\be\label{ssaid-2}
\lim_{N\to \infty} 
\int \rd\bv O(\bv)  \left(\rho_{k,E}^{(N, \, \rm{resc})}(\bv )-\wt \rho_{k,E,t}^{(N, \, \rm{resc})}(\bv)\right)
=0
\ee
 for any test function $O$.   
 Now we can apply  \eqref{co}  with $\wt H$ and $\wt H_t$  playing the role of $H_0$ and $H_t$, respectively, 
 since $\wt H$ satisfies the assumption in Definition \ref{def:wig} and \eqref{eqn:forHansonWright}.
 We obtain that the  correlation functions of $\wt H_t$  asymptotically match  those of $ H^G$: 
 \be\label{ssaid-3}
 \limsup_{N\to \infty} 
\left| \int \rd\bv O(\bv) \wt  \rho_{k,E,t}^{(N, \, \rm{resc})}(\bv)
-
\int\rd\bv O(\bv)  \rho_{k}^{(\text{\rm\tiny GOE})}(\bv) \right|\le C  \| O \|_{W^{2,\infty}}\tau^{1/2}.
\ee 
Combining \eqref{ssaid-2} and \eqref{ssaid-3}, and letting $\tau\to 0$
after the $N\to\infty$ limit,  we obtain  that
\be\label{eq:mai}
\lim_{N\to \infty} 
 \int \rd\bv O(\bv) \left( \rho_{k,E}^{(N, \, \rm{resc})}(\bv )
- \rho_{k}^{(\text{\rm\tiny GOE})}( \bv)  \right)=0
\ee 
holds for compactly supported
test functions $O\in W^{2,\infty}$. To extend this result to
a general continuous 
 function $O$ supported in $[-L, L]^k$, we use a simple approximation.  For  any $\e>0$, there exist $W^{2,\infty}$  
functions  $O_{+,\e}$ and   $O_{-,\e}$, supported in $[-L-\e, L+\e]^k$, such that 
  $$
 O_{+,\e}\ge O\ge O_{-,\e}, \quad \| O_{+,\e}-O_{-,\e}\|_\infty\le 2\e.
  $$
 Applying \eqref{eq:mai}  to  $O_{\pm,\e}$, we obtain
 $$
  \limsup_{N\to\infty} \int \rd\bv O(\bv)\rho_{k,E}^{(N, \, \rm{resc})}(\bv)
\le 
\int\rd\bv O_{+,\e}(\bv)\rho_{k}^{(\text{\rm\tiny GOE})}\left(\bv\right).
$$  and similar lower bound for $\liminf$. 
 Together with the fact that $ \rho_{k}^{(\text{\rm\tiny GOE})} $ is bounded, 
   it implies that \eqref{eq:mai} holds for any continuous, compactly supported observable, 
which completes the proof of Theorem \ref{thm:main}.
\end{proof}
 
\subsection{Reduction to observables with compact Fourier support. } 
 This section presents an approximation argument: we show that universality for a special class of test functions 
can be extended to $W^{2,k}$ test functions as required in  Lemma \ref{main lem}. After a change of
variables, we will work with test functions
that have a compact support in the Fourier space in the energy variable. Universality for such test functions
is stated in Lemma~\ref{lem: flatO} below and  will be proven in the
subsequent Section~\ref{sec:pftech}. 

First we will need the following  precise estimates on the correlation functions of GOE,  
 which were proved in \cite[Theorem 3]{Shch10} and \cite{DeiGio2009}. 

\begin{lemma}\label{lem: Gaussian} 
 As in \eqref{GOEconv},   
 \be\label{GUEsine}
\begin{split}
\frac{1}{\varrho(E)^k}\rho_{k}^{(N,\text{\rm\tiny GOE})}\left(E+\frac{\bv}{N\varrho(E)}\right)=  \rho_{k}^{(\text{\rm\tiny GOE})}\left(\bv\right)+O(N^{-1/2}), \quad 
 \end{split}
\ee
uniformly holds for $(\bv, E)$ in any fixed compact subset of $\R^k\times (-2,2)$ (for matrices from the GUE, the same statement holds with a different limit
$\rho_{k}^{(\text{\rm\tiny GUE})}\left(\bv\right)$).
 \end{lemma}

Let $\bx^G$ be  the vector of ordered  eigenvalues of $H^G$ and let $\bx(t)$  be the eigenvalues of 
  $H_t$ in \eqref{def Ht}. 
Simply rescaling the variables in $O$ with $\varrho(E)$,  
the above lemma shows that  for the proof of
 \eqref{co} it is sufficient to prove that  
\be\label{ssaid}
  \left|
 \E  \, {\sum_{i_1, i_2\cdots i_k=1}^N} O \left(\big\{N( x_{i_j}(t) -E)\big\}_{j=1}^k\right)
 -
  \E   \, {\sum_{i_1, i_2\cdots i_k=1}^N} O \left(\big\{N(x^G_{i_j  } -E)\big\}_{j=1}^k\right)
 \right|
\le C\tau^{1/2} 
\ee 
holds for any compactly supported $O\in W^{2,\infty}(\R^k)$.

For brevity, 
 we assume that $O$  has only two arguments, 
i.e., $k=2$;  the general case is proven analogously. 
Furthermore, with a change of variables   $(a,b)\to (a, b-a)$, we 
use the test function of the form 
\be\label{change}
Q(N(x_i-E), N(x_j-x_i)) \quad \mbox{instead of} \quad O(N(x_i-E), N(x_j-E)).
\ee 
The new test function $Q$ is still  compactly supported and lies in $W^{2, \infty}$;  
its advantage is that it depends on $E$ only through its first variable.  
  Therefore, under the assumption of Lemma \ref{main lem}, it is sufficient to prove that for small enough $\tau$,  
   \be\label{ssaidO2}
 \left|
 \E \, {\sum_{i, j=1}^N} Q \Big( N(x_{i}(t)-E), N(x_{j }(t)-x_{i}(t))\Big)
 -
  \E  \, {\sum_{i, j=1}^N} Q \Big(N(x^G_{i }-E), N(x^G_{j }-x^G_{i })\Big)
 \right|
\le C\|Q\|_{W^{2,\infty}}\tau^{1/2} 
\ee 
holds for sufficiently large  $N$ and with  $C$  depending only on $L$ and $\kappa$. 
For simplicity,  we   define 
\be\label{lfei}
  \mathcal Q(\bx, E) := \sum_{i,j=1}^NQ \Big( N(x_{i}-E), N(x_{j}-x_{i})\Big), \quad  E\in \R. 
\ee
 Let $\wh Q(p, y)$  be the Fourier transform of $Q$ w.r.t. the first argument, i.e.,  
\be\label{defOK}
\wh Q(p, y)=\int_\R Q(x, y)e^{-  i p x  }\rd x .
\ee
 In this section, hat always denotes 
 a partial Fourier transform, i.e. Fourier transform only in the first variable.
 The Fourier-space variables will be denoted by $p$.
 We will also say that $\wh Q(p,y)\in W^{2,\infty} $ if $\wh Q$ as a function of $p, y$ is in the Sobolev space.

 The following lemma, proven in Section~\ref{sec:pftech},   states that 
 that universality for large time holds for observables whose Fourier transforms 
have compact support.

\begin{lemma}\label{lem: flatO}
Let $Q: \R^2\to \R$ be a   function such that  $\wh Q\in W^{2,\infty} $, and  
\be\label{iwin}
\supp \,\wh Q\subset[-m,m\,]\times [-L,L]
\ee
for some fixed $m$, $L\in \mathbb N$.
There exists a constant $\delta_0 $ independent of $m$   and $L$ 
  such that 
for any   $  \kappa>0$
\be\label{bus0-1}
   \E \, \mathcal Q(\bx(t), E) -  \E \, \mathcal Q(\bx^G, E) = {\rm O}(N^{-\tau}) , \qquad t=N^{-\tau}
\ee
holds  uniformly for $ |E|\le 2-\kappa$ and $\tau \le \frac{\delta_0}{m^2+1}$. 
  \end{lemma}

We now  prove Lemma \ref{main lem}  assuming Lemma \ref{lem: flatO} holds. 
The first step is to  approximate a  compactly supported 
observable  $Q(x, y) \in W^{2,\infty}$  by an observable $Q(x, y)$ whose Fourier transform    $\wh Q(p, y)$ 
is compactly supported as required in Lemma \ref{lem: flatO}. 
The following lemma provides an effective control on this approximation.

\begin{lemma}\label{lemmaOO} Let  
  $q\in W^{\, 2 , \infty}(\R) $ be a symmetric cutoff function, supported on $[-1, 1]$ such that $q(p)=1$ for $|p|\le 1/2$,  and $q\,'(p)\le0$ for $p>0$.  For $Q\in W^{2,\infty}$ and $\supp Q\in [-L,L]^2$,   define $Q_m$  via
its partial Fourier transform as 
\be\label{dkps}
\wh {Q_m}(p, y):=\wh {Q }(p, y)  q_m (p ),     \quad \quad q_m(p ):= q(p /m).
\ee
Then  there exists a constant $C$, depending only on $L$ and $q$,  such that  for   any $m\in \N$, $(x, y) \in \R^2$,  we have 
\begin{align}
&|   (Q_m-Q) (x, y )|\le  \frac{ C\|Q\|_{W^{ 2, \infty}}}{(1+x ^2)  } m^{-1},\label{rO1}\\
&\|\wh {Q_m}\|_{W^{2, \infty}} \le C \|Q \|_{W^{2, \infty}}. \label{rO2}
\end{align}
\end{lemma}
 
\begin{proof}[Proof of Lemma \ref{lemmaOO}] We have
\begin{align}
\label{FFor1}
&x^n  (Q_m-Q) (x, y)
= \int_{  \bR }   \frac{(i \partial_{p})^n}{2\pi}  \left[ \left(\wh {Q_m}-\wh Q \right) (p, y)\right]e^{  ipx  }\rd p,\\
\label{FFor2}
&p^{n_1} (\partial_{p})^{ n_2} \wh Q  (p, y)= \int_{  \bR } (-i)^{n_1+n_2} (\partial_{x})^{n_1}  \left[  x ^{n_2}Q  (x, y)\right]e^{- ipx  }\rd x.
\end{align}
Using \eqref{FFor2} with $n_1=0,2$,   since $Q$ is compactly supported in $[-L,L]^2$,  we have 
\be\label{wltt}|(\partial_{p})^n \wh Q (p, y)|\le C_{n,L}(1+|p|^2)^{-1} \|Q\|_{W^{2,\infty}}. 
\ee
  Similarly, using \eqref{FFor1} we have  
\[
    \left| x^n  (Q_m-Q) (x, y)\right| 
\le C_{n }  \int  
\sum_{n_1+n_2 =n} \left|  (\partial_{p})^{n_1}(q_m-1)(p)\right|
\left|(\partial_{p})^{n_2} \wh Q   (p, y)\right|
\rd p.
\]
 By definition, if $n_1\neq 0$ then $\left|(\partial_{p})^{n_1}  (q_m-1)\right|  \le C m^{-1}$. If $n_1= 0$ then $\left|(\partial_{p})^{n_1}  (q_m-1)\right| $    is supported in  $\{p: |p| \ge m/2\}$. Together with \eqref{wltt},   we have  
$$\left| x^n  (Q_m-Q) (x, y)\right| \le C_{n,  L} \|Q\|_{W^{2,\infty}} \,\,m^{-1}.
$$ 
Choosing $n=0 $ and $2$, we complete the proof of  \eqref{rO1}.  The  \eqref{rO2} can be easily derived 
from the definition of $Q_m$ and \eqref{FFor2} (with $n_1=0$).  
\end{proof}

Lemma \ref{lemmaOO} provides an approximation for any smooth observables with compact support
by observables with compact support in the Fourier space. On the  other hand,  to estimate the error resulting from 
this approximation, we will need the following corollary which gives an effective  bound on the density of $\bx(t)$, the
 eigenvalues of $H_t$, at the local scale $1/N$.

 \begin{corollary}\label{cor: mon} 
Define
\be\label{def HEcc}
  \#(\bx, E, s_1,s_2) :=\, \left|\left\{(i, j)\in \mathbb N^2 : |x_{i}-E  |\le 
 \frac {s_1}N, \quad |x_{i}-x_{j} |\le  \frac {s_2} N\right\}\right|.
 \ee
For any fixed $L\ge 1$ and $\kappa>0$, with the $\delta_0$ in Lemma \ref{lem: flatO},  there exists constant $C>0$ such that for  $\tau\le \delta_0/2$ 
 \be\label{figC}
 \limsup_{N\to \infty} \max_{|E|\le 2-\kappa}\E \, \#(\bx(t), E, 1, L)\le C, \quad\quad\quad t=N^{-\tau}, \quad \tau\le \delta_0/2.
 \ee 
 \end{corollary}
\begin{proof} Let $ g, h \in W^{2,\infty}(\R)$ be two real 
functions such that $ \supp \, (\,\wh  g\,)=[-1,1]$, $\supp (h)=[-2L,2L]$.
 We assume that $\min_{|x| \le a} g(x)\ge b$  for some $0<a\le 1$ and $b>0$.  Furthermore, we assume that $1\ge h(x)\ge0$ for any $x\in \mathbb R$ and $h(x)=1$ for $|x|\le L$.   Define 
\be\label{volO}
Q(x, y):=    b^{-2} g^2(x) h(y).
\ee
 Since $\wh{g^2} = \wh g\star \wh g$, it
 is clear that  $Q$ satisfies the assumption in Lemma \ref{lem: flatO} with $m=2$, 
 and $ L$ being replaced by  $2L$.  Then applying \eqref{bus0-1} to   $Q$ 
defined in \eqref{volO}, and using that $\E \, \mathcal Q(\bx^G, E)$ is
bounded   from Lemma \ref{lem: Gaussian}, we have that 
\be\label{heyf}
\limsup_{N\to \infty}\E \, \mathcal Q(\bx(t), E)\le C 
\ee holds uniformly  for any $ |E|\le 2-\kappa$. 
 From the definition of $Q$ in \eqref{volO}, we have   
$$
\min_{x,y\in\R}Q (x,y) \ge 0 \quad {\rm and } \quad  Q(x, y )\ge 1, \quad (x,y)\in [-a,a]\times [L, L]. 
$$
Then \eqref{heyf}  implies   $$\limsup_{N\to \infty}\max_{|E|\le 2-\kappa}\E\,    \#(\bx(t), E, a, L)\le C$$ for some constant $C$. 
  Hence \eqref{figC} also holds, since $a\sim 1$,  which completes the proof of  Corollary \ref{cor: mon}. 
\end{proof}

We now have all the ingredients to complete   the proof of Lemma \ref{main lem}. 

\begin{proof}[Proof of Lemma \ref{main lem}]
For any compactly supported $Q\in W^{2,\infty}(\R^2)$,
 we   construct $Q_m$ as in \eqref{dkps}. The definition of $Q_m$,   and \eqref{rO2} guarantee that $Q_m$ satisfy the assumption of Lemma \ref{lem: flatO}.  Then applying Lemma \ref{lem: flatO} for $Q_m$,
  we obtain  for any fixed $m\in N$ that
\be\label{bus1}
    \E \, \mathcal Q_m(\bx(t), E) =  \E \, \mathcal Q_m(\bx ^G, E) + {\rm O}(  N^{-\tau}),  \qquad
t=N^{-\tau}, \qquad  \tau \le \frac{\delta_0}{m^2+1},
\ee
where $\delta_0$ is from  \eqref{bus0-1}. 
On the other hand,  we will show below 
 that  if  $\tau\le \delta_0/2$,  then 
\be\label{bus2}
\left| \E \, \mathcal Q (\bx(t), E)
-
  \E \, \mathcal Q_m(\bx(t), E)\right|\le C\,\|  Q\|_{W^{2, \infty}}  m^{-1} 
\ee
holds for some $C$ independent of $m$ and  $\tau$ and 
 and large enough $N$. Notice that \eqref{bus2}
  also holds if we replace the $\bx(t)$ with $\bx^G$, since $\bx (t) \overset{d}{\sim }  \bx^G$ if $\bx(0) \overset{d}{\sim }  \bx^G$.  Combining \eqref{bus1} and \eqref{bus2}, and choosing $m=\delta\tau^{-1/2}$ with a small $\delta\le \delta_0$, 
  we obtain \eqref{ssaidO2}, i.e., \eqref{ssaid} in the case $k=2$. One can easily extend the above proof to the general $k$ 
case. Together with Lemma \ref{lem: Gaussian}, it implies the desired result \eqref{co}.

  Hence it only remains to prove \eqref{bus2}.    Using   \eqref{rO1}, we have 
 $$
 \max_y \max_{x:|x-n|\le 1}\left|(Q_m-Q)(x, y)\right|\le \frac{C\|Q\|_{W^{2,\infty} }}{1+n^2} m^{-1}.
$$
With the definition of $\#(\bx(t), E+\frac n N, 1, L)$ in \eqref{def HEcc}, it implies that 
\be\label{yin30}
\left| \E \, \mathcal Q (\bx(t), E)
-
  \E \, \mathcal Q_m(\bx(t), E)\right|
\le  \|Q\|_{W^{2,\infty}}  \frac{1}{m} \sum_n \frac{C}{1+n^2}\E\,  \#(\bx(t), E+\frac n N, 1, L).
\ee
For $n\le N^{1/2}$, $ \E\, \#(\bx(t), E+\frac n N, 1, L)$ can be bounded by \eqref{figC}
 (after replacing $\kappa$ with $\kappa/2$).  For  $n\ge N^{1/2}$ we can use the trivial bound
$$ \E\, \#(\bx(t), E+\frac n N, 1, L)\le N^{\xi}
$$
for any $\xi>0$
that directly follows from the  rigidity of eigenvalues
 of $H_t$ if $N\ge N_0(\xi)$ is large enough. 
 Inserting these bounds into \eqref{yin30}, we obtain \eqref{bus2}  and complete the proof of Lemma \ref{main lem}. 
\end{proof}

\subsection{Universality for test functions with compact Fourier support: Proof of Lemma \ref{lem: flatO}. }\label{sec:pftech}
 
Recall that $E$ satisfies $|E|\le 2-\kappa$. 
All the constants in the following proof depend  on $\kappa$, but we will not carry this dependence 
explicitly in the notation.  For any nonnegative integer $\al$ introduce the notation  
$$
Q^{(\al)}(x, y):= (\partial_{x})^\al Q(x,y).
$$
 Recall   
$$
x^{n_1} (\partial_{x})^{n_2}Q (x, y)
= \int  \frac{  i^{n_1+n_2}}{2\pi} (\partial_{p})^{n_1} \left[p^{n_2}\wh Q  (p, y)\right]e^{  i px }\rd p.
$$
 As in \eqref{rO1}, with assumption \eqref{iwin} and $\wh Q\in W^{2,\infty} $  we obtain that there exists  some constant $C$ such that for any  $\al\in \mathbb Z_{\ge 0}$, and $y\in\R$
 \be\label{teyz}
 \left| Q^{(\al)} (x, y) \right|\le C d_{\al }(1+x^2)^{-1} , \quad\quad  d_\al:=  \left(m^2+\al^2\right)  m ^{\al -1} \|\wh Q\|_{W^{2, \infty}}.  \quad  \ee
The multiindex $\al$ used in this section has nothing to do with the 
threshold $\al$ to indicate indices away from the edge, see e.g., \eqref{pupp}.\nc

 The main input to prove Lemma \ref{lem: flatO} is
 the homogenization result,  Theorem 
 \ref{T1}, stating that for any  $\tau<\tau_0$ with a sufficiently small  $\tau_0$,  two coupled 
  DBMs \eqref{DBMdynamics} driven by the same Brownian motions satisfy the estimate 
\be\label{520}
Nx_i(t)-N  \Big (y_i(t)  +  (\Psi_{t-t_0} \b x (t_0))_i 
-  (\Psi_{t-t_0} \b y (t_0))_i\Big) = O(N^{-\delta_2}), \qquad t=N^{-\tau}\ge 2t_0 = N^{-\tau_0},
\ee
for all $i\in I (\delta_1)$ with probability bigger than $1-N^{-\delta_3}$. 
 We recall from \eqref{Psidef} that    $\Psi_t \b x$ is given by 
$
(\Psi _t \bx)_i=N^{-1}\sum _k p_t(\gamma_{i}, \gamma_k ) x_k,
$
and
$\delta_1, \delta_2, \delta_3$ are small positive exponents.

 In our application  we choose $\by( t_0 )$  to be distributed by  $\mu_G$, i.e, 
the eigenvalue distribution of  a Gaussian matrix ensemble. 
Since $\bx( t_0)$'s are  the eigenvalues of a generalized Wigner matrix, 
we denote their  distribution by  $\mu_W$.  The joint distribution of the coupled DBM processes
$ \{\bx (s)\}_{0\le s\le t}$ and $ \{\by (s)\}_{0\le s\le t}$,  
as defined in \eqref{DBMdynamics}, is given by 
$\mu_W\otimes\mu_G\otimes \mu_ B$, where 
$\mu_B=\mu \left(\{ B_\ell (t)\}_{1\leq \ell\leq N, 0\le s\le t}\right)$
is the measure of the independent Brownian motions.
  For simplicity, for expectation w.r.t. $\mu^W\otimes\mu^G\otimes \mu ^ B$, we just use $\E$. 
 For the expectation of functionals $f$ of $\bx(t)$, we will sometimes use  $\E^{\mu_W}f (\bx(t))$
instead of $\E^{\mu_W\otimes \mu_B}f (\bx(t))$ and 
similarly we use
 $\E^{\mu_G}f (\by(t))$ for functionals of $\by(t)$.

Below, we apply the  homogenization result \eqref{520} to $Q^{(\al)}$. 
Recall the definition of $I(\delta)$ from \eqref{definICE}.

 \begin{lemma}\label{loan}
For $\bx\in \R^N$ and $s\ge 0$,  define  
\be\label{defxi}
  \xi^{\bx }_s: = \xi^{\bx }_s(E):= 
N(\Psi_s {\bx})_{i _0}- N(\Psi_s {\b \gamma})_{i _0}, \quad i_0:= \min \{ i:\; \gamma_{i}\ge E\}.
\ee 
(The notation $\xi^{\bx}_s$ 
should not be confused  with the rigidity exponent $\xi$.)  Recalling $\tau_0$ provided by 
Theorem~\ref{T1},
 there exists $\delta_Q\le \tau_0$  such that for any $0<\delta\le \delta_Q$ and any $0<\tau\le  \delta/5$
we have  (with the usual $t=N^{-\tau}$, $t_0=N^{-\tau_0}/2$ conventions) 
\begin{multline}\label{chqu3}
\max_{|E'-E|\le N^{4\tau-1}} \left|\E^{\mu_W}\; \mathcal Q^{(\al)}(\bx(t), E')- \sum_{ i,j  \in  I(\delta)}\E \; Q^{(\al)}
\Big(N(y_{i}(t)-E')+\xi^{\bx(t_0)}_{ t-t_0} -\xi^{\by(t_0)}_{ t-t_0} ,  \quad  N( y_{j}(t)-y_{i}(t))\Big)\right|\\
\le  3 N^{-\delta/2}d_{\al+1} 
\end{multline}
  for large enough $N\ge N_0$ where $N_0$ is independent  of $\al$, the order of derivatives.  
\end{lemma}

 Notice that although Lemma \ref{lem: flatO} is formulated at a fixed energy $E$, for its proof we will
need to understand $\E\; \mathcal Q^{(\al)}(\bx(t), E')$ for nearby  energies $E'$ as well,
which explains the introduction of $E'$ in \eqref{chqu3}.

\begin{proof}  First we show that the summation over all indices $i,j$ in
the definition of $\mathcal Q$ \eqref{lfei} can be restricted to the interval $I(\delta)$.
 This directly follows from 
the rigidity of eigenvalues $\bx(t)$  and from the bound \eqref{teyz}: there exists some $\delta_{\rm c}>0$ 
 (here we use subscript $c$ for cutoff) such that for  $ 0<\delta \le  \delta_{\rm c}$ and $\tau\le \delta/5$,  we have 
\be\label{chqu}
\max_{|E'-E|\le N^{-1+4\tau}} \left|\mathcal Q^{(\al)} (\bx(t), E')-\sum_{ i,j \in I ( \delta)  }Q^{(\al)} \Big( N(x_{i }(t)-E'), \quad  N(x_{j }(t)-x_{i }(t))\Big)\right|\le   d_\al N^{-\delta/2 }
\ee
holds with probability greater than $1-N^{-10}$ for large enough $N$ independent of $\al$. 
 
With Theorem \ref{T1}, and 
from the derivative estimate  from \eqref{pderi},  
 one can easily check that there exists some constants   $\delta_{h}$  (``$h$" stands for homogenization), and $\tau_0$
 such that  \eqref{zhimao} holds for any  $0<\tau\le  \tau_0 $ and $\delta_1$, $\delta_2$, $\delta_3 \le \delta_h$
and we also have
 \be\label{wezia}
  \left|p_t(\gamma_{i}, \gamma_j)-p_t(\gamma_{i+1}, \gamma_j)\right|\le 
 N^{-3\delta_h}, \quad \quad \forall\;  i \in I (\delta_h) , \quad 1\le  j\le N.
\ee
Using \eqref{wezia} and the rigidity of eigenvalues, we know that  for any $0<\tau< \tau_0$,
 \be\label{juanmao}
\P\left(\max_{i,  j\in I (\delta_h)} \left|(\Psi_{t-t_0} ({\bx}(t_0)-{\b \gamma}))_{i }-
(\Psi_{t-t_0} ({\bx}(t_0)-{\b \gamma}))_{j}\right|\ge N^{-1-\delta_h } \right) \le N^{-10}.
\ee
Define $\mathcal Q^{(\al)} ( \bx, E)$  as in \eqref{lfei} with $Q^{(\al)}$  replacing $Q$. 
 Combining \eqref{zhimao}, \eqref{juanmao} and \eqref{chqu}, and using rigidity and \eqref{teyz}, we obtain that  for  any $\delta$: $0<\delta\le \delta_Q:=\min (\delta_c, \delta_h){ /3}$ and $0<\tau<\min( \delta/5, \tau_0)$,  
 \be\label{chqu2}
\max_{|E'-E|\le N^{4\tau-1}} \left|\mathcal Q^{(\al)} (\bx(t), E')-\sum_{ i,j  \in  I(\delta)}Q^{(\al)} 
\Big(N(y_{i}(t)-E')+\xi^{\bx(t_0)}_{t-t_0}-\xi^{\by(t_0)}_{t-t_0}, \quad  N( y_{j}(t)-y_{i}(t)) \Big)\right|\le   2 N^{-\delta/2 }d_{ \al +1}
\ee
holds with probability larger than $1-2N^{-\delta }$  for large enough $N$ independent of $\al$.   Here for the second variable of $Q^{(\al)}$ we first 
use  \eqref{zhimao}, for any $i, j\in I(\delta)$,
\begin{align*}
   x_{j }(t)-x_{i }(t) = & y_{j}(t)-y_{i}(t) +  (\Psi_{t-t_0}\bx)_{j} - (\Psi_{t-t_0}\bx)_i  -  (\Psi_{t-t_0}\by)_{j} 
 + (\Psi_{t-t_0}\by)_i 
  + O(N^{-1-\delta}) \\
    = &  y_{j}(t)-y_{i}(t) +  (\Psi_{t-t_0}(\bx-{\b\gamma}))_{j} - (\Psi_{t-t_0}(\bx - {\b\gamma}))_i  - (\Psi_{t-t_0}(\by-{\b\gamma}))_{j}  
+ (\Psi_{t-t_0}(\by - {\b\gamma}))_i
  + O(N^{-1-\delta}) \\
  = &  y_{j}(t)-y_{i}(t) + O(N^{-1-\delta}),
\end{align*}
with the shorthand writing $\bx = \bx(t_0)$, $\by=\by(t_0)$, 
where in the second step we smuggled in the $\gamma$'s and in the last step we used \eqref{juanmao}.
Similar argument applies to the first variable of $Q^{(\al)}$. 

  On the complement event of probability at most $2N^{-\delta}$ but still on the event where
the rigidity holds, we use that 
 for any fixed $\tau$, $\delta$ and $\xi>0$, 
 $$
\left| \mathcal Q^{(\al)} ( \bx(t), E') \right|
+
\left|\sum_{ i,j  \in  I(\delta)}Q^{(\al)} 
\Big(N(y_{i}(t)-E')+\xi^{\bx(t_0)}_{t-t_0}-\xi^{\by(t_0)}_{t-t_0}, \quad N( y_{j}(t)-y_{i}(t))\Big)
\right|\le d_{ \al }N^\xi
 $$
 holds for all $E':$ $ |E'-E|\le N^{4\tau-1} $.   Finally, on the event where the rigidity does not
hold, we can estimate $Q^{\al}$ by maximum norm; the contribution of this event is still negligible in
the expectation. 
  Together with \eqref{chqu2}, we obtain  \eqref{chqu3} and  complete the proof of Lemma \ref{loan}. 
\end{proof} 
 
 To understand the second term in \eqref{chqu3}, we define a (non-random) function $F$ as follows:
\be\label{defFA}
F(a):=\E^{\mu_G} \sum_{ i,j \in  I (\delta_Q) }  Q   \Big(N(y_{i}(t)-E)+a-\xi^{\by(t_0)}_{t-t_0}, \quad  N(
 y_{j}(t)-y_{i}(t))\Big),
\ee 
 and we always assume
$
 \delta_Q\le 10^{-4}.
$
 We can now rewrite \eqref{chqu3} as follows: for  $\tau<\min( \delta/5, \tau_0)$
\be\label{chqu31}
\max_{|h|\le N^{4\tau}} \left|\E^{\mu_W} \; \mathcal Q^{(\al)}(\bx(t), E+ h/N )- \E^{\mu_W} \; F^{(\al)}(\xi^{\bx(t_0)}_{t-t_0}
- h)\right|\le  3 N^{-\delta_Q/2}d_{\al+1}. 
\ee
 We do not have a direct understanding of $\xi^{\bx(t_0)}_{t-t_0}$; although it concerns local statistics on
a relatively large mesoscopic scale $t-t_0\sim N^{-\tau}$, but in \eqref{chqu31} we would need it with a precision that cannot be obtained
from the available local semicircle laws for Wigner matrices. The key observation is that $F$ is essentially
a constant function, so the actual distribution of  $\xi^{\bx(t_0)}_{t-t_0}$ does not matter.   
 The following lemma formalizes the statement that $F$ is essentially a constant: 
 \begin{lemma} \label{lem: FF} With   
the choice  $\delta_0:=\min(\delta_Q/3, \tau_0)$ and $\tau < \frac{\delta_0}{m^2+1}$,   
we have 
\be\label{TI}
F (a)-F(0)=O(d_2 N^{-\tau}), \quad \forall \;  a:\;   |a|\le N^{4\tau}.
\ee
\end{lemma}
 We first prove Lemma \ref{lem: flatO}   assuming that 
Lemma \ref{lem: FF} holds and then we will prove Lemma \ref{lem: FF}  in the next Section~\ref{sec:48}. 
 Using  rigidity for $\bx$  and the fact that
$$p_t(\gamma_{i}, \gamma_j) \le \frac{C t}{t^{\,2}+( \gamma_i- \gamma_j)^2  }, \quad  i \in I  (\delta_Q) , \quad 1\le  j\le N,
$$
(see \eqref{pupp}), 
 we obtain  that for any $0<\tau< \tau_0$ and $\xi>0$,   \be\label{xnhao}
\P\left(\max_{ i\in I (\delta_Q)} \left|(\Psi_{t-t_0}\b x)_{i}-(\Psi_{t-t_0}{\b \gamma})_{i }\right|\ge N^{-1+\xi} \right)\le N^{-10}.
\ee
 Choosing $\xi$ small enough in  \eqref{xnhao}, we have 
\be\label{xnhao2}
|\xi^{\bx(t_0)}_{t-t_0}| \le N^\tau.
\ee 
Hence for    $\tau  <  \frac{\delta_0}{m^2+1} $,   and $|E'-E|\le N^{4\tau-1}$ we have     
\be
\E^{\mu_W} \; \mathcal Q (\bx(t), E') = \E^{\mu_W} \; F(\xi^{\bx(t_0)}_{t-t_0} ) + O(N^{-\tau}d_{ 1}) 
 = F(0)+ O(N^{-\tau}d_{ 2}).
\ee
 In the first step we used \eqref{chqu31} and in the second we used \eqref{TI} and \eqref{xnhao2}. 
This  implies that,   up to a negligible error, 
 the left side of the last equation is independent of   the specific initial
Wigner ensemble $\mu_W$, in particular, it is the same as 
for the Gaussian ensemble, i.e. $\mu_G$. Since  in the Gaussian case, we have $\bx (t) \overset{d}{\sim }  \bx^G$,
this proves
\eqref{bus0-1} and completes the proof of Lemma \ref{lem: flatO}. \qed

\subsection{Constantness of $F$: Proof of Lemma~\ref{lem: FF}.  } \label{sec:48}
 
Notice that $F$ is defined exclusively by the Gaussian ensemble, so the proof of  Lemma~\ref{lem: FF}
will be a Gaussian calculation where additional tools are available.\\

\noindent {\it Step 1: Apriori bounds on $F$}. 
For convenience, we define 
\be
  F_h(a):= F(a-h)-F(a).
\ee
By definition, for the $\al$-th derivative of $F(a)$, we have 
\be
  F^{(\al)} (a):=\E \sum_{ i,j \in  I (\delta_Q) }Q^{(\al)} \Big(N(y_{i}(t)-E)+a-\xi^{\by(t_0)}_{t-t_0}, \quad 
  N(y_{j}(t)-y_{i}(t)) \Big), \quad F^{(\al)}_h(a):= F^{(\al)}(a-h)-F^{(\al)}(a).
\ee
It follows from  \eqref{chqu31} that for $\tau<\min( \delta_Q/5, \tau_0)$
\be\label{asia}
 \left|\E^{\mu_W} \mathcal Q^{(\al)} \big(\bx(t), E-r N^{-1} \big)
 -
  \E^{\mu_W} \mathcal Q^{(\al)} \big(\bx(t), E+hN^{-1} -rN^{-1}  \big) \,-\,\E^{\mu_W} 
 F^{(\al)}_h(\xi^{\bx(t_0)}_{t-t_0}+r)\right| \,\le \,   6 d_{\al+1} N^{-\delta_Q/2} 
\ee
uniformly holds  for $\al\ge 0$ and       $|h|$, $|r|\le N^{4\tau}$.

  The following lemma provides an a priori bound on the derivatives of $F$. 
  
\begin{lemma}\label{lem: F}With $F$ defined in \eqref{defFA}, for any  positive  $\xi$, and   $\tau <\min( \delta_Q/5, \tau_0) $
 \be\label{mzz}
 \|F^{(\al )} \|_{  \infty } \le    d_{\al} N^{ \xi} 
\ee
holds uniformly for   $\al\ge 0$. Furthermore,  uniformly for   $\al\ge 0$ we have 
 \be\label{mzz-2}
{\bf 1}( |a|\ge N^{2\delta_Q})|F^{(\al)}(a)|\le  |a|^{-2}  d_{\al} N^{ \xi}.
\ee
\end{lemma}

\begin{proof}
For any fixed $a>0$,  we define a 
 subset of the probability space  $\Omega_a:=\Omega_{\xi, \delta_Q, a}$.
 If  $|a|\le  3N^2$ then $\Omega_a$  is
 the event such  that 
$$
\max_{i\in  I (\delta_Q) } \left|\gamma_i-y_{i}(t)\right| 
\le N^{-1+\xi}, \quad  \quad \quad |\xi^{\by(t_0)}_{t-t_0}|\le N^{\delta_Q}, \quad\hbox{ and}\quad \max_i |y_{i}(t)|\le   N^2 
$$
hold. If $|a|\ge  3N^2$, then let the event $\Omega_a$ be the set on which
$$
\max_i |y_{i}(t)|\le  \frac{|a|}{3N}
$$
holds.  Note that in the second case, the upper bound of $y_i$'s implies  $|\xi^{\by(t_0)}_{t-t_0}|\le |a|/3$. {Since $\sum_i y_i^2(t) = \Tr H(t)^2 = \sum_{ij} |h_{ij}|^2\sim N^{-1}\chi^2_{N^2}$, where $\chi^2_{N^2}$ is chi-square distribution with $N^2$ degrees of freedom, one can easily check that $\chi^2_{N^2}$ 
is smaller than $2N^2$ with a very high probability, and the probability density decay faster than polynomials.  Together with rigidity,  and \eqref{xnhao}, it implies that  }\be\label{Pom} \P (\Omega ^c_a)\le \min (N^{-10},   N^{-5}a^{-2}).
 \ee
 By the definition of $\Omega_a$ and \eqref{teyz}, for any  $\al\ge 0 $, $\xi>0 $ and $|a|\le 3N^2$,  we have 
 \begin{multline}\label{gongp}
 \Big|\E {\bf 1}(\Omega_a)  \sum_{ i,j \in  I (\delta_Q) }Q^{(\al)} \Big(N(y_{i}(t)-E)+a-\xi^{\by(t_0)}_{t-t_0}, 
 N(y_{j}(t)-y_{i}(t)) \Big)
\Big|\\
\le \;   \max_{b, b'\in \R} {\sup_{\omega\in \Omega_a}}\sum_{i\in  I (\delta_Q)}   \Big| Q^{(\al)}\left(N(y_{i}(t)-E)+b, \; b'\right) \Big| N^{2\xi}\le \;   d_{\al} N^{4\xi}.
\end{multline}
One can easily obtain the same bound for $|a|\ge 3N^2$, since in that case $|N(y_{i}(t)-E)+a-\xi^{\by(t_0)}_{t-t_0}|\ge |a|/10$   on the event $\Omega_a$. 

   On the other hand, the contribution from  $\Omega ^c$ 
 to $F^{(\al)}(a)$ is negligible   thanks to \eqref{Pom}. 
  Hence together with \eqref{gongp},   we obtain  \eqref{mzz}. Similarly, with \eqref{teyz} and
 $ |\xi^{\by(t_0)}_{t-t_0}| \le N^{\delta_Q}$, we have     \eqref{mzz-2}. 
\end{proof}

\noindent  {\it Step 2. Estimating $F$ with a Gaussian convolution.} 
In order to show that $F_h(a)$ is negligible, we first prove that its convolution
with a Gaussian kernel is small (and in Step 3 below we remove this convolution).
This is formulated in Lemma~\ref{mzaiq} below.
 We cannot prove this result  directly, but we can show
that $\E F_h (X)$ is small, where $X$ is a random variable close to a Gaussian.
The key is to choose the random variable $X$ appropriately: it will be
the mesoscopic statistics $\xi_{t-t_0}^{\bx}$ defined in \eqref{defxi} but applied to
the case where  $\bx$ is distributed by GOE. On one hand, 
by going back to the homogenization result, we show that 
$\E F_h (\xi_{t-t_0}^{\bx})$ is small, this will be formulated in \eqref{nanar} below.
 On the other hand, by using  the
Gaussian fluctuation of mesoscopic eigenvalue statistics, we 
show that $\xi_{t-t_0}^{\bx}$ is close to a Gaussian random variable,
this will follow from the combination of \eqref{xizz} and  \eqref{lizzae} below.
Now we explain these two ingredients in detail.

 The homogenization results
in the form \eqref{chqu31} and \eqref{asia}  hold
any Wigner ensemble $\bx (0)$. In particular,
they also hold 
for the case $\mu_W=\mu_G$.
To avoid confusion with the other Gaussian ensemble denoted by $\by$ earlier, 
 when taking $\mu_W$ to be $\mu_G$ we denote the eigenvalues by $\bz$ instead of $\bx$
in this argument. 
  Since for any $t>0$, the probability measure of $\bz(t)$ is also $\mu_G$, then for any $|E|\le 2-\kappa/2$ 
(for brevity we write $\rho^{(N, G)}_2$ instead of   $\rho_2^{ (N, {\tiny \rm GOE}) }$ and similarly for 
the limiting correlation functions) 
 { 
 \begin{align}\nonumber
 \E^{\mu_G} \mathcal Q^{(\al)} \big( \bz (t), E  \big)
 =& 
  \int Q^{(\al)}\left(u_1 , u_2-u_1\right) \rho^{ (N, G )}_2(E +\frac{\bu}N)\rd \bu,  \quad   \bu=(u_1,u_2)
 \\\nonumber
 =&  \sum_{n\in \Z}\int_{|u_1|\le \frac12} Q^{(\al)}\left(n+u_1 , u_2-u_1\right) \rho^{ (N, G)}_2(E_n+\frac{\bu}N)\rd \bu , \quad E_n=E+n N^{-1} 
 \\\label{onei}
   =&  \sum_{|n|\le N^{1/2}}\int_{|u_1|\le \frac12, |u_2-u_1|\le L} Q^{(\al)}\left(n+u_1 , u_2-u_1\right) \rho^{ (N, G)}_2(E_n+\frac{\bu}N)\rd \bu 
   +O(d_\al N^{-1/2+2\xi})
 \end{align}
 for any fixed $\xi>0$. Here for the last line, we used \eqref{teyz}  and rigidity of eigenvalues. 
  It follows from  Lemma \ref{lem: Gaussian} (with choosing the compact set $\{ \bu : |u_1|\le \frac12, |u_2-u_1| \le L\}\times \{x: |x|\le 2-\kappa/3\}$)
 that  the last line of \eqref{onei} equals
 \be\label{tuiyiba0}  
 \sum_{|n|\le N^{1/2}}  \int_{|u_1|\le \frac12, |u_2-u_1| \le L} Q^{(\al)}\left(n+u_1, u_2-u_1\right)  
\varrho(E_n )^2\rho^{(G)}_2\big(\varrho(E_n )\bu \big) \rd \bu  +O(d_\al N^{-1/2+2\xi})
 \ee
 where $C$ depends on $\kappa$ and $L$. For $|h|$, $|r|\le N^{1/2}$, we define the notations
$$
E^*=E_0-r N^{-1}, \quad E^{**}=E_0+hN^{-1}-r N^{-1}, \quad E^*_n=E^*+n, \quad E^{**}_n=E^{**}+n.
$$
It is well known from the explicit formula that $\rho^{(G)}_2(\bv)$ is uniformly smooth on any compact support. Then
$$
\varrho(E^*_n )^2\rho^{(G)}_2(\varrho(E^*_n )\bu)-\varrho(E^{**}_n )^2\rho^{(G)}_2(\varrho(E^{**}_n )\bu)=O(d_\al N^{-1/2}).
$$
Together with \eqref{tuiyiba0} and \eqref{onei}   we obtain that
\be\label{tuiyiba} 
  \left|\E^{\mu_G} \mathcal Q^{(\al)} \big( \bz (t), E-r N^{-1} \big)
 -
  \E^{\mu_G} \mathcal Q^{(\al)} \big(\bz (t), E+hN^{-1} -rN^{-1}  \big)\right|\le Cd_\al N^{-1/2+\xi}\le  d_\al N^{-1/3}
\ee
 uniformly holds  for $\al\ge 0$ and       $|h|$, $|r|\le N^{4\tau}$.

 We remark that  one can also prove \eqref{tuiyiba} directly from \eqref{tuiyiba0}
without using the smoothness of $\rho^{(G )}_2(\bv)$ but using a  version of   \eqref{teyz} 
  for
$\pt_y Q^{(\alpha)}$. It requires $ \wh Q\in W^{3,\infty}$, so it can be implemented
by  increasing the regularity condition
from $W^{2,\infty}$ to $W^{3,\infty}$ from the beginning of the proof.
Therefore, with \eqref{asia}  applied to $\mu_G$ instead of $\mu_W$,
 we have
 \be\label{nanar}
\left|\E^{\mu_G}  F^{(\al)}_h(\xi^{\bz (t_0)}_{t-t_0}+r)\right|\le 7d_{\al+1} N^{-\delta_Q/2}
\ee  
 for  any   $0<\tau < \min (\delta_Q/5, \tau_0)$.

 The next ingredient is to show that $\xi^{\bz (t_0)}_{t-t_0}$ is close to a Gaussian random variable. 
  Recall $\xi^{\bz(t_0)}_{t-t_0}$ is defined as
$$
\xi^{\bz(t_0)}_{t-t_0}=\sum_{k}p_{t-t_0}(\gamma_{i_0}, \gamma_k) \left({z}_{k }(t_0)- 
 \gamma _{k}\right), \quad i_0:= \min \{ i:\; \gamma_{i}\ge E\}.
$$
 The kernel $p_s(x,y)$ is originally defined on $[-2,2]^2$; we 
now extend it linearly   to a larger set in the second variable
so that it remains a differentiable function. 
 For $|\gamma|\ge 2$, we simply define $p_s(\gamma_{i_0},\gamma)$ such that $\partial_\gamma p_s(\gamma_{i_0},\gamma)=\partial_\gamma  p_s(\gamma_{i_0},\pm 2)$. 
 We also define $P_s: \R\to \R$ as a function such that  
\begin{equation}\label{eqn:Pt}
  P_s (\gamma):=  \int_{\gamma_{i_0}}^\gamma p_s(\gamma_{i_0},x)\rd x \for |\gamma|\le 3
\end{equation}
and $\supp P_s=[-4,4]$ and $|P_s''(\gamma)|\le C$ for $2\le |\gamma|\le 4$.   
{  With lemma {lem:diagonalization} on $p_s$, it is easy to check that for $i_0: i_0\sim N$, and $N-i_0\sim N$, 
  and $s\ll 1$, 
\be\label{zzaq3}
\|P_s\|_\infty\le C, \quad P\,'_s(\gamma)\le \frac{Cs}{s^2+(\gamma_{i_0}-\gamma)^2}, \quad 
P\,''_s(\gamma) \le \frac{Cs |\gamma_{i_0}-\gamma|}{s^4+(\gamma_{i_0}-\gamma)^4}, \quad 
\gamma\in \R.
\ee
 Then with \eqref{zzaq3}, rigidity of eigenvalues
 $\bz(t_0)$  and mean value theorem, for any $\xi>0$, we have  
\be\label{xizz}
\P\left(\left|\xi^{\bz(t_0)}_s - \zeta^{\bz(t_0)}_s \right|\ge N^{-1+\xi}\right)\le N^{-10}, 
\quad \zeta^{\bz(t_0)}_s:=\sum_j  \Big[ P_s \big( {z}_{j }(t_0)\big)-  P_s ( \gamma_{j })\Big], \qquad \forall s\ll 1.
\ee
Combining  \eqref{xizz}, \eqref{mzz} and \eqref{nanar},  then we obtain that  for any fixed $\tau < \min (\delta_Q/5, \tau_0)$

\be\label{nanar2}
\left|\E^{\mu_G}  F^{(\al)}_h( \zeta^{\bz(t_0)}_{t-t_0} +r)\right|
\le 8d_{\al+1} N^{- \delta_Q /2} 
\ee
uniformly holds  for $\al\ge 0$ and       $|h|$, $|r|\le N^{4\tau}$.

The characteristic  function of linear  statistics 
of $z_j(t_0)$ in the form 
\be\label{charf}
  \sum_j  P_{t-t_0} \big( {z}_{j }(t_0)\big)- \int_{-2}^2 P_{t-t_0} ( s)\varrho(s)\rd s
\ee
will be analyzed in  Section~\ref{sec:gauss}  in details.
 The main result (Theorem \ref{thm:GaussFluct}) states
that this linear statistics is asymptotically Gaussian with parameters (expectation and variance) expressed as certain functionals of $P_{t-t_0}$. These functionals
are somewhat complicated and will be defined later right above \eqref{eqn:condZ}. 
{With \eqref{zzaq3}}, a simple calculation gives that their values on $P_{t-t_0}$ are given by
$$
\sigma^2(  P_{t-t_0})= \tau \log N +\oo(\log N), \quad \delta(  P_{t-t_0})=O(1), \quad \eps(  P_{t-t_0})=O(N^{2\tau}), \quad \quad t-t_0\approx N^{-\tau}.
$$
With these values, Theorem~\ref{thm:GaussFluct} states that
\be\label{applthm}
\E ^{\mu_G}  \exp\left(i \lambda  \zeta^{\bz(t_0)}_{t-t_0} \right) =e^{-\frac{\la^2}{2}\sigma(P_{t-t_0})^2+\ii\la \left( \delta(P_{t-t_0}) + \wt \delta (P_{t-t_0})\right)}+\OO\left(N^{-1/100} \right),
\ee
for $|\lambda| \le (2\tau)^{-1/2}$,  where $\wt \delta(P_s)$ is defined as
$$
\wt \delta (P_s):= \sum_j   P_s ( \gamma_{j })-\int_{-2}^2 \varrho(u)   P_s(u)\rd u
$$
 to account for the difference between \eqref{charf} and the
definition of $\zeta_{t-t_0}^{\bz(t_0)}$ in \eqref{xizz}. 
By a  Riemann sum approximation, one can easily obtain  $\wt\delta(P_{t-t_0})=O(1)$.   Theorem~\ref{thm:GaussFluct} concerns only the small $\lambda$
regime; but Lemma 5.6 complements it 
  in the regime $(2\tau)^{-1/2}\le |\lambda|\le N^{1/10}$ 
with a crude estimate of order $N^{-1/100}$. Note that in this
regime and for small $\tau$ the first term in the r.h.s. of 
\eqref{applthm} is smaller than $N^{-1/100}$, so \eqref{applthm}
holds throughout the regime $|\lambda|\le N^{1/10}$.

We now define   $\zeta$ as a new Gaussian random variable
 with expectation $\delta(P_{t-t_0})+\wt\delta (P_{t-t_0})$ and variance $\sigma(P_{t-t_0})$:
\be\label{defzeta}
\zeta\sim\mathcal{N}\Big(\delta(P_{t-t_0})+\wt\delta (P_{t-t_0}) , \; \sigma(P_{t-t_0})  \Big).
\ee
Using \eqref{applthm}, the distribution of $\zeta$ 
is close to that of  $\zeta^{\bz(t_0)}_{t-t_0}$ in the following way; 
\be\label{lizzae}
\E ^{\mu_G}  \exp\left(i \lambda\, \zeta^{\bz(t_0)}_{t-t_0} \right) =\E  \exp\left(i \lambda\,  \zeta \right) +\OO\left(N^{-1/100} \right)
\ee
for $|\lambda |\le N^{1/10}$.
 Finally,  we use \eqref{lizzae} to replace $\zeta_{t-t_0}^{\bz(t_0)}$
with $\zeta$ in the bound \eqref{nanar2}. This gives the following main result
of Step 2.

 \begin{lemma}\label{mzaiq} Define $\zeta$ as in \eqref{defzeta} with $t=N^{-\tau}$ and $\tau  < \min (\delta_Q/5, \tau_0)$. Then   
\be\label{loren}
\left|\E F^{(\al)}_h( \zeta  +r) \right| \le Cd_{\al+2} N^{- \delta_Q /2} 
\ee
uniformly holds  for $\al\ge 0$ and       $|h|$, $|r|\le N^{4\tau}$,  where $C$ is independent of $\al$, $h$ and $r$. 

 \end{lemma}  
\begin{proof} We define 
 $$
 F_{h,r, \al}  (a):=F_h^{(\al)}(r-a).
 $$
 It follows from \eqref{lizzae} that 
 \be\label{biesq}
 \left|\E F^{(\al)}_h( \zeta  +r)-\E F^{(\al)}_h( \zeta^{\bz(t_0)}_{t-t_0}  +r)\right|\le 
 C N^{-1/100}\int_{|p|\le N^{1/10} }\left|\wh {\left(F_{h,r, \al }\right)} (p)\right|\rd p
  +\int_{|p|\ge N^{1/10} }\left|\wh {\left(F_{h,r, \al }\right)}(p)\right|\rd p. 
 \ee
 For  the last term, using Lemma \ref{lem: F}, we have  
 $$
\left|\wh {\left(F_{h,r, \al }\right)}(p)\right|= 
\frac{1}{|p|^2}  \left|\int     \left(F_{h,r, \al }\right)''(a)e^{-  i p a }\rd a\right|
\le C d_{\al+2} N^{3\delta_Q}|p|^{-2}.
 $$
Similarly, we have $\|\wh{ F_{h,r, \al} }\|_\infty\le C d_{\al+2} N^{3\delta_Q}$. Then 
\be\label{pyi}
\|\wh {\left(F_{h,r,\al}\right)} \|_{1}\le C d_{\al+2} N^{3\delta_Q}, \qquad \int_{|p|\ge N^{1/10} }\left|\wh {\left(F_{h,r, \al }\right)}(p)\right|\rd p \le C d_{\al+2} N^{-1/10+3\delta_Q}.
 \ee
 Together with \eqref{biesq} and \eqref{nanar2}, we obtain \eqref{loren} and complete the proof of Lemma~\ref{mzaiq}. 
\end{proof}

\noindent{\it Step 3: Removal of the Gaussian convolution.  }
 The expectation w.r.t. the Gaussian variable $\zeta$
in \eqref{loren} can be viewed as evolving the standard heat equation
on the function $F_h$ and its derivatives up to times given
by the variance $\sigma=\sigma(P_{t-t_0})$. We will thus show that 
 from the estimates on the heat evolution on $F_h$ given in \eqref{loren} 
we have effective estimates on the function $F_h$. 
This is similar to backward uniqueness of the heat equation
 for analytic functions, supplemented by precise bounds. 
 This step is the reason why we need to consider
test functions with compact Fourier support in Section~\ref{sec:pftech}. 

 Recall $F^{(\al )}$ are uniformly  bounded in \eqref{mzz}. Together with \eqref{teyz},  we can define:  $$
 U(r, t):=\sum_{\al =0}^{\infty} \frac{s^{\al}}{\al!}F_h^{(2\al)}(r),   \quad  \quad {  r, s \in \R}.
 $$
With \eqref{mzz}  and the estimate $d_\al\le C_m^\al$ from \eqref{teyz},
this power series is convergent and termwise differentiable
 in both variables arbitrary many times. 
  It is easy to check that
 $$
 \partial_s U(r,s)= \partial^2_{r} U(r,s).
 $$
  Thus $U(r,s)$ is the solution of the heat equation with
an initial condition $U(r,0)=F_h(r)$ that is analytic
in a strip around the real axis. Therefore the usual
semigroup property  extends to negative times as well
and  for any  $\mu, \sigma>0$ we have
 $$
F_h(r)=U(r, 0)=\int  U\left(r- r' , -\sigma^2\right)  \frac{1}{\sigma\sqrt{2\pi}}e^{-\frac{(r')^2}{\sigma^2}}\rd r' 
=\int  U\left(r- r' +\mu, -\sigma^2\right)  \frac{1}{\sigma\sqrt{2\pi}}e^{-\frac{(r'-\mu)^2}{\sigma^2}}\rd r' .
$$
Recall $ \zeta$ defined in \eqref{defzeta}.  Choosing $\mu=\delta(P_{t-t_0})+\wt\delta (P_{t-t_0})$ and $\sigma=\sigma(P_{t-t_0})$, we obtain 
$$
F_h(r)=\sum_{\al=1}^{\infty} \frac{\Big(-\sigma^2(P_{t-t_0})\Big)^{\al}}{\al!} \E F^{(2\al)}_h\big( \zeta  +\delta(P_{t-t_0})+\wt\delta(P_{t-t_0})\big), \quad  r\in \R.
 $$
Using \eqref{loren} and \eqref{teyz}, for  $\tau   < \min (\delta_Q/5, \tau_0)$,   we obtain that 
\be\label{baiyan}
 |F_h(r)|\le  C d_{2} N^{m^2 \tau -\delta_Q/2}
 \ee
uniformly holds  for         $|h|$, $|r|\le N^{4\tau}$. Let $\delta_0:=\min(\delta_Q/3, \tau_0)$. Inserting \eqref{baiyan} into \eqref{asia} with $r=\al=0$,  we obtain  \eqref{TI}  and complete the proof of Lemma \ref{lem: FF}.
  
\section[Mesoscopic fluctuations for Gaussian ensembles]
{Mesoscopic fluctuations for Gaussian ensembles}
\label{sec:gauss}

This section follows Johansson's method \cite{Joh1998} to prove Gaussian fluctuations of linear statistics at any mesoscopic scale $N^{-1+\e}$. An important ingredient is the optimal rigidity
of the eigenvalues obtained in \cite{BouErdYau2011,BouErdYau2012,BouErdYau2013}, allowing the choice of any $\e>0$. 
Moreover, while limiting Gaussian behaviour of linear statistics is obtained in  \cite{Joh1998} by characterizing the Laplace transform, 
in this section we choose to work with the Fourier transform, for the sake of better estimates on the speed of convergence. This implies technical complications: the partition function may vanish.

Consider the probability measure
\begin{equation}\label{eqn:GOE}
\rd\mu(\by):=\frac{1}{Z}\prod_{1\leq k<\ell\leq N}|y_k-y_\ell|^\beta e^{-\beta\frac{N}{4}\sum_{k=1}^N y_k^2}\rd\by
\end{equation}
on the simplex $y_1<\dots<y_N$ .
For a given function  $f:\RR\to\RR$ we consider the general linear statistics
$$
S_N(f):=\sum_{k=1}^N f(y_k)-N\int f(s)\varrho(s)\rd s,
$$
and we are interested in the Fourier transform
$$
Z(\la):=Z_N(\la)=\E_{\mu}(e^{\ii \la S_N(f)}).
$$
We will need the following complex measure, modification of the GOE: assuming $Z(\lambda)\neq 0$, we define
$$
\rd\mu^{\lambda}(\by):=
\frac{e^{\ii\lambda S_N(f)}}{Z(\lambda)}\rd\mu(\by).
$$
The following lemma about the total variation of $\mu^{\la}$ is elementary.

\begin{lemma}\label{lem:badrig}
If $Z(\la)\neq 0$, for any measurable $A$ we have 
$
|\mu^{\la}|(A)\leq \frac{\mu(A)}{|Z(\la)|}.
$
\end{lemma}

We will use the following rigidity estimate, proved for a wide class of $\beta$-ensembles including the quadratic beta ensemble
in \cite{BouErdYau2013}. We use the notation $\widehat k=\min(k,N+1-k)$.

\begin{lemma}\label{lem:rig}
For any $\xi>0$ there exists $c>0$ such that for any $N\geq 1$ and $k\in\llbracket 1,N\rrbracket$ we have
$$
\mu\left(|y_k-\gamma_k|>N^{-\frac{2}{3}+\xi}(\widehat k)^{-\frac{1}{3}}\right)\leq e^{-N^c}.
$$
\end{lemma}

As an easy consequence of Lemmas \ref{lem:badrig} and \ref{lem:rig}, rigidity estimates for $\mu$ yield rigidity estimates for $\mu^\la$, at the expense of a factor $Z(\la)^{-1}$. It also gives estimates on the $1$-point function and variances for the measure $\mu^\la$.  We recall the definition
of the correlation functions from \eqref{def:corr}, in particular
the $1$-point function satisfies 
$$
N\int h(s) \varrho_1^{(N,\la)}(s)\rd s=\E_{\mu^\la}\left(\sum_k h(y_k)\right)=\int \sum_k h(y_k)\rd\mu^\la(\by),
$$
for any continuous bounded test-function $h$. We also define
the complex variance  by
$
\var^{\mu^\la}\left(X\right)=\E^{\mu^\la}\left(X^2\right)-\left(\E^{\mu^\la} X\right)^2
$. 
We introduce the notation
for the Stieltjes transform of the empirical measure, and its expectation w.r.t. $\mu^\la$, by
$$
s_N(z):=\frac{1}{N}\sum_k\frac{1}{z-y_k},\qquad
 m_{N,\la}(z):=\E^{\mu^\la}(s_N(z)).
$$
We will
also use the following notation for the  Stieltjes transform of the semicircle distribution :
$$
m(z):=\int\frac{\varrho(s)}{z-s} \rd s=\frac{z-\sqrt{z^2-4}}{2},
$$
where the square root is chosen so that $m$ is holomorphic on $[-2,2]^{\rm c}$ and $m(z)\to 0$ as $|z|\to\infty$.

\begin{lemma}\label{lem:rigandcons}
Assume that $Z(\la)\neq 0$. For any $\xi>0$ there exists $c>0$ such that for any $N\geq 1$ and $k\in\llbracket 1,N\rrbracket$ we have
\begin{equation}\label{eqn:poorRig}
\left|\mu^\la\right|\left(|y_k-\gamma_k|>N^{-\frac{2}{3}+\xi}(\widehat k)^{-\frac{1}{3}}\right)\leq \frac{e^{-N^c}}{|Z(\la)|}.
\end{equation}
As a consequence, the following estimates hold: for fixed $\xi>0$, 
for any $0<|\eta|<1$ (remember $z=E+\ii\eta$),
$N\geq 1$,  and $f\in \mathscr{C}^2(\R)$  we have
\begin{align}
\int_{\mathbb{R}}\frac{f'(s)}{z-s}\varrho_1^{(N,\la)}(s)\rd s\label{eqn:1}
-\int_{\mathbb{R}}\frac{f'(s)}{z-s}\varrho(s)\rd s
&=\frac{N^{-1+\xi}}{|Z(\la)|}\OO\left(\int \frac{|f''(s)|}{|z-s|}\rd s+\int \frac{|f'(s)|}{|z-s|^2}\rd s+\frac{e^{-N^c}}{\eta^2}(\|f'\|_\infty+\|f''\|_\infty)\right),\\
\label{eqn:etaone}
m_{N,\la}(z)-m(z)&=\OO\left(\frac{N^{-1+\xi}}{|\eta Z(\la)|}\right),\\
m_{N,\la}'(z)-m'(z)&=\OO\left(\frac{N^{-1+\xi}}{\eta^2 |Z(\la)|}\right),\label{eqn:2}\\
\var^{\mu^\la}\left(\frac{1}{N}\sum_k\frac{1}{z-y_k}\right)&=\OO\left(\frac{N^{-2+2\xi}}{\eta^2 |Z(\la)|^2}
\right).\label{eqn:3}
\end{align}
\end{lemma}

\begin{proof}
The rigidity estimate (\ref{eqn:poorRig}) is immediate from Lemmas \ref{lem:badrig} and \ref{lem:rig}. For the proof of (\ref{eqn:1}), we first write the left hand side of (\ref{eqn:1}) as
\begin{equation}\label{eqn:inter1}
\frac{1}{N}\sum_{k=1}^N\E^{\mu^\la}\left(\int_{\gamma_k}^{y_k}\partial_s\left(\frac{f'(s)}{z-s}\right)\rd s\right)+\sum_{k=1}^N\int_{\gamma_{k-1}}^{\gamma_{k}}\rd s\varrho(s)\int_{s}^{\gamma_k}\partial_u\left(\frac{f'(u)}{z-u}\right)\rd u.
\end{equation}
Let $\gamma(u):=\max\{\gamma_k:\gamma_k\leq u\}$. The second sum above is easily bounded by
$$
\int_{-2}^2\left|\partial_u\frac{f'(u)}{z-u}\right|\rd u \int_{\gamma(u)}^u\varrho(s)\rd s
=
\OO\left(\frac{1}{N}\int \left(\frac{|f'(s)|}{|z-s|^2}+\frac{|f''(s)|}{|z-s|}\right)\rd s\right).
$$
To bound the first term in (\ref{eqn:inter1}), we first
denote $A=\{\forall k\in\llbracket 1,N\rrbracket,\ |y_k-\gamma_k|<N^{-\frac{2}{3}+\xi}(\widehat k)^{-\frac{1}{3}}\}$. Thanks to (\ref{eqn:poorRig}),
$$
\frac{1}{N}\sum_{k=1}^N\E^{\mu^\la}\left(\1_{A^c}\int_{\gamma_k}^{y_k}\partial_s\left(\frac{f'(s)}{z-s}\right)\rd s\right)
=
\OO\left(\frac{e^{-N^c}}{|Z(\la)|\eta^2}(\|f'\|_\infty+\|f''\|_\infty)\right).
$$
In the event $A$, we have
$$
\frac{1}{N}\sum_{k=1}^N\E^{\mu^\la}\left(\1_{A}\int_{\gamma_k}^{y_k}\partial_s\left(\frac{f'(s)}{z-s}\right)\rd s\right)
=
\OO\left(
\frac{1}{|Z(\la)|N}\int \left|\partial_s\frac{f'(s)}{z-s}\right|\sum_{k=1}^N\1_{s\in I_k}\rd s
\right)
$$
where $I_k=\{s:|s-\gamma_k|\leq N^{-\frac{2}{3}+\xi}(\widehat k)^{-\frac{1}{3}}\}$. This concludes the proof of (\ref{eqn:1}) by noting that for any fixed $s$ we have $|\{k:s\in I_k\}|\leq N^\xi$.

The bounds (\ref{eqn:etaone}), (\ref{eqn:2}) and (\ref{eqn:3}) can be proved the same way, by discussing the cases $A$ and $A^{\rm c}$. For example,
the considered variance can be written
$$
\E^{\mu^\la}\left(\left(\frac{1}{N}\sum_k\frac{1}{z-y_k}-\frac{1}{z-\gamma_k}\right)^2\right)-\left(\E^{\mu^\la}\frac{1}{N}\sum_k\left(\frac{1}{z-y_k}-\frac{1}{z-\gamma_k}\right)\right)^2.
$$
The first term can be bounded as previously and yields an error of size $\frac{N^{-2+2\xi}}{|Z(\la)|\eta^2}$ (in $A$). The second one 
yields the higher order error $\frac{N^{-2+2\xi}}{|Z(\la)|^2\eta^2}$, concluding the proof.
\end{proof}

For the following theorem,  we need the notations
\begin{align*}
&\kappa(s):=\max\{N^{-2/3},\min(|s-2|,|s+2|)\},\\
&\rd\nu(s):=\frac{1}{2}\left(\delta_{s-2}+\delta_{s+2}-\frac{1}{2\pi}\frac{\rd s}{\sqrt{4-s^2}}\right).
\end{align*}

\begin{theorem}\label{thm:GaussFluct}
Let $f$ be a ($N$-dependent) real function of class $\mathscr{C}^2$ such that, for any $N$, we have $\|f\|_\infty<C$, $\|f'\|_\infty,\|f''\|_\infty\leq N^C$, $ \int|f'|<C$. Let
\begin{align*}
&\sigma(f)^2:=\frac{1}{2\pi^2\beta}\iint_{(-2,2)^2}\left(\frac{f(x)-f(y)}{x-y}\right)^2\frac{4-x y}{\sqrt{4-x^2}\sqrt{4-y^2}}\rd x\rd y,\\
&\delta(f):=\left(\frac{2}{\beta}-1\right)\int f(s)\rd\nu(s),\\
&\e(f):=\left(1+\int |f''(s)|\kappa(s)^{-\frac{1}{2}}\rd s\right)^2.
\end{align*}
For any fixed $\xi>0$, uniformly in the set
\begin{equation}\label{eqn:condZ}
\Big\{\la\; : \;
\la^2\sigma(f)^2<\Big|\log|N^{-1+3\xi}\e(f)|\Big|
\Big\}\cap\{|\la|<N^{\xi}\}
\end{equation}
we have
$$
Z(\la)=\E^\mu\left(e^{\ii\la S_N(f)}\right)
=
e^{-\frac{\la^2}{2}\sigma(f)^2+\ii\la \delta(f)}+\OO\left(N^{-1+3\xi}\e(f)\right).
$$
\end{theorem}

\begin{proof}
The main tools for the proof of this theorem are the loop equation (\ref{eqn:loop}) and the Helffer-Sj{\"o}strand
formula to go from the Stieltjes transform to any test function. To derive proper asymptotics in the loop equation, an important input is the optimal rigidity and its consequences, Lemma \ref{lem:rigandcons}.

We begin with $
\frac{\rd}{\rd\la}\log Z(\la)=\E^{\mu^{\la}}\left(\ii S_N(f)\right)$, and therefore want to estimate expectation of general linear statistics for the measure $\mu^\la$. We begin with the expectation of the Stieltjes transform.\\

\noindent{\it First step: analysis of the loop equation.}
The loop equation  is a well-known algebraic identity for the
expectation of the empirical measure. In our case it takes the
following form: 
\begin{multline}\label{eqn:loop}
(m_{N,\la}(z)-m(z))^2-\sqrt{z^2-4}\, (m_{N,\la}(z)-m(z))
+\ii\frac{\la}{\beta N}\int_{\mathbb{R}}\frac{f'(s)}{z-s}\varrho_1^{(N,\la)}(s)\rd s\\
-\frac{1}{N}\left(\frac{2}{\beta}-1\right)m_{N,\la}'(z)-\var^{\mu^{(\la)}}\left(s_N(z)\right)
=0.
\end{multline}
Note that, when compared to the loop equation initiated in \cite{Joh1998} (written in a form closer  to (\ref{eqn:loop}) in \cite{BouErdYau2013} Section 6.2), we only consider the special case of quadratic external potential, hence extra simplifications occur.
From the estimates from Lemma \ref{lem:rigandcons} and our assumptions for the theorem, the loop equation (\ref{eqn:loop}) implies
that uniformly in $\eta>N^{-1+\xi}$ we have
\begin{align}
&X_N(z)^2-b(z)X_N(z)+c_{N}(z)=\OO(\om_N(z)),\label{eqn:quad}\\
&X_N(z)=m_{N,\la}(z)-m(z),\notag\\
&b(z)=\sqrt{z^2-4},\notag\\
&c_{N}(z)=\ii\frac{\la}{\beta N}\int_{\mathbb{R}}\frac{f'(s)}{z-s}\varrho(s)\rd s-\frac{1}{N}\left(\frac{2}{\beta}-1\right)m'(z),\notag\\
&\om_N(z)=\frac{N^{-2+{2\xi}}}{|Z(\la)|^2}\left(\frac{1}{\eta}\int|f''|+\frac{1}{\eta^2}\right).\notag
\end{align}
Let 
$$
\Omega_N:=\{z=E+\ii\eta: N^\xi\min(N^{-2/3},N^{-1}\kappa(E)^{-1/2}) \leq |\eta|\leq 3, |E|\leq 4\}.
$$
A simple analysis exercise shows that
$$
  \sup_{s\in[-2,2]}\frac{\varrho(s)}{|z-s|}\leq C\eta^{-1} \max(\eta,\kappa(E))^{1/2}.
$$
Using this estimate together with
 $\int|f'|<C$ and $|\la|\leq N^\xi$,
   we have $|b(z)|^2>c N^\xi |c_N(z)|$
  for any $z\in\Omega_N$. 
We consider two cases to identify the relevant root of (\ref{eqn:quad}).
\begin{enumerate}[(i)]
\item If $|b(z)|^2>N^\xi \om_N(z)$, by monotonicity we also have $|b(z')|^2>N^\xi \om_N(z')$ for any $z'=E+\ii\eta',|\eta'|>|\eta|$.

Moreover, using (\ref{eqn:etaone}),
together with $|Z(\la)|>N^{-1+\frac{3\xi}{2}}$ (obtained from our assumption $|b(z)|^2>N^\xi \om_N(z)$),
we have $m_{N,\la}(z)-m(z)\to 0$ when $|\eta|$ is of order 1. 

All together, by continuity we proved that in this case, for any $z\in\Omega_N$,
$$
m_{N,\la}(z)-m(z)=c_N(z)/b(z)+\OO(\om_N(z)/b(z)).
$$
\item Assume $|b(z)|^2\leq N^\xi \om_N(z)$ (in particular $|\om_N(z)|>|c_N(z)|$). Any solution of (\ref{eqn:quad}) satisfies
$$
|X_N(z)|\leq C\max(|b(z)|, \sqrt{|c_N(z)|+|\om_N(z)|})\leq C N^{\xi/2}\sqrt{\om_N(z)}\leq C N^\xi |\om_N(z)/b(z)|.
$$
\end{enumerate}
In all cases, we therefore proved that uniformly in $\Omega_N$ we have
\begin{equation}\label{eqn:goodRoot}
m_{N,\la}(z)-m(z)=\frac{c_N(z)}{b(z)}+\OO\left(N^\xi\frac{\om_N(z)}{|b(z)|}\right).
\end{equation}
~\vspace{0.3cm}
\noindent{\it Second step, integration.}
Let $\widetilde f$ coincide with $f$ on $(-3,3)$, such that $\widetilde f(x)=0$ for $|x|>4$ and $\|(f-\widetilde f)^{(\ell)}\|_\infty<C$
for $\ell=0,1,2$. From (\ref{eqn:poorRig}) we have
$$
\E^{\mu^\la}\left(\ii S_N(f)\right)
=
\E^{\mu^\la}\left(\ii S_N(\widetilde f)\right)+\OO\left(\frac{e^{-{N^c}}}{|Z(\la)|}\right).
$$
Let $\chi:\RR\to\RR^+$ be a smooth symmetric function such that $\chi(x)=1$ for $x\in(-1,1)$ and $\chi(x)=0$ for $|x|>2$. By the Helffer-Sj{\"o}strand formula
\cite{HelSjo1989}
we have
$$
\E^{\mu^\la}\left(\ii S_N(\widetilde f)\right)
=
-\frac{\ii}{2\pi}\iint_{\RR^2} \left(\ii y \widetilde f''(x)\chi(y)+\ii(\widetilde f(x)+\ii y \widetilde f'(x))\chi'(y)\right)N(m_{N,\la}(x+\ii y)-m(x+\ii y))\rd x\rd y.
$$
We now bound some error terms.
\begin{enumerate}[(i)]
\item Using the estimate (\ref{eqn:etaone}), we have (note that $\widetilde f(x)\chi'(y)$ and $\widetilde f'(x))\chi'(y)$ both vanish for $z=x+\ii y\not\in\Omega_N$, and $f''(x)\chi(y)=f''(x)$ when $z\not\in\Omega_N$)
\begin{multline*}
\iint_{\Omega_N^{\rm c}} \left(\ii y \widetilde f''(x)\chi(y)+\ii(\widetilde f(x)+\ii y \widetilde f'(x))\chi'(y)\right)N(m_{N,\la}(x+\ii y)-m(x+\ii y))\rd x\rd y\\
=
\iint_{\Omega_N^{\rm c}} \ii y \widetilde f''(x)\chi(y)N(m_{N,\la}(x+\ii y)-m(x+\ii y))\rd x\rd y
=\OO\left(\frac{N^{2\xi}}{N|Z(\la)|}\int |\widetilde f''(x)|\kappa(x)^{-1/2}\rd x\right).
\end{multline*}
\item A simple calculation yields (note that $|b(x+\ii y)|>c$ when $\chi'(y)\neq 0$)
$$
\iint_{\Omega_N} (|\widetilde f(x)|+|y \widetilde f'(x)|)|\chi'(y)|N\frac{|\om_N(z)|}{|b(z)|}\rd x\rd y
=\OO\left(\frac{N^{-1+2\xi}}{|Z(\la)|^2}\left(1+\int|f''|\right)\right).
$$
Moreover,
$$
\iint_{\Omega_N} |\widetilde f''(x)|y\chi(y)N\frac{\om_N(z)}{|b(z)|}\rd x\rd y
=\OO\left(\frac{N^{-1+2\xi}}{|Z(\la)|^2}\left(1+\int|f''|\right)\int|\widetilde f''(x)|\kappa(x)^{-1/2}\rd x\right).
$$
\item Finally, thanks to the easy estimate $|c_N(z)|\leq C(|\la|/(Ny)+1/(Ny))$, we have
$$
\iint_{\Omega_N^{\rm c}} \left(\ii y \widetilde f''(x)\chi(y)+\ii(\widetilde f(x)+\ii y \widetilde f'(x))\chi'(y)\right)N\frac{c_N(z)}{b(z)}\rd x\rd y=
N^{-1+\xi}\OO\left(\int |\widetilde f''(x)|\kappa(x)^{-1/2}\rd x\right).
$$
\end{enumerate}
Let 
\begin{align*}
&\widetilde\sigma(f)^2:=
-\frac{1}{2\pi\beta}\iint_{\RR^2} \left(\ii y \widetilde f''(x)\chi(y)+\ii(\widetilde f(x)+\ii y \widetilde f'(x))\chi'(y)\right)b(z)^{-1}\left(\int\frac{f'(s)}{z-s}\varrho(s)\rd s\right)\rd x\rd y,\\
&\widetilde\delta(f):=\left(\frac{2}{\beta}-1\right)\frac{1}{2\pi}
\iint_{\RR^2} \left(\ii y \widetilde f''(x)\chi(y)+\ii(\widetilde f(x)+\ii y \widetilde f'(x))\chi'(y)\right)b(z)^{-1}m'(z)\rd x\rd y.
\end{align*}
Using (i), (ii) and (iii) all together, we proved that
$$
\E^{\mu^\la}\left(\ii S_N(f)\right)
=
-\la\widetilde\sigma(f)^2+\ii\widetilde\delta(f)
+
\OO\left(\frac{N^{-1+2\xi}}{|Z(\la)|^2}\e(f)\right).
$$
Let $g(\la)=e^{\la^2\widetilde\sigma(f)^2-2\ii\widetilde\delta(f)\la}Z(\la)^2$. The above equation implies
$
g'(\la)=e^{\la^2\widetilde\sigma(f)^2-2\ii\widetilde\delta(f)\la}\OO(N^{-1+2\xi}\e(f)),
$
so
$g(\la)=1+e^{\la^2\widetilde\sigma(f)^2}\OO(N^{-1+3\xi}\e(f))$. On our set (\ref{eqn:condZ}), by continuity in $\la$ 
this implies
\begin{equation}\label{eqn:CLT}
Z(\la)=
e^{-\frac{\la^2}{2}\widetilde\sigma(f)^2+\ii\la \widetilde\delta(f)}+\OO\left(N^{-1+3\xi}\e(f)\right).
\end{equation}
We now want to prove $\widetilde\sigma(f)^2=\sigma(f)^2$ and $\widetilde\delta(f)=\delta(f)$.
If $f$ is fixed independent of $N$, (\ref{eqn:CLT}) proves that $S_N(f)$ converges to a Gaussian random variable with variance $\widetilde\sigma(f)^2$ and shift $\widetilde\delta(f)$.
Thanks to \cite[Theorem 2.4]{Joh1998} we can identify this shift: we know that $S_N(f)$ converges to a Gaussian with shift $\delta(f)$. Thanks to \cite[Theorem 2]{LytPas2009}, we can identify the variance: for $\beta=1$, $S_N(f)$ converges to a Gaussian with variance $\sigma(f)^2$. This implies the identity $\widetilde\sigma(f)^2=\sigma(f)^2$ and $\widetilde\delta(f)=\delta(f)$ for any $f$, and concludes the proof.
\end{proof}

\begin{remark}
In the previous theorem, the error term $\e(f)$ is quadratic in $\int|f''|$, which is sufficient for our purpose, as we will apply it for $f$ fluctuating at
the mesoscopic scale $N^{-\tau}$ for some small $\tau$. 

If one is interested in the mesoscopic statistics at scale $N^{-1+\e}$ for some small $\e$ and the support of $f$ is of order 1, the above reasoning fails. On the other hand, if $f$ is supported in the bulk, with support size $(\int |f''|)^{-1}$, then by taking in the previous reasoning $\chi$ a cutoff function on scale $(\int |f''|)^{-1}$ one obtains an error linear in $\int|f''|$ instead of quadratic, which is sufficient to prove Gaussianity of $S_N(f)$ at this very small mesoscopic scale.
\end{remark}

Assuming $\e(f)$ has size $N^\theta$ for some $\theta\in(0,1)$, Theorem \ref{thm:GaussFluct} gives a very accurate control of $Z(\la)$ in the regime $|\la|\leq c(\theta)(\log N)^{1/2}/\sigma(f)$. The purpose of the following lemma is to get a rough polynomial bound on $Z$ in the regime $|\la|>c(\theta)(\log N)^{1/2}/\sigma(f)$.

\begin{lemma}\label{n100}
Let $f$ be a ($N$-dependent) real function of class $\mathscr{C}^2$ such that, for any $N$, we have $\|f\|_\infty<C$, $\|f'\|_\infty,\|f''\|_\infty\leq N^C$, $ \int|f'|<C$. 

Assume that $\e(f)\leq N^{1/2}$, $c\leq (\log N)^{1/2}/\sigma(f)$, and $\sigma(f)>c$. Then for any $|\la|\in[(\log N)^{1/2}/\sigma(f),N^{1/10}]$ we have
\be\label{result n100}
|Z(\la)|\leq C\ N^{-1/100}.
\ee
\end{lemma}

\begin{proof}
Without loss of generality, we can assume $\la>0$. Note that for $\la=(\log N)^{1/2}/\sigma(f)$, from Theorem \ref{thm:GaussFluct} we have $|Z(\la)|\leq N^{-1/100}$, so we only need to prove the following statement: 
if $\la\in[(\log N)^{1/2}/\sigma(f),N^{1/10}]$ and $|Z(\la)|> N^{-1/100}$ then
\begin{equation}\label{eqn:negZ}
\frac{\rd}{\rd\la}\Re\log Z(\la)<0.
\end{equation}
To prove the above statement, we begin as in the proof of Theorem \ref{thm:GaussFluct} with
$\frac{\rd}{\rd\la}\Re\log Z(\la)=\Re\E^{\mu^{\la}}(\ii S_N(f))$. If we repeat exactly the proof of 
Theorem \ref{thm:GaussFluct} except that we substitute $\Omega_N$ with
$$
\Omega_{N,\la}:=
\{z=E+\ii\eta: \la\min(N^{-2/3}, N^{-1}\kappa(E)^{-1/2}) \leq |\eta|\leq 3, |E|\leq 4\}.
$$
Then the following variant of $(\ref{eqn:goodRoot})$ holds uniformly in $\Omega_{N,\la}$:
$$
m_{N,\la}(z)-m(z)=\frac{c_N(z)}{b(z)}+\OO\left(\la \frac{\om_N(z)}{|b(z)|}\right).
$$
This allows us to reproduce all error estimates (i), (ii) and (iii) in the integration step, always replacing $\Omega_N$ with $\Omega_{N,\la}$. We end up with
$$
\Re\E^{\mu^{\la}}(\ii S_N(f))
=
-\la \widetilde\sigma(f)+\OO\left(\frac{N^{-1}\la^2}{|Z(\la)|^2}\e(f)\right).
$$
From our strong assumptions $\la\leq N^{1/10}$, $|Z(\la)|>N^{-1/100}$ and $\e(f)<N^{1/2}$ the above term is positive for large enough $N$. This conclues the proof of (\ref{eqn:negZ}) and the lemma.
\end{proof}

\begin{appendices}

\section{H\"older regularity}
\label{sec:Holder}

We now explain the proof of  Lemma \ref{holder}, i.e., the H\"older regularity for \eqref{ve1}. It  directly follows from 
Theorem 10.3 of \cite{ErdYau2012singlegap} after checking the 
conditions. We recall that the setup of  \cite{ErdYau2012singlegap}
was the  discrete equation  
\be\label{veq7}
   \pt_s \bv(s) = -\cA(s)\bv(s), \qquad \cA(s) = \cB(s) + \cW(s), 
\ee
in a finite $I \subset  \llbracket 1, N \rrbracket$ of size $|I|=K$
and in a time interval $s\in [0,\sigma]$.
Here $\cW(t)$ is an diagonal operator given by $(\cW(t) \bv)_i = 
W_i(t) v_i$. We will apply this result for $I = \llbracket 1, N \rrbracket$,
i.e. $K=N$. The time interval is $[0,\sigma]:= [0, t]$. 
The key assumption  on the coefficients $B_{jk}$  is the following strong  regularity condition 
(Definition 9.7 in  \cite{ErdYau2012singlegap}).  We remind the reader that, 
compared with the scalings of this paper,  the time 
in  \cite{ErdYau2012singlegap} is rescaled by a factor $N$ while   
the coefficient $B_{jk}$ is  rescaled by  a factor $1/N$. 
The microscopic coordinates used in  \cite{ErdYau2012singlegap} are chosen so   that 
the eigenvalue spacing is of order one and the  time to equilibrium is of order $N$. 
In this paper, all scalings are dictated by the original scalings of the DBM,
so the 
 following setup uses the scaling convention in this paper. 
 
 Theorem 10.3 in \cite{ErdYau2012singlegap} had 
two conditions, called $({\bf C1})_\rho$ and $({\bf C2})_\xi$.
The first condition is the following concept of strong regularity: 

\begin{definition}\label{def:strreg}
The equation 
\be\label{veq8}
   \pt_t \bv(t) = -\cB(t)\bv(t) 
\ee
 is called {\bf regular} at the
space-time point $(z, \si)$ with exponent $\rho$, if
\be\label{Kassnew3}
   \sup_{0\le s\le \sigma}\sup_{1 \le M \le N} 
\frac{1}{ 1/N+ |s-\si| }\Big|  \int_s^\si \frac{1}{M}
 \sum_{i\in I\, : \, |i-z| \le M}\sum_{j\in I\, : \, |j-z| \le M} 
 B_{ij}(s')
 \rd s' \Big| \le N^{1+\rho}.
\ee
Furthermore, the equation is called {\bf strongly regular} at the
space-time point $(z, \si)$ with exponent $\rho$ if it is 
regular at all points $\{ z\}\times \{ \sigma \Xi + \si\}$, where 
$$
  \Xi = \big\{ -   2^{-m}(1+2^{-k}) \; : \; 0\le k,m \le C\log N  \big\}.
$$
\end{definition}

Strong regularity \eqref{Kassnew3}  at $(z,t)$ with exponent
$2\rho$  follows from  \eqref{EB} on a set $\cR_{z,t}$  of probability
at least $1-CN^{-\rho}(\log N)^4$. Without the double supremum in \eqref{Kassnew3}
this would clearly follow from  the Markov inequality
and the cardinality $|\Xi|\le C(\log N)^2$. However, the suprema over all $s$ and $M$ can be 
replaced by suprema over  a dyadic choice of  $s= 2^{-a}\sigma$, $M= 2^b$, with intgers $a, b\le C\log N$, 
explaining the additional logarithmic factors. 

The other condition, $({\bf C2})_\xi$, expresses various a priori
bounds on $B_{ij}$ that follow from  \eqref{B1} and \eqref{B2}.
 More precisely, we need
for any $0\le s \le t$ 
\begin{align}\label{g3}
     B_{ij}(s) \ge & \frac{N^{-\xi}}{N|i-j|^2}, \; \mbox{for any $i,j$
     with $\wh i, \wh j \ge cN$},\\
     \frac{{\bf 1}(\min \{ \wh i, \wh j\} \ge cN)}{ C N|i-j|^2} \le  B_{ij} (s) \le & \frac{C}{N |i-j|^2}, 
  \;  \mbox{for any}  \; |i-j|\ge C'N^{\xi}
\end{align}
with some constants $C, C', c$, 
where recall that $\wh i = \min\{ i, N+1-i\}$ denotes the distance from the edge.  
Finally,  in \cite{ErdYau2012singlegap} the diagonal operator is assumed to satisfy 
\be\label{g4}
  W_i(s)\le   \frac{N^{\xi}}{N \wh i},  \quad \mbox{if}\quad \wh i\ge N^\xi,
\ee  
 but in our application the diagonal operator is not present.
Having verified these conditions  (with a  
possible modified value of $\rho$),  Lemma \ref{holder}
directly follows from Theorem 10.3 of \cite{ErdYau2012singlegap}. \qed

\section{Level repulsion estimate}
\label{sec:Wegner}

The following level repulsion estimate is adapted from \cite{ErdSchYau2010}. 
The main differences are:
\begin{enumerate}[(i)] 
\item it is given for symmetric matrices instead of Hermitian; 
\item we consider the generalized Wigner class instead of Wigner;
\item the matrix entries are smooth on scale $N^{-\tau/2}$ instead of $1$. 
\end{enumerate}
We closely
follow the method from \cite{ErdSchYau2010}, where the Hermitian case was given in details.
Since the adjustment of the proof to the symmetric case requires technical changes,
for the convenience of the reader, we will give the main steps of the proof and
explain the modifications.

\begin{proposition}\label{prop:Wegner}
Let $H_N$ be a symmetric generalized Wigner matrix satisfying (\ref{eqn:forHansonWright}), and
$G_N$ a $N\times N$ GOE matrix. For any $t>0$ we denote $\mu_1(t)\leq\dots\leq \mu_N(t)$ the eigenvalues of
$\sqrt{1-t}H_N+\sqrt{t}G_N$.
Define the set
\begin{equation}\label{eqn:rigid}
\mathcal{G}_{\xi}=\left\{|\mu_i-\gamma_i|\leq N^{-2/3+\xi}(\hat i)^{-1/3}\ \mbox{for all }\ i\in\llbracket 1,N\rrbracket\right\}.
\end{equation}

For any fixed $\kappa$ there exists $C_1>0$ such that for any $k\geq 1$, $\tau,\xi>0$, there exists $C_2>0$ such that for any $N\in \mathbb{N}$, $E\in(-2+\kappa,2-\kappa)$, $t\in[N^{-\tau},1]$ and $\e>0$ we have
$$
\P\left(\left\{\left|\{\mu_i(t)\in [E,E+\e/N]\}\right|\geq k\right\}\cap\mathcal{G}_{\xi}\right)\leq C_2 N^{2 k\xi+ C_1 k^2 \tau} \e^{\frac{k(k+1)}{2}}.
$$
\end{proposition}

\noindent
Compared to  \cite[Theorem 3.5]{ErdSchYau2010}, the above Wegner estimate bound has extra $N^\xi$ factors, because our proof does not use subgaussian decay 
of the matrix entries (we only assume condition (\ref{eqn:forHansonWright}) instead). The same comment applies to the following corollary.

\begin{corollary}\label{cor:Wegner}
Assume the same conditions as Proposition \ref{prop:Wegner}.

For any fixed $\alpha>0$ there exists $C_1>0$ such that for any $\tau,\xi>0$, there exists $C_2>0$ sucht that  for any $N\in \mathbb{N}$, $i\in\llbracket\alpha N,(1-\alpha)N\rrbracket$, $t\in[N^{-\tau},1]$ and $\e>0$ we have
$$
\P\left(\left\{|\mu_{i+1}(t)-\mu_i(t)|\leq\frac{\e}{N}\right\} \cap \mathcal{G}_{\xi}\right)\leq C_2 N^{ 2k\xi+ C_1\nc \tau} \e^2.
$$
\end{corollary}

\begin{proof}
For any $j\in\mathbb{Z}$, define $E_j=\gamma_i+j\frac{\e}{N}$. We then have the events inclusion
$$
\left\{|\mu_{i+1}(t)-\mu_i(t)|\leq\frac{\e}{N}\right\}\cap\mathcal{G}_{\xi}
\subset
\bigcup_{|j|\leq \frac{N^\xi}{\e}+1}\left(\left\{\left|\left\{\mu_\ell(t)\in\left[E_j,E_j+\frac{2\e}{N}\right]\right\}\right|\geq2\right\}
\cap
\mathcal{G}_{\xi}
\right).
$$
The union bound together with Proposition \ref{prop:Wegner} applied with $k=2$ allow to conclude.
\end{proof}

The above Corollary \ref{cor:Wegner} actually holds for eigenvalues up to the edge (with the exponents $\e N^{-1}$ and $N^{-1+\delta}$ being replaced by $\e N^{-2/3}(\widehat i)^{-1/3}$, $N^{-2/3+\delta}(\widehat i)^{-1/3}$, respectively). The proof requiring just formal changes, we will only present the bulk case here, for notational simplicity.

To prepare the proof of Proposition \ref{prop:Wegner}, we need the following lemmas.
In particular, Proposition \ref{prop:Wegner} will require a regularity assumption of type (\ref{eqn:dens}) for the matrix entries. Note that this condition was weakened in \cite{MalSch2011} to $\int (f'/f)^4f<\infty$ (where $f$ is the density of real and imaginary parts of the matrix entries), but we will not need this improvement.

\begin{lemma}\label{lem:smoothness}
Let $H=(h_{ij})_{1\leq i,j\leq N}$ be a symmetric generalized Wigner matrix satisfying $(\ref{eqn:forHansonWright})$ and $\tau>0$.

We denote $f=e^{-g}$ the probability density of $\sqrt{1-t}\sqrt{N}h_{ij}+\sqrt{t}\mathcal{N}$, where
$t\in[N^{-\tau},1]$ and $\mathcal{N}$ is a standard Gaussian independent from $H$. Then there exists $C>0$ such that for any $a\geq 1$ there exists $c_a>0$ such that uniformly in 
$N,i,j, s\in\RR$,  we have
\begin{equation}\label{eqn:dens}
|\widehat f(s)|\leq c_a\frac{N^{C a\tau}}{(1+s^2)^a},\ \ 
|\widehat{f g''}(s)|\leq c_a\frac{N^{C a\tau}}{(1+s^2)^a}.
\end{equation}
\end{lemma}

\begin{proof}
The first inequality is elementary:
$$
\left|\int e^{\ii s x} f(x)\rd x\right|=
|\E(e^{\ii s\sqrt{1-t}\sqrt{N}h_{ij}})|
|\E(e^{\ii s\sqrt{t}\mathcal{N}})|\leq e^{-s^2 \frac{t}{2}}\leq c_a\frac{t^{-a}}{(1+s^2)^a}.
$$
For the second one, we have $fg''=f'^2/f-f''$ and
$$
\left|\int e^{\ii s x} f''(x)\rd x\right|\leq s^{2}\left|\int e^{\ii s x} f(x)\rd x\right|\leq
s^2c_{a+1}\frac{t^{-{(a+1)}}}{(1+s^2)^{a+1}}\leq c_{a+1}\frac{t^{-2a}}{(1+s^2)^{a}},
$$
so we only need to bound (without loss of generality we can assume $a$ is an integer)
\begin{equation}\label{eqn:cutbound}
\left|\int e^{\ii s x} \frac{f'(x)^2}{f(x)}\rd x\right|\leq
\frac{2^a}{(1+s^2)^{a}}\left|\int e^{\ii s x} \left(\frac{\rd}{\rd x}\right)^{2a}\frac{f'(x)^2}{f(x)}\rd x\right|\1_{|s|\geq 1}+
\left|\int \frac{f'(x)^2}{f(x)}\rd x\right|\1_{|s|\leq 1}.
\end{equation}
Let $\nu$ be the distribution of $\sqrt{1-t}\sqrt{N}h_{ij}$. Then for any $K>0$ we have
\begin{multline*}
|f'(x)|=\left|\frac{1}{\sqrt{2\pi t}}\int\nu(\rd u)\frac{x-u}{t}e^{-\frac{(x-u)^2}{2t}}\right|
\leq
C\ K t^{-1} |f(x)|+
C\ t^{-3/2} e^{-\frac{K^2}{2t}}\int\nu(\rd u)|x-u|\\
\leq C x N^{C\tau}|f(x)|+C N^{C\tau} e^{-\frac{x^2}{2t}},
\end{multline*}
where we chose $K=x$ and used (\ref{eqn:forHansonWright}) so that $\nu$ has finite first moment.
Moreover, we obviously have $|f(x)|> c N^{-C\tau}e^{-\frac{x^2}{2t}}$, so we can easily bound the second term on the right hand side of (\ref{eqn:cutbound}):
$$\left|\int \frac{f'(x)^2}{f(x)}\rd x\right|\leq 
C\left|\int |f'(x)| (1+|x|)N^{C\tau}\rd x\right|\leq C N^{C\tau}.
$$
For the first term of the right hand side in (\ref{eqn:cutbound}), and expansion of the $2a$-th derivative of this ratio and the same cut argument by $K=x$ yields
$$
\int \left|\left(\frac{\rd}{\rd x}\right)^{2a}\frac{f'(x)^2}{f(x)}\right|\rd x\leq C N^{C\tau a},
$$
which concludes the proof.
\end{proof}

\begin{lemma}\label{lem:moments}
Fix $p\in\mathbb{N}^*$ and $N\geq p+3$. Let $\bu_1,\dots,\bu_{N-1}$be an orthonormal basis in $\mathbb{R}^N$, and set $\xi_\alpha=|\bbb\cdot \bu_\alpha|^2$, where the components of $\bbb$ are independent centered real random variables with density $f=e^{-g}$ satisfying $\var b_i\sim 1$, the decay (\ref{eqn:forHansonWright}) and the density smoothness assumption (\ref{eqn:dens}), uniformly in $N$ and $i\in\llbracket 1,N-1\rrbracket$.

Let $\alpha_1,\dots,\alpha_p,\beta_1,\beta_2,\beta_3$ be distinct indices in $\llbracket 1,N-1\rrbracket$. Let $c_j>0$, $j\in\llbracket 1,p\rrbracket$,  $d_\alpha\in\mathbb{R}$ for all $1\leq \alpha\leq N-1$, $d_{\beta_1}, d_{\beta_2}, d_{\beta_3}>0$.

\begin{enumerate}[(i)]
\item For any $r\in\left(1,\frac{p}{2}+1\right)$, there exists a constant $C_{r,p}<\infty$ such that
\begin{equation}\label{eqn:first}
\E_{\bbb}\left(\left(\sum_{j=1}^p c_j\xi_{\alpha_j}\right)^2+\left(E-\sum_{\alpha=1}^{N-1}d_\alpha \xi_\alpha\right)^2\right)^{-\frac{r}{2}}\leq
C_{r,p}\frac{N^{2C(r-1)\tau}}{\left(\prod_{j=1}^p c_j^{1/2}\right)^{\frac{2(r-1)}{p}}\min(d_{\beta_1},d_{\beta_2},d_{\beta_3})}.
\end{equation}
\item For any $r\in\left(\frac{p+1}{2},\frac{p}{2}+1\right)$, there exists a constant $C_{r,p}<\infty$ such that
\begin{equation}\label{eqn:second}
\E_{\bbb}\left(\left(\sum_{j=1}^p c_j\xi_{\alpha_j}\right)^2+\left(E-\sum_{\alpha=1}^{N-1}d_\alpha \xi_\alpha\right)^2\right)^{-\frac{r}{2}}\leq
C_{r,p}\frac{N^{C p\tau}}{\left(\prod_{j=1}^{p-1} c_j^{1/2}\right)c_p^{r-\frac{p+1}{2}}\min(d_{\beta_1},d_{\beta_2},d_{\beta_3})}.
\end{equation}
\item For any $r\in\left(1,\frac{p}{2}\right)$, there exists a constant $C_{r,p}<\infty$ such that
\begin{equation}\label{eqn:third}
\E_{\bbb}\left(\sum_{j=1}^p c_j\xi_{\alpha_j}\right)^{-r}\leq
C_{r,p}\frac{N^{C(r-1)\tau}}{(\min_jc_j^{1/2})^r}.
\end{equation}
\end{enumerate}
\end{lemma}

\begin{proof}
We closely follow the method of Lemma 8.2 in \cite{ErdSchYau2010}. The main differences in the estimates (due to considering real instead of complex random variables) are the exponents $c_j^{1/2}$ in the upper bounds (instead of $c_j$ when the $b_j$'s take complex values),
and the fact that we need to consider $3$  variables $d_{\beta_1}\xi_{\beta_1},d_{\beta_2}\xi_{\beta_2},
d_{\beta_3}\xi_{\beta_3}$ for convergence purpose (instead of 2 when the $b_j$'s take complex values).  
Moreover, the extra error terms $N^{C(r-1)\tau}$ are of course  due to our smoothness scale.

Let $O$ be the orthogonal matrix with columns $\bu_1,\dots,\bu_{N-1}$, $\bx=O^*\bbb$,
$\rd\mu(\bx)=e^{-\Phi(\bx)}\prod_{\alpha=1}^{N-1}\rd x_\alpha$, with $\Phi(\bx)=\sum_{\ell=1}^{N-1}g((O\bx)_\ell)$, $F(t)=\int_{-\infty}^t\left((\sum_{j=1}^pc_j x_{\alpha_j}^2)^2+s^2\right)^{-r/2}$ and $D=x_{\beta_{1}}\partial_{x_{\beta_1}}+x_{\beta_{2}}\partial_{x_{\beta_2}}+x_{\beta_3}\partial_{x_{\beta_3}}$.
Then the analogue of \cite[equation (8.20)]{ErdSchYau2010} is
$$
I:=\E_{\bbb}\left(\left(\sum_{j=1}^p c_j\xi_{\alpha_j}\right)^2+\left(E-\sum_{\alpha=1}^{N-1}d_\alpha \xi_\alpha\right)^2\right)^{-\frac{r}{2}}
=
\frac{1}{2}\int \rd\mu(\bx)\frac{F(E-\sum_{\alpha=1}^{N-1}d_{\alpha}x_\alpha^2)}{d_{\beta_1}x_{\beta_1}^2
+d_{\beta_2}x_{\beta_2}^2+d_{\beta_3}x_{\beta_3}^2}(1-D\Phi(\bx)).
$$
We then can follow \cite[equations (8.21), (8.23)]{ErdSchYau2010} and bound $I\leq (A_1+A_2+A_3+|B_1|+|B_2|+|B_3|)/2$ where
\begin{align*}
A_1&:=\int\rd\mu(\bx)\frac{\1_{\sum_{j=1}^pc_jx_{\al_j}^2\leq \kappa}}{(\sum_{j=1}^p c_jx_{\al_j}^2)^{r-1}},\\
A_2&:=\frac{1}{\kappa^{r-1}}\int\rd\mu(\bx)\frac{1}{x_{\beta_1}^2+x_{\beta_2}^2+x_{\beta_3}^2},\\
A_3&:=\int\rd\mu(\bx)\frac{\1_{\sum_{j=1}^pc_jx_{\al_j}^2\leq \kappa}\1_{x_{\beta_1}^2+x_{\beta_2}^2
+x_{\beta_3}^2\leq 1}}{(\sum_{j=1}^p c_jx_{\al_j}^2)^{r-1}(x_{\beta_1}^2+x_{\beta_2}^2+x_{\beta_3}^2)^2},\\
B_k&:=\int \rd\mu(\bx)\frac{1}{(\sum_{j=1}^pc_jx_{\al_j}^2)^{r-1}}(\partial_{x_{\beta_k}}\Phi)^2,\ k=1,2,3.
\end{align*}
To prove (i), we first bound $A_3$. For this, let 
$$
\widetilde f(\bx):=\frac{\1_{\sum_{k=1}^p c_j x_j^2\leq \kappa}\1_{x_{p+1}^2+x_{p+2}^2+x_{p+3}^2\leq 1}}{(\sum_{j=1}^p c_j x_j^2)^{r-1}(x_{p+1}^2+x_{p+2}^2+x_{p+3}^2)}.
$$
The integral of $\widetilde f$ over $x_{p+1},x_{p+2},x_{p+3}$ is finite (for this we need at least 3 such terms), and changing the other variables $c_j^{1/2} x_j\to x_j$ and using $r<\frac{p}{2}+1$ we have $\|f\|_1\leq C_{r,p}\kappa^{\frac{p}{2}+1-r}/\prod_{j=1}^p c_j^{1/2}$. The reasoning of \cite[equation (8.26)]{ErdSchYau2010}, with the first equation of (\ref{eqn:dens}) as an input, gives
$$
A_3\leq C_{r,p} \|\widetilde f\|_1 N^{C p\tau}\leq C_{r,p}\frac{\kappa^{\frac{p}{2}+1-r}}{\prod_{j=1}^p c_j^{1/2}}N^{C p\tau}.
$$
The terms $A_1$ can be controlled in the same way and $A_2\leq C \kappa^{r-1}$. The bound on 
$B_1, B_2, B_3$ amounts to the same estimate as $A_1, A_2, A_3$, thanks to the representation analogue to \cite[equation (8.28)]{ErdSchYau2010}, and it requires the second estimate in (\ref{eqn:dens}). 
We therefore obtained
$$
I\leq C_{r,p}\left(\frac{\kappa^{\frac{p}{2}+1-r}}{\prod_{j=1}^p c_j^{1/2}}N^{C p\tau}+\frac{1}{\kappa^{r-1}}\right).
$$
Optimization over $\kappa$ concludes the proof of (i).

For (ii), we bound $I\leq (A_4+A_5+A_6+A_7+|B_1|+|B_2|+|B_3|)/2$ where
\begin{align*}
A_4&:=\int\rd\mu(\bx)\frac{\1_{x_{\al_p}^2\leq 1}}{(\sum_{j=1}^{p-1} c_jx_{\al_j}^2+c_p x_{\al_p}^2)^{r-1}},\\
A_5&:=\int\rd\mu(\bx)\frac{1}{(\sum_{j=1}^{p-1} c_jx_{\al_j}^2+c_p)^{r-1}},\\
A_6&:=\int\rd\mu(\bx)\frac{\1_{x_{\beta_1}^2+x_{\beta_2}^2+x_{\beta_3}^2\leq 1}}{(\sum_{j=1}^{p-1} c_jx_{\al_j}^2+c_p)^{r-1}(x_{\beta_1}^2+x_{\beta_2}^2+x_{\beta_3}^2)},\\
A_7&:=\int\rd\mu(\bx)\frac{\1_{x_{\alpha_p}^2\leq 1}\1_{x_{\beta_1}^2+x_{\beta_2}^2+x_{\beta_3}^2\leq 1}}{(\sum_{j=1}^{p-1} c_jx_{\al_j}^2+c_p x_{\alpha_p}^2)^{r-1}(x_{\beta_1}^2+x_{\beta_2}^2+x_{\beta_3}^2)}.
\end{align*}
To bound the term $A_7$, we now introduce the function
$$
\widetilde f(x_1,\dots,x_{p+2}):=\frac{\1_{x_p^2\leq 1}\1_{x_{p+1}^2+x_{p+2}^2+x_{p+3}^2\leq 1}}{(\sum_{j=1}^{p-1}c_j x_j^2+c_p x_p^2)^{r-1}(x_{p+1}^2+x_{p+2}^2+x_{p+3}^2)}.
$$
Again, the integral of $\widetilde f$ over $x_{\beta_1},x_{\beta_2},x_{\beta_3}$ is finite.
By changing the variable $c_j^{1/2}x_j\to x_j$ for $j\in\llbracket 1,p-1\rrbracket$ we get
$$
\|f\|_1\leq \frac{1}{(\prod_{j=1}^{p-1} c_j^{1/2})c_p^{r-\frac{p}{2}-\frac{1}{2}}}\int_{|x_p|<1}\frac{1}{(\sum_{j=1}^{p-1}x_i^2+x_p^2)^{r-1}}\rd x_1\dots\rd x_p\leq C_{r,p}
\frac{1}{(\prod_{j=1}^{p-1} c_j^{1/2})c_p^{r-\frac{p}{2}-\frac{1}{2}}},
$$
where we used that the above integral is finite when $r\in(\frac{p+1}{2},\frac{p}{2}+1)$: after changing variables, it is also of order
$$
\int_{0<x<1,y>0}\frac{x^{-\frac{1}{2}}y^{\frac{p-1}{2}-1}}{(y+x)^{r-1}}\rd x\rd y
\leq 
\int_{0<x<1}x^{-\frac{1}{2}}\int_{a>x}a^{\frac{p-1}{2}-r}\rd a\rd x\leq \int_0^1 x^{\frac{p}{2}-r}\rd x<\infty.
$$
Proceeding as in \cite[equation (8.28)]{ErdSchYau2010} we conclude that 
$$A_7\leq C_{r,p} \frac{N^{C p \tau}}{(\prod_{j=1}^{p-1} c_j^{1/2})c_p^{r-\frac{p}{2}-\frac{1}{2}}}.$$
The terms $A_4, A_5, A_6$ can be bounded in the same way. Similarly, $B_1$, $B_2$ and $B_3$ can be bounded by the previous reasoning after using the analogue of representation \cite[equation (8.28)]{ErdSchYau2010} in the real context. This concludes the proof of (ii).
Finally, the proof of (iii) is elementary.
\end{proof}

\begin{proof}[Proof of Proposition \ref{prop:Wegner}]
We follow the method from \cite[Sections 8 and 9]{ErdSchYau2010}. The required preliminary results from \cite{ErdSchYau2010}
are listed below, as well as their substitute used in our proof, so that we can accomodate the weaker condition (\ref{eqn:forHansonWright}) instead of the subgaussian decay of the matrix entries.
\begin{enumerate}[(1)]
\item The localization result \cite[Theorem 3.1]{ErdSchYau2010} states in particular the following.
Writing $\mathcal{N}_{\eta}(E)=\mathcal{N}_{I}$ the number of eigenvalues in $I=[E-\eta/2,E+\eta/2]$, then for any $\delta>0$ one has
\begin{equation}\label{eqn:rig}
\Prob\left(\left|\frac{\mathcal{N}_{\eta^*}(E)}{N\eta^*}-\varrho(E)\right|\geq \delta\right)
\le Ce^{-c\delta\sqrt{N\eta^*}}.
\end{equation} 
Equation (\ref{eqn:rig}) states that in windows of scale $\eta$ 
the fluctuations of the number of eigenvalues is of order $\sqrt{N\eta}$.
Instead, we will use the rigidity result (remember the notation (\ref{eqn:rigid})), for any (small) $\xi>0$ and (large) $D>0$ we have
\begin{equation}\label{eqn:rigitityi}
\mathbb{P}\left(\mathcal{G}_{\xi}\right)\geq 1-N^{-D}.
\end{equation}
The above rigidity estimate was proved in \cite{ErdYauYin2012Rig}
assuming subexponential decay of the entries distribution, but this is easily weakened to
the finite moment assumption (\ref{eqn:forHansonWright}) (see remark 2.4 in \cite{ErdKnoYauYin2013}).
Compared to  (\ref{eqn:rig}), the above bound on fluctuations of eigenvalues is better for mesoscopic scales ($\eta\sim N^{-1+c}$ for small fixed $c$) but it becomes worse at the microscopic scale ($\eta\sim N^{-1}$), explaining the final extra $N^\xi$ factor in our statement of Proposition \ref{prop:Wegner}.

\item The tail distribution of the eigenvalue gap \cite[Theorem 3.3]{ErdSchYau2010}: denoting by $\mu_\alpha$ the largest eigenvalue greater than $E$, there are constants $C,c>0$ such that uniformly in $N,K\geq 0$ and E in the bulk of the spectrum we have
\begin{equation}\label{eqn:decaynext}
\Prob\left(\mu_{\alpha+1}-E\geq \frac{K}{N},\alpha\leq N-1\right)\leq C e^{-c\sqrt{K}}.
\end{equation}
Again, the above result assumes subgaussian decay of the entries, in this paper it will therefore be substituted by (\ref{eqn:rigitityi}).

\item The analogue of \cite[Theorem 3.4]{ErdSchYau2010} requires smoothness of the entries. We therefore  now assume the $\mu_i(t)$'s are as in Proposition \ref{prop:Wegner}, so that they satisfy the density bounds (\ref{eqn:dens}). The average density of states becomes, in our context: denoting $I=[E-\e/(2N),E+\e/(2N)]$, there exists $C>0$ such that uniformly in  $0\leq \e\leq 1$,
we  have
\begin{equation}\label{eqn:1st}
\Prob(\{\mathcal{N}_I\geq 1\}\cap\mathcal{G}_{\xi})\leq C N^{\xi+C\tau}\e.
\end{equation}
For the proof, we denote $(\lambda^{(j)}_\alpha)_\alpha$  for the eigenvalues of the minor  obtained from $H$ by removing the $j$-th row and column, $(u^{(j)}_\alpha)_\alpha$ the eigenvectors, and $\xi^{(j)}_\alpha=|\bbb^{(j)}\cdot u^{(j)}_\alpha|^2$ where $\bbb^{(j)}=\sqrt{N}(h_{j2},\dots,h_{j,N})$.

The proof of (\ref{eqn:1st}) is the same as \cite[Theorem 3.4]{ErdSchYau2010}, except that: (i) one needs to replace
the definition \cite[(8.3)]{ErdSchYau2010} by $\Delta=N(\lambda^{(1)}_{\gamma+3}-E)$ by $N(\lambda^{(1)}_{\gamma+4}-E)$, because the analogue (\ref{eqn:first}) of \cite[(8.12)]{ErdSchYau2010} requires three indexes $d_\beta$ in the real case instead of two for complex entries, for convergence reasons; (ii) the error term has a factor $N^{C\tau}$ due to the deteriorated smoothness (\ref{eqn:dens}) and its consequeces in (\ref{eqn:first}), (\ref{eqn:third}); (iii)
the rigidity input (\ref{eqn:rig}) and (\ref{eqn:decaynext}) used in \cite{ErdSchYau2010} are replaced by (\ref{eqn:rigitityi}), explaining the above extra 
$N^\xi$ factor in the Wegner estimate (\ref{eqn:1st}).
\end{enumerate}

Thanks to these preliminary results (1), (2), (3),  the analogue of \cite[Theorem 3.5]{ErdSchYau2010}, Proposition \ref{prop:Wegner}, can be proved as follows.
First, the inequality \cite[(9.2)]{ErdSchYau2010} still holds:
\begin{equation}\label{eqn:base}
\mathcal{N}_I\leq \frac{C\e}{N}\sum_{j=1}^N\left(\left(\eta+\frac{\eta}{N}\sum_{\al=1}^{N-1}\frac{\xi^{(j)}_\al}{(\la_{\al}^{(j)}-E)^2+\eta^2}\right)^2+\left(E-h_{jj}+\frac{1}{N}\sum_{\alpha=1}^{N-1}\frac{(\la_\al^{(j)}-E)\xi_\al^{(j)}}{(\la_\al^{(j)}-E)^2+\eta^2}\right)^2\right)^{-\frac{1}{2}}.
\end{equation}
We follow \cite{ErdSchYau2010} and denote
\begin{align}
&d_\al^{(j)}:=\frac{N(\la^{(j)}_\al-E)}{N^2(\la^{(j)}_\alpha-E)^2+\e^2},\ c^{(j)}_\alpha=\frac{\e}{N^2(\la^{(j)}_\al-E)^2+\e^2},\notag \\
&\mu_{\gamma(N)}:=\min\left\{\mu_\al:\mu_\al-E\geq \frac{\e}{N}\right\},\ \Delta_d^{(\mu)}=N(\mu_{\gamma(N)+d-1}-E).\label{defDelta}
\end{align}
In the following, $\Delta_d^{(\mu)}$ is always well defined because we will always consider 
$d=\OO(1)$ as $N\to\infty$: in the set $\mathcal{G}_{\xi}$ there are always many more than $d-1$ 
eigenvalues above 
$E+\e/N$.

Note that the proof of (\ref{eqn:1st}) actually gives a bit more, i.e. the analogue of \cite[Corollary 8.1]{ErdSchYau2010}, which is the first step in the following induction (\ref{eqn:induction}): for any $M,d\geq 1$, we have
\begin{equation}\label{eqn:1st2nd}
\E\left(\1_{\mathcal{N}_I\geq 1}(\Delta^{(\mu)}_d)^M\1_{\mathcal{G}_{\xi}}\nc\right)\leq C N^{ M \xi\nc+C\tau}\e.
\end{equation}

To bound $\Prob(\mathcal{N}_I\geq k,\mathcal{G}_{\xi})$, we introduce the more general quantity
$$
I_N^{(\mu)}(M,k,\ell):=\E(\1_{\mathcal{N}^{(\mu)}_I\geq k}(\Delta_\ell^{(\mu)})^M\1_{\mathcal{G}_{\xi}}\nc).
$$
We will prove that
\begin{equation}\label{eqn:induction}
I_N^{(\mu)}(M,k,\ell)\leq C_k N^{C k\tau} \e^k \max_{1\leq j\leq N}I^{(j)}_{N-1}(M+2,k-1,\ell+1).
\end{equation}
By induction over $k$, together with the initial condition (\ref{eqn:1st2nd}), this will conclude the proof, noting that 
$1+\sum_{j=2}^kj=\frac{k(k+1)}{2}$. To prove (\ref{eqn:induction}), thanks to (\ref{eqn:base}) for any $r\geq 1$ we have
\begin{align*}
I_N^{(\mu)}(M,k,\ell)&\leq C_{k,s} \e^r\max_{1\leq j\leq N}\E\frac{\1_{\mathcal{N}_I^{(j)}\geq k-1}(\Delta_\ell^{(\mu)})^M\1_{\mathcal{G}_{\xi}}\nc}{\left(\left(\sum_{\al=1}^{N-1}c_\al^{(j)} \xi^{(j)}_\al\right)^2+\left(E-h_{jj}+\sum_{\al=1}^{N-1}d^{(j)}_\al\xi^{(j)}_\al\right)^2\right)^\frac{r}{2}}\\
&\leq 
C_{k,s} \e^r\max_{1\leq j\leq N}\E\frac{\1_{\mathcal{N}_I^{(j)}\geq k-1}(\Delta_{\ell+1}^{(\la^{(j)})})^M\1_{\mathcal{G}_{\xi}}\nc}{\left(\left(\sum_{\al=1}^{N-1}c_\al^{(j)} \xi^{(j)}_\al\right)^2+\left(E-h_{jj}+\sum_{\al=1}^{N-1}d^{(j)}_\al\xi^{(j)}_\al\right)^2\right)^\frac{r}{2}},
\end{align*}
where we used Markov's inequality and convexity of $x\mapsto x^r$ in the first inequality, and interlacing in the second.
We used the definition (\ref{defDelta}), applied to the eigenvalues of the minor, $\la^{(j)}=(\la^{(j)}_1,\dots,\la^{(j)}_{N-1})$, instead of $\mu$. Quantities of type $\Delta_{\ell}^{(\la^{(j)})}$ are well defined
becuse $\ell=\OO(1)$ as $N\to\infty$ and 
in the set $\mathcal{G}_{\xi}$ there are always many more than $\ell$ 
eigenvalues above 
$E+\e/N$, by interlacing.\nc
We therefore have $I_N(M,k,\ell)\leq C_{k,s}(\max_{1\leq j\leq N}A_j+\max_{1\leq j\leq N}B_j)$ where
\begin{align*}
A_j&:=\e^r\E\frac{\1_{\mathcal{N}_I^{(j)}\geq k+2}(\Delta_{\ell+1}^{(\la^{(j)})})^M\1_{\mathcal{G}_{\xi}}\nc}{\left(\left(\sum_{\al=1}^{N-1}c_\al^{(j)} \xi^{(j)}_\al\right)^2+\left(E-h_{jj}+\sum_{\al=1}^{N-1}d^{(j)}_\al\xi^{(j)}_\al\right)^2\right)^\frac{r}{2}},\\
B_j&:=\e^r\E\frac{\1_{k-1\leq \mathcal{N}_I^{(j)}\leq k+1}(\Delta_{\ell+1}^{(\la^{(j)})})^M\1_{\mathcal{G}_{\xi}}\nc}{\left(\left(\sum_{\al=1}^{N-1}c_\al^{(j)} \xi^{(j)}_\al\right)^2+\left(E-h_{jj}+\sum_{\al=1}^{N-1}d^{(j)}_\al\xi^{(j)}_\al\right)^2\right)^\frac{r}{2}}.
\end{align*}
To bound $A_j$, denoting $\la^{(j)}_{\al_1},\dots,\la^{(j)}_{\al_{k+2}}$ the first $k+2$ eigenvalues in $I_\eta$, we have $c_{\al_i}\geq 1/(2\e)$, so  (\ref{eqn:third}) gives, by a reasoning identical to \cite[(9.8)]{ErdSchYau2010},
$
A_j\leq C_{k,s}\e^{2r}I^{(j)}_{N-1}(M,k-1,\ell+1),
$ provided $r\in(1,\frac{k+2}{2})$ (in particular one can get the exponent $\e^k$). 

To bound the main term $B_j$,  let $\al_1,\dots,\al_{k-1}$ be indices so that $\la^{(j)}_{\al_i}\in I_\eta$, $1\leq j\leq k-1$. As there are less than $k+2$ eigenvalues in $I_\eta$, we can assume that for $N$ large enough there are four eigenvalues at distance greater than $\e/(2N)$ from  $E$, on its right for example. Let 
$\la_{\al_k}=\min\{\la_\al:\la_\al>E+\frac{\e}{2N}\}$. We also denote $\la^{(j)}_{\beta_1}\leq\la^{(j)}_{\beta_2}\leq\la^{(j)}_{\beta_3}$ the eigenvalues immediately on the right of $\la^{(j)}_{\al_k}$, and $\Delta=\Delta^{(\la^{(j)})}_4=N(\la^{(j)}_{\beta_3}-E)$. Then, analogously to \cite[(9.10)]{ErdSchYau2010}, we have
\begin{multline*}
B_j\leq \e^r
C_{k,s}\E_{\la^{(j)},h_{jj}}\left(
\1_{\mathcal{N}_I^{(j)}\geq k-1}(\Delta_{\ell+1}^{(\la^{(j)})})^M\1_{\mathcal{G}_{\xi}\nc}\right.
\left.\E_{\bbb^{(j)}}\left(\left(\sum_{i=1}^{k-1}\frac{\xi^{(j)}_{\al_i}}{\e}+\frac{\e}{\Delta^2}\xi^{(j)}_{\al_k}\right)^2+\left(E-h_{jj}+\sum_{\al=1}^{N-1}d^{(j)}_\al\xi^{(j)}_\al\right)^2\right)^{-\frac{r}{2}}
\right)
\end{multline*}
We use (\ref{eqn:second}) with $p=k$, $c_j=\e^{-1}$, $1\leq j\leq p-1$, $c_p=\e\Delta^{-2}$,
$\min(d^{(j)}_{\beta_1},d^{(j)}_{\beta_2},d^{(j)}_{\beta_3})\geq 1/(2\Delta)$, $r\in(\frac{k+1}{2},\frac{k}{2}+1)$:
$$
B_j\leq C_{k,r}\e^{r}\frac{1}{(\prod_{i=1}^{k-1}\e^{-1/2})\e^{r-\frac{k+1}{2}}}N^{C k\tau} I_{N-1}^{(j)}(M+2,k-1,\ell+1)\leq \e^k N^{C k\tau} I_{N-1}^{(j)}(M+2,k-1,\ell+1).
$$
This concludes the proof.
\end{proof}

\end{appendices}

\end{document}